\documentclass[a4paper,11pt]{article}
\usepackage{amsmath,amsthm,amssymb,enumitem,xcolor}

\usepackage[hang]{footmisc}
\usepackage[nosort,nocompress,noadjust]{cite}

\usepackage[bookmarks=false,hyperfootnotes=false,colorlinks,
    linkcolor={red!60!black},
    citecolor={blue!50!black},
    urlcolor={blue!80!black}]{hyperref}

\renewcommand{\eqref}[1]{\hyperref[#1]{(\ref{#1})}}

\newlist{enumlist}{enumerate}{2}
\setlist[enumlist,1]{labelindent=0cm,label=\arabic*.,ref=\arabic*,labelwidth=2.5ex,labelsep=0.5ex,leftmargin=3ex,align=left,topsep=0.5ex,itemsep=1ex,parsep=1ex}
\setlist[enumlist,2]{labelindent=0cm,label=\theenumlisti.\arabic*.,ref=\arabic*,labelwidth=5ex,labelsep=0.5ex,leftmargin=5.5ex,align=left,topsep=0.5ex,itemsep=1ex,parsep=1ex}

\newlist{itemlist}{itemize}{2}
\setlist[itemlist,1]{labelindent=0cm,label=$\bullet$,labelwidth=2.5ex,labelsep=0.5ex,leftmargin=3ex,align=left,topsep=0.5ex,itemsep=1ex,parsep=1ex}
\setlist[itemlist,2]{labelindent=0cm,label=$\circ$,labelwidth=2.5ex,labelsep=0.5ex,leftmargin=3ex,align=left,topsep=0.5ex,itemsep=1ex,parsep=1ex}

\numberwithin{equation}{section}

{\theoremstyle{definition}\newtheorem{definition}{Definition}[section]
\newtheorem*{definition*}{Definition}
\newtheorem{letterdef}{Definition}

\newtheorem{remark}[definition]{Remark}

\newtheorem*{example*}{Example}
\newtheorem*{examples*}{Examples}}

\newtheorem{proposition}[definition]{Proposition}
\newtheorem{lemma}[definition]{Lemma}
\newtheorem{theorem}[definition]{Theorem}
\newtheorem{corollary}[definition]{Corollary}

\newtheorem{letterthm}[letterdef]{Theorem}
\newtheorem{lettercor}[letterdef]{Corollary}

{\theoremstyle{definition}}

\renewcommand{\Re}{\operatorname{Re}}
\renewcommand{\Im}{\operatorname{Im}}

\newcommand{\C}{\mathbb{C}}

\newcommand{\eps}{\varepsilon}
\newcommand{\al}{\alpha}
\newcommand{\be}{\beta}
\newcommand{\ot}{\otimes}

\newcommand{\Z}{\mathbb{Z}}
\newcommand{\vphi}{\varphi}
\newcommand{\cO}{\mathcal{O}}
\newcommand{\id}{\mathord{\text{\rm id}}}
\newcommand{\om}{\omega}
\newcommand{\N}{\mathbb{N}}
\newcommand{\ovt}{\mathbin{\overline{\otimes}}}

\newcommand{\Om}{\Omega}

\newcommand{\si}{\sigma}
\newcommand{\R}{\mathbb{R}}
\newcommand{\F}{\mathbb{F}}

\newcommand{\cZ}{\mathcal{Z}}
\newcommand{\Ad}{\operatorname{Ad}}
\newcommand{\cG}{\mathcal{G}}
\newcommand{\cK}{\mathcal{K}}

\newcommand{\cF}{\mathcal{F}}

\newcommand{\actson}{\curvearrowright}

\newcommand{\cB}{\mathcal{B}}
\newcommand{\cW}{\mathcal{W}}
\newcommand{\cU}{\mathcal{U}}
\newcommand{\Ker}{\operatorname{Ker}}

\newcommand{\cR}{\mathcal{R}}

\newcommand{\cV}{\mathcal{V}}

\newcommand{\cS}{\mathcal{S}}

\newcommand{\pr}{\operatorname{pr}}
\newcommand{\Stab}{\operatorname{Stab}}
\newcommand{\SL}{\operatorname{SL}}
\newcommand{\SO}{\operatorname{SO}}
\newcommand{\Q}{\mathbb{Q}}
\newcommand{\PSL}{\operatorname{PSL}}
\newcommand{\PSO}{\operatorname{PSO}}
\newcommand{\Gtil}{\widetilde{G}}
\newcommand{\Gammatil}{\widetilde{\Gamma}}
\newcommand{\omtil}{\widetilde{\omega}}
\newcommand{\bH}{\mathbb{H}}
\newcommand{\Omtil}{\widetilde{\Omega}}
\newcommand{\Ptil}{\widetilde{P}}
\newcommand{\Spin}{\operatorname{Spin}}

\newcommand{\GL}{\operatorname{GL}}
\newcommand{\cT}{\mathcal{T}}
\newcommand{\Ind}{\operatorname{Ind}}

\newcommand{\Ctil}{\widetilde{C}}

\newcommand{\resprod}{\sideset{}{'}\prod}

\begin{document}
\begin{center}
{\boldmath\LARGE\bf Superrigidity for dense subgroups of Lie groups\vspace{0.5ex}\\
and their actions on homogeneous spaces}

\bigskip

{\sc by Daniel Drimbe\footnote{\noindent KU~Leuven, Department of Mathematics, Leuven (Belgium).\\ E-mails: daniel.drimbe@kuleuven.be and stefaan.vaes@kuleuven.be.}\textsuperscript{,}\footnote{\noindent D.D.\ holds the postdoctoral fellowship fundamental research 12T5221N of the Research Foundation -- Flanders.} and Stefaan Vaes\textsuperscript{1,}\footnote{S.V.\ is supported by FWO research project G090420N of the Research Foundation Flanders and by long term structural funding~-- Methusalem grant of the Flemish Government.}}

\end{center}

\begin{abstract}\noindent
An essentially free group action $\Gamma \actson (X,\mu)$ is called W$^*$-superrigid if the crossed product von Neumann algebra $L^\infty(X) \rtimes \Gamma$ completely remembers the group $\Gamma$ and its action on $(X,\mu)$.
We prove W$^*$-superrigidity for a class of infinite measure preserving actions, in particular for natural dense subgroups of isometries of the hyperbolic plane. The main tool is a new cocycle superrigidity theorem for dense subgroups of Lie groups acting by translation. We also provide numerous countable type II$_1$ equivalence relations that cannot be implemented by an essentially free action of a group, both of geometric nature and through a wreath product construction.
\end{abstract}

\section{Introduction}

Popa's deformation/rigidity theory \cite{Pop06b} provided a wealth of rigidity theorems for probability measure preserving (pmp) group actions $\Gamma \actson (X,\mu)$, both from the von Neumann algebra and the orbit equivalence point of view. The most extreme form of rigidity, usually referred to as W$^*$-superrigidity, requires that both the group $\Gamma$ and its action $\Gamma \actson (X,\mu)$ can be entirely recovered from the ambient crossed product II$_1$ factor $L^\infty(X) \rtimes \Gamma$, see Definition \ref{def.superrigid}. A recent survey of such rigidity theorems in a pmp context can be found in \cite{Ioa18}.

Beyond the pmp setting, there are basically no known W$^*$-superrigidity theorems. As far as we know, \cite[Proposition D]{Vae13} provides the only known W$^*$-superrigid actions that are not pmp, but these examples are very much ad hoc. In this paper, we prove that several natural families of infinite measure preserving actions are W$^*$-superrigid, including many natural isometric actions on the hyperbolic plane. Note here that also \cite[Theorems C and D]{Iso19} provides W$^*$-superrigidity beyond the type II$_1$ setting, but is entirely noncommutative, dealing with group actions on type~III factors, rather than measure spaces.

To prove W$^*$-superrigidity of an essentially free ergodic action $\Gamma \actson (X,\mu)$, one needs to prove at the same time that $L^\infty(X) \rtimes \Gamma$ has a unique (group measure space) Cartan subalgebra, up to unitary conjugacy, and that $\Gamma \actson (X,\mu)$ is orbit equivalence (OE) superrigid (see Section \ref{sec.prelim-OE} for basic terminology). This makes W$^*$-superrigidity a rare and difficult to establish phenomenon.

There are by now several orbit equivalence and cocycle superrigidity theorems for infinite measure preserving actions. We mention in particular \cite[Theorems 1.3 and 1.5]{PV08}, \cite[Theorems B and C]{Ioa14} and \cite[Theorem B and Corollary D]{GITD16}. None of these were compatible with the existing uniqueness theorems for Cartan subalgebras proven in, among others, \cite{PV11,PV12,HV12,Ioa12,BDV17}.

The main goal of this paper is to prove OE- and W$^*$-superrigidity for actions $\Gamma \actson G/P$ of dense subgroups $\Gamma$ of a Lie group $G$ acting on the homogeneous space $G/P$. Before formulating a more conceptual result, we provide a sharp dichotomy for dense subgroups of isometries of the hyperbolic plane $\bH^2$. We realize $\bH^2 = \{z \in \C \mid \Im z > 0\}$ as the upper half plane and its group of orientation preserving isometries as $\PSL(2,\R)$ acting by fractional transformations:
$$\begin{pmatrix} a & b \\ c & d \end{pmatrix} \cdot z = \frac{a z + b}{cz + d} \; .$$

\begin{letterthm}\label{thm.main-hyperbolic-plane}
Let $\cS$ be a finite set of prime numbers. Define $\Gamma = \PSL(2,\Z[\cS^{-1}])$. Consider the isometric action $\Gamma \actson \bH^2$.
\begin{enumlist}
\item If $\cS = \emptyset$, as is well known, the action $\Gamma \actson \bH^2$ admits a fundamental domain.
\item If $|\cS|=1$, the action $\Gamma \actson \bH^2$ is strongly ergodic, of type II$_\infty$ and the associated orbit equivalence relation is treeable. In particular, the action $\Gamma \actson X$ does not satisfy any cocycle superrigidity or OE/W$^*$-superrigidity property.
\item If $|\cS| \geq 2$, the action $\Gamma \actson \bH^2$ is strongly ergodic, of type II$_\infty$, cocycle superrigid with countable target groups, OE-superrigid and W$^*$-superrigid (in the sense of Definition \ref{def.superrigid}).
\end{enumlist}
\end{letterthm}

Point~2 of Theorem \ref{thm.main-hyperbolic-plane} has to be compared with the results in \cite{Ioa14}, where several strong rigidity theorems are proven for actions $\Gamma \actson G/P$ where $\Gamma < G$ is a dense subgroup of a connected Lie group and $P < G$ is a closed connected subgroup. In particular in \cite[Corollary I]{Ioa14}, under the assumption of strong ergodicity for the translation action $\Gamma \actson G$, it is proven that two such actions $\Gamma_i \actson G_i/P_i$ can only be stably orbit equivalent if they are conjugate. So while the action $\PSL(2,\Z[1/p]) \actson \bH^2$ is treeable and thus orbit equivalent to a huge variety of group actions, it follows from \cite[Corollary I]{Ioa14} that these actions are not stably orbit equivalent among each other if we vary the prime $p$.

When $\cS$ is an infinite set of prime numbers, we still have that the action $\Gamma \actson \bH^2$ in Theorem~\ref{thm.main-hyperbolic-plane} is OE-superrigid, but we do not know if W$^*$-superrigidity holds in this case (see Proposition~\ref{prop.S-infinite}).

Our main tool to prove Theorem \ref{thm.main-hyperbolic-plane} is a new cocycle superrigidity theorem for dense subgroups $\Gamma < G$ of noncompact Lie groups and the action $\Gamma \actson G$ by translation. Inspired by \cite[Theorem B]{Ioa14} and \cite[Theorem A]{GITD16}, we introduce the following concept. Recall that a surjective group homomorphism $\pi : \Gtil \to G$ between locally compact second countable (lcsc) groups is called a covering homomorphism if $\pi$ is continuous, surjective, open and $\Ker \pi$ is a discrete subgroup of $\Gtil$.

\begin{letterdef}\label{def.essential-cocycle-superrigid}
A dense subgroup $\Gamma < G$ of a lcsc group is said to be \emph{essentially cocycle superrigid with countable targets} if for any measurable $1$-cocycle $\om : \Gamma \times G \to \Lambda$ for the translation action with values in an arbitrary countable group $\Lambda$, there exists an open subgroup $G_0 < G$ and a covering $\pi : \Gtil \to G_0$ such that with $\Gammatil = \pi^{-1}(\Gamma \cap G_0)$, the lifted $1$-cocycle
$$\omtil : \Gammatil \times \Gtil \to \Lambda : \omtil(g,x) = \om(\pi(g),\pi(x))$$
is cohomologous to a group homomorphism $\Gammatil \to \Lambda$.
\end{letterdef}

If $G$ is totally disconnected, then the notion of essential cocycle superrigidity is equivalent with \emph{virtual cocycle superrigidity}: for any measurable $1$-cocycle $\om : \Gamma \times G \to \Lambda$, there exists an open subgroup $G_0 < G$ such that the restriction of $\om$ to $(\Gamma \cap G_0) \times G_0$ is cohomologous to a group homomorphism. When $G$ is a simply connected Lie group, then essential cocycle superrigidity of a dense subgroup $\Gamma < G$ is equivalent with cocycle superrigidity: any measurable $1$-cocycle $\om : \Gamma \times G \to \Lambda$ is cohomologous to a group homomorphism. Moreover, such a plain cocycle superrigidity can only hold if $G$ is simply connected. All this is explained in Propositions \ref{prop.clarify-essential-cocycle-superrigid} and \ref{prop.cocycle-superrigid-homogeneous-spaces} below.

By \cite[Theorem B]{Ioa08}, whenever $\Gamma$ has property~(T) and $\Gamma < G$ is a dense embedding in a \emph{compact profinite} group, then $\Gamma < G$ is essentially cocycle superrigid. The same was proven in \cite[Theorem 5.21]{Fur09} for dense embeddings of property~(T) groups in arbitrary \emph{compact} groups and for cocycles with values in a torsion-free group $\Lambda$. By \cite[Theorem A]{DIP19}, essential cocycle superrigidity also holds for \emph{profinite} completions $\Gamma < G$ of certain irreducible lattices $\Gamma$. We prove in this paper the cocycle superrigidity theorem \ref{thm.main-cocycle-superrigid} below for dense subgroups $\Gamma$ of lcsc groups $G$. In the particular case when $G$ is compact, we recover the just mentioned results of \cite{Ioa08,Fur09,DIP19}.

But when $G$ is noncompact, much less is known up to now.
By \cite[Theorem B]{Ioa14}, if $G$ is a simply connected lcsc group and $\Gamma < G$ is a dense subgroup such that the translation action $\Gamma \actson G$ has property~(T) in the sense of Zimmer, then $\Gamma < G$ is cocycle superrigid. So far, this is the only known noncompact case of an essentially cocycle superrigid dense subgroup $\Gamma < G$.

The key theorem of this paper is Theorem \ref{thm.main-cocycle-superrigid} below, proving essential cocycle superrigidity for dense subgroups $\Gamma < G$ that are given by projecting irreducible lattices satisfying a strong ergodicity property. We discuss numerous concrete examples in Section \ref{sec.examples}, which in particular leads to a proof of Theorem \ref{thm.main-hyperbolic-plane}, but also leads to many other superrigidity results, in particular for dense subgroups of orientation preserving isometries of hyperbolic $n$-space and for dense subgroups of $\SL(n,\R)$ acting linearly on $\R^n$, see Theorem \ref{thm.list-Wstar-superrigid} and Corollary \ref{cor.superrigid-Rn}.

We also mention here that in \cite{GITD16}, several cocycle superrigidity theorems are proven for left-right translation actions $\Gamma \times \Lambda \actson G$, when $\Gamma,\Lambda$ are dense subgroups of a lcsc group $G$. Although several of our methods are inspired by \cite{GITD16}, and also by \cite{Ioa14,DIP19}, cocycle superrigidity for left-right translation actions is quite different from cocycle superrigidity for a translation action $\Gamma \actson G$. The reason is that such a left-right translation action $\Gamma \times \Lambda \actson G$ typically remains cocycle superrigid if we enlarge $\Gamma$ and $\Lambda$, and would often hold in cases where $\Gamma \cong \Lambda \cong \F_\infty$, while cocycle superrigidity for a translation action $\Gamma \actson G$ itself is necessarily more rare: adding a generic element to $\Gamma$ would give a group that is isomorphic with the free product $\Gamma * \Z$, for which cocycle superrigidity always fails.

\begin{letterthm}\label{thm.main-cocycle-superrigid}
Let $G$ and $H$ be lcsc groups with $H$ being compactly generated. Let $\Gamma < G \times H$ be a lattice and denote by $\pr : G \times H \to G$ the projection onto the first factor. Assume that $\pr(\Gamma) < G$ is dense. Make one of the following assumptions.
\begin{enumlist}
\item $H = H_1 \times H_2$ with $H_i \actson (G \times H)/\Gamma$ being strongly ergodic for $i=1$ and ergodic for $i=2$.
\item $H$ has property~(T).
\end{enumlist}
Then, $\pr(\Gamma) < G$ is essentially cocycle superrigid with countable target groups, in the sense of Definition \ref{def.essential-cocycle-superrigid}.
\end{letterthm}

The main novelty of Theorem \ref{thm.main-cocycle-superrigid} is to prove cocycle superrigidity for translation actions $\Gamma \actson G$ for noncompact groups $G$ and without relying on property~(T). This provides many new cases given by irreducible lattices in products of rank one groups. This is discussed in detail in Section~\ref{sec.examples}, but we already mention here that it follows from Theorem \ref{thm.main-cocycle-superrigid} that the dense subgroups $\SL(2,\Z[\cS^{-1}]) < \SL(2,\R)$ and $\SO(n,1,\Z[\cS^{-1}]) < \SO(n,1,\R)$ are essentially cocycle superrigid whenever $\cS$ is a set of prime numbers with $|\cS| \geq 2$. The same holds for the dense subgroups $\SL(2,\Z[\cO_K]) < \SL(2,\R)$ and $\SO(n,1,\Z[\cO_K]) < \SO(n,1,\R)$ where $\Q \subset K \subset \R$ is a real algebraic number field with $[K:\Q] \geq 4$ and $\cO_K$ is its ring of algebraic integers. A more precise formulation can be found in Propositions \ref{prop.p-adic} and \ref{prop.real}.

When $G$ is a \emph{compact profinite} group, then the first point of Theorem \ref{thm.main-cocycle-superrigid} is precisely \cite[Theorem A]{DIP19} and the second point is precisely cocycle superrigidity for profinite actions of property~(T) groups, as proven in \cite[Theorem B]{Ioa08}. When $G$ is a simply connected lcsc group, then the second point of Theorem \ref{thm.main-cocycle-superrigid} is a special case of \cite[Theorem B]{Ioa14}. All this is discussed in detail in Remark \ref{rem.link-earlier-work} and shows that the main novelty of Theorem \ref{thm.main-cocycle-superrigid} really consists of the first point, using spectral gap rigidity, for noncompact groups $G$.

When $\Gamma < G$ is essentially cocycle superrigid with countable targets, one can deduce OE-superrigidity results for the associated actions $\Gamma \actson G/P$ whenever $P < G$ is a closed subgroup. Depending on the nature of $P$, we may get plain OE-superrigidity (as in Theorem \ref{thm.main-hyperbolic-plane}, in the sense of Definition \ref{def.superrigid}), or virtual OE-superrigidity (up to taking quotients by normal subgroups, see Proposition \ref{prop.OE-superrigid-homogeneous}). As an illustrative example, we mention as a corollary the following OE- and W$^*$-superrigidity for pmp actions, which complements \cite[Corollary K]{Ioa14} dealing with the case $\PSL(n,\R)$, $n \geq 3$.

\begin{lettercor}\label{cor.OE-superrigid-homogeneous-lattice}
Let $G = \PSL(2,\R)$ and let $\Sigma < \PSL(2,\R)$ be a lattice. Let $\cS$ be a set of prime numbers with $|\cS| \geq 2$. Define $\Gamma = \PSL(2,\Z[\cS^{-1}])$. The action $\Gamma \actson G/\Sigma$ is essentially free, ergodic and pmp. If $\Lambda \actson (Y,\eta)$ is any essentially free ergodic pmp action, the following are equivalent.
\begin{enumlist}
\item The actions $\Gamma \actson G/\Sigma$ and $\Lambda \actson Y$ are stably orbit equivalent.
\item There exist a finite index subgroup $\Lambda_0 < \Lambda$ and a finite normal subgroup $\Sigma_0 \lhd \Lambda_0$ such that $\Lambda \actson Y$ is induced from $\Lambda_0 \actson Y_0$ and $\Lambda_0/\Sigma_0 \actson Y_0/\Sigma_0$ is conjugate to $\Gamma \actson G/\Sigma$.
\end{enumlist}
If $\cS$ is finite, these statements are moreover equivalent with:
\begin{enumlist}[resume]
\item The actions $\Gamma \actson G/\Sigma$ and $\Lambda \actson Y$ are stably W$^*$-equivalent.
\end{enumlist}
\end{lettercor}

Note that in \cite[Corollary I]{Ioa14}, it was already proven that the actions $\Gamma \actson G/\Sigma$ appearing in Corollary \ref{cor.OE-superrigid-homogeneous-lattice} satisfy an OE-rigidity property among each other: if two such actions are stably orbit equivalent, then the corresponding sets of primes must be equal. We now prove that these actions are plainly OE-superrigid.

Another reason why OE-superrigid, infinite measure preserving actions $\Gamma \actson (X,\mu)$ are interesting, is the following. Assume that the action is essentially free and ergodic. Then for any Borel set $Y \subset X$ with $0 < \mu(Y) < +\infty$, the restriction of the orbit equivalence relation $\cR(\Gamma \actson X)$ to $Y$ gives a countable, ergodic, type II$_1$ equivalence relation $\cR$ with the property that none of its finite amplifications $\cR^t$, $0 < t < +\infty$, can be implemented by an essentially free group action. The question whether every countable, type II$_1$ equivalence relation can be implemented by an essentially free group action was posed in \cite{FM75} and answered negatively only much later in \cite{Fur98}.

Since then, there were found several constructions of equivalence relations $\cR$ such that none of the finite amplifications $\cR^t$ can be implemented by an essentially free group action: \cite[Theorem D]{Fur98}, \cite[Corollary 5.10]{Pop05b} and \cite[Theorem 4.1 and Proposition 7.1]{PV08}. Our Theorem \ref{thm.main-hyperbolic-plane} thus adds examples to this list. In the second part of this paper, we then prove in very high generality that such a non-implementation property is stable under taking left-right wreath products.

We first recall this wreath product construction. Assume that $\cR_0$ is a countable pmp equivalence relation on a standard probability space $(X_0,\mu_0)$. We do not assume that $\cR_0$ is ergodic, nor that the orbits of $\cR_0$ are infinite. Whenever $\Gamma$ is a countable group, we consider the product space $(X,\mu) = (X_0,\mu_0)^\Gamma$ and define the left-right wreath product equivalence relation $\cR = \cR_0 \wr_\Gamma (\Gamma \times \Gamma)$ on $(X,\mu)$ by $(x,y) \in \cR$ if and only if there exist a $(g,h) \in \Gamma \times \Gamma$ and a finite subset $\cF \subset \Gamma$ such that $x_{gkh} = y_k$ for all $k \in \Gamma \setminus \cF$ and $(x_{gkh},y_k) \in \cR_0$ for all $k \in \cF$. So, $\cR$ is generated by the product equivalence relation $\cR_0^{(\Gamma)}$ and the left-right Bernoulli action $\Gamma \times \Gamma \actson (X,\mu)$. In particular, whenever $\Gamma$ is an infinite group, the equivalence relation $\cR$ is ergodic.

\begin{letterthm}\label{thm.stable-non-implement}
Let $\cR_0$ be any nontrivial countable pmp equivalence relation on a standard probability space $(X_0,\mu_0)$ and let $\Gamma$ be any nonamenable group with infinite conjugacy classes. Let $\cR = \cR_0 \wr_\Gamma (\Gamma \times \Gamma)$ be the left-right wreath product equivalence relation. Let $t > 0$.

Then $\cR^t$ can be implemented by an essentially free group action if and only if $t \in \N$ and $\cR_0$ can be implemented by an essentially free group action.
\end{letterthm}

Since $\cR_0$ is completely arbitrary, Theorem \ref{thm.stable-non-implement} provides the following wide class of ergodic, type II$_1$ equivalence relations $\cR$ with the property that no finite amplification $\cR^t$ can be implemented by an essentially free group action. It suffices to take $X_0 = \{1,2,3\}$ with the uniform probability measure $\mu_0$ and define $\cR_0$ as the equivalence relation on $X_0$ generated by $1 \sim 2$. Clearly, $\cR_0$ cannot be implemented by an essentially free group action, since the group should have at the same time two elements and one element. So, for any nonamenable icc group $\Gamma$, the left-right wreath product $\cR = \cR_0 \wr_\Gamma (\Gamma \times \Gamma)$ has the property that no finite amplification $\cR^t$ can be implemented by an essentially free group action.

Theorem \ref{thm.stable-non-implement} is an immediate corollary of a much more general result, Theorem \ref{thm.main-wreath} below, in which we completely describe all essentially free group actions $G \actson (Y,\eta)$ that are stably orbit equivalent to $\cR = \cR_0 \wr_\Gamma (\Gamma \times \Gamma)$. We prove the following OE-superrigidity result. If $G \actson (Y,\eta)$ is any essentially free group action such that $\cR(G \actson Y) \cong \cR^t$, then $G \actson Y$ is induced from $G_1 \actson Y_1$, the group $G_1$ admits a finite normal subgroup $\Sigma \lhd G_1$ and the action $G_1/\Sigma \actson Y_1/\Sigma$ is conjugate with the natural action of the left-right wreath product group $\Lambda_0 \wr_\Gamma (\Gamma \times \Gamma)$ on $(X,\mu)$, where $\Lambda_0 \actson X_0$ is an essentially free action implementing $\cR_0$. Moreover, the amplification factor $t$ equals $|\Sigma| \, [G:G_1]$.

{\bf Acknowledgment.} We thank Filippo Calderoni for kindly pointing out that the formulation of Propositions \ref{prop.p-adic} and \ref{prop.real} in the first version of this paper was imprecise, because we overlooked that in the exceptional $4$-dimensional case, the groups $\SO(n,m)$ with $n+m=4$ are not almost simple, but locally a product of two rank $1$ groups.

\section{Preliminaries}\label{sec.prelim}

\subsection{Stable orbit equivalence and OE-superrigidity}\label{sec.prelim-OE}

A \emph{stable isomorphism} between two countable nonsingular ergodic equivalence relations $\cR_i$ on standard probability spaces $(X_i,\mu_i)$ is a nonsingular isomorphism $\Delta : \cU_1 \to \cU_2$ between nonnegligible Borel sets $\cU_i \subset X_i$ such that $\Delta$ is an isomorphism between the restricted equivalence relations $\cR_i|_{\cU_i}$. We say that $\cR_1$ and $\cR_2$ are \emph{stably isomorphic} if such a stable isomorphism exists. If the measures $\mu_i$ are $\cR_i$-invariant probability measures, the ratio $\mu_2(\cU_2)/\mu_1(\cU_1)$ is called the \emph{compression constant} of the stable isomorphism $\Delta$. Note that ergodic equivalence relations of type II$_\infty$ are stably isomorphic if and only if they are isomorphic.

We say that two nonsingular essentially free ergodic actions $\Gamma_i \actson (X_i,\mu_i)$ of countable groups $\Gamma_i$ are \emph{(stably) orbit equivalent} if their associated orbit equivalence relations are (stably) isomorphic.

Let $\Gamma \actson (X,\mu)$ be an essentially free ergodic nonsingular action of a countable group $\Gamma$. Let $\Lambda$ be a countable group. A \emph{$1$-cocycle for $\Gamma \actson X$ with target group $\Lambda$} is a Borel map $\om : \Gamma \times X \to \Lambda$ such that $\om(gh,x) = \om(g,h\cdot x) \, \om(h,x)$ for all $g,h \in \Gamma$ and a.e.\ $x \in X$. The $1$-cocycles $\om_1$ and $\om_2$ are called \emph{cohomologous} if there exists a Borel map $\vphi : X \to \Lambda$ such that $\om_1(g,x) = \vphi(g \cdot x) \, \om_2(g,x) \, \vphi(x)^{-1}$ for all $g \in \Gamma$ and a.e.\ $x \in X$. Every group homomorphism $\delta : \Gamma \to \Lambda$ gives rise to a $1$-cocycle $\om(g,x) = \delta(g)$.

Whenever $\delta : \Gamma \to \Lambda$ is an embedding of $\Gamma$ into a larger countable group $\Lambda$, there is a natural \emph{induced} action of $\Lambda$ on $\Lambda/\delta(\Gamma) \times X$, which is stably orbit equivalent with $\Gamma \actson X$. In this paper, we use the terminology \emph{OE-superrigidity} and \emph{W$^*$-superrigidity} in the following strongest possible sense.

\begin{definition}\label{def.superrigid}
Let $\Gamma \actson (X,\mu)$ be an essentially free, ergodic, nonsingular action of a countable group $\Gamma$.
\begin{itemlist}
\item We say that $\Gamma \actson X$ is \emph{OE-superrigid} if any essentially free, ergodic, nonsingular, countable group action $\Lambda \actson Y$ that is stably orbit equivalent to $\Gamma \actson X$ must be conjugate to an induction of $\Gamma \actson X$.
\item We say that $\Gamma \actson X$ is \emph{W$^*$-superrigid} if any essentially free, ergodic, nonsingular, countable group action $\Lambda \actson Y$ such that $L^\infty(Y) \rtimes \Lambda$ is stably isomorphic with $L^\infty(X) \rtimes \Gamma$ must be conjugate to an induction of $\Gamma \actson X$.
\end{itemlist}
We say that a nonsingular action $\Gamma \actson (X,\mu)$ of a countable group is \emph{cocycle superrigid with countable targets} if any $1$-cocycle with values in an arbitrary countable group is cohomologous to a group homomorphism.
\end{definition}

In the second part of this paper, proving Theorem \ref{thm.stable-non-implement} and related results on left-right wreath product equivalence relations, we will make use of a variant of Furman's approach from \cite[Section 3]{Fur98} to measure equivalence and stable orbit equivalence of group actions. We gather here the preliminary definitions and results that we need.

First recall that given an ergodic countable pmp equivalence relation $\cR$ on a standard probability space $(X,\mu)$, the amplification $\cR^t$, $t > 0$, is defined as follows: denoting by $\cR^\infty$ the equivalence relation on $X \times \N$ given by $(x,n) \sim_{\cR^\infty} (y,m)$ iff $x \sim_\cR y$, we define $\cR^t$ by restricting $\cR^\infty$ to any subset of $X \times \N$ of measure $t$. The amplification $\cR^t$ is well defined up to isomorphism.

Whenever $G_1 \actson (Y_1,\eta_1)$ is a nonsingular action of a countable group $G_1$ on a standard probability space $(Y_1,\eta_1)$ and whenever $G_1$ is embedded as a subgroup of $G$, there is a natural induced action $G \actson Y = G \times_{G_1} Y_1$. By construction, $Y_1$ is a $G_1$-invariant Borel subset of $Y$ satisfying the following conditions: $g \cdot Y_1 \cap Y_1$ is nonnegligible iff $g \in G_1$ and $G \cdot Y_1$ has a complement of measure zero. These properties precisely characterize induced actions. Note that by construction $\cR(G \actson Y) \cong \cR(G_1 \actson Y_1)^t$, where $t = [G:G_1]$.

Also recall that a countable nonsingular equivalence relation $\cT$ on a standard $\sigma$-finite measure space $(Z,\eta)$ is said to be of \emph{type~I} if one of the following equivalent conditions is satisfied.
\begin{itemlist}
\item $\cT$ admits a fundamental domain: there exists a Borel set $D \subset Z$ such that $\cT|_D$ is trivial and such that a.e.\ $x \in Z$ is equivalent with a point in $D$.
\item There exists a sequence of Borel sets $D_n \subset Z$ such that $\bigcup_n D_n$ has a complement of measure zero and $\cT|_{D_n}$ has finite orbits for every $n$.
\end{itemlist}
Note that subequivalence relations, as well as restrictions of type~I equivalence relations remain of type~I. Also note that if the measure $\eta$ is $\cT$-invariant, then every choice of fundamental domain has the same measure. When $\cT$ is of type~I, the quotient $Z/\cT$ is a well defined standard measure space, which can be either identified with a fundamental domain of $\cT$ or be characterized by $L^\infty(Z/\cT) = L^\infty(Z)^\cT$, the von Neumann algebra of $\cT$-invariant functions on $(Z,\eta)$. Also note that if $\eta$ is a $\cT$-invariant probability measure, then $\cT$ is of type~I if and only if a.e.\ $x \in Z$ has a finite orbit.

In \cite[Section 3]{Fur98}, a one-to-one correspondence between stable orbit equivalence of group actions and measure equivalence was developed. It is straightforward to adapt these constructions to our context, where we have a countable equivalence relation on one side and a group action on the other side. For completeness, we formulate and prove this result as Lemma \ref{lem.correspondence}, which is entirely similar to \cite[Section 3]{Fur98}.

Recall that a cocycle $\om : \cR \to G$ on a countable nonsingular equivalence relation $\cR$ on $(X,\mu)$ with values in a countable group $G$ is a Borel map satisfying $\om(x,y) \, \om(y,z) = \om(x,z)$ and $\om(y,x) = \om(x,y)^{-1}$ for a.e.\ $x \sim_\cR y \sim_\cR z$. To any such cocycle is associated the skew product equivalence relation $\cR_\om$ on $X \times G$ defined by $(x,g) \sim_{\cR_\om} (y,h)$ iff $(x,y) \in \cR$ and $g = \om(x,y) h$. If $\mu$ is an $\cR$-invariant probability measure, the product with the counting measure on $G$ is an $\cR_\om$-invariant measure on $X \times G$. Every right translation by $g \in G$ in the second variable defines an automorphism of $\cR_\om$, meaning that
$$(x,h) \sim_{\cR_\om} (x',h') \quad\text{iff}\quad (x,hg) \sim_{\cR_\om} (x',h'g) \; .$$
So, whenever $\cR_\om$ is of type~I, we obtain the natural essentially free action $G \actson Y = (X \times G)/\cR_\om$. If $\cR$ is ergodic, the action $G \actson Y$ is ergodic. If $\mu$ is moreover an $\cR$-invariant probability measure, the action $G \actson Y$ is measure preserving with the invariant measure being finite iff the fundamental domain of $\cR_\om$ has finite measure.

\begin{lemma}\label{lem.correspondence}
Let $\cR$ be an ergodic countable nonsingular equivalence relation on the standard probability space $(X,\mu)$. Let $G \actson (Y,\eta)$ be an essentially free, ergodic, nonsingular action of a countable group $G$ on a standard probability space $(Y,\eta)$. There is a natural correspondence between
\begin{itemlist}
\item stable orbit equivalences $\Delta : \cU \subset X \to \cV \subset Y$ between $\cR$ and the orbit equivalence relation $\cR(G \actson Y)$~;
\item pairs $(\om,\theta)$ where $\om : \cR \to G$ is a cocycle such that the skew product $\cR_\om$ is of type~I and $\theta : (X \times G)/\cR_\om \to Y$ is a nonsingular isomorphism of $G$-actions~;
\end{itemlist}
characterized by $\theta(x,e) \in G \cdot \Delta(x)$ for a.e.\ $x \in \cU$. This correspondence is uniquely defined and bijective up to the following similarities.
\begin{itemlist}
\item Two stable orbit equivalences $\Delta_i : \cU_i \to \cV_i$ are called similar iff $\Delta_2(\cU_2 \cap \cR \cdot x) \in G \cdot \Delta_1(x)$ for a.e.\ $x \in \cU_1$.
\item Two pairs $(\om_i,\theta_i)$ as above are called similar iff there exists a Borel map $\vphi : X \to G$ such that $\om_2(x,y) = \vphi(x) \, \om_1(x,y) \, \vphi(y)^{-1}$ for a.e.\ $(x,y) \in \cR$ and $\theta_2(x,g) = \theta_1(x,\vphi(x)^{-1}g)$ for a.e.\ $(x,g) \in X \times G$.
\end{itemlist}
\end{lemma}
\begin{proof}
First assume that $\Delta : \cU \subset X \to \cV \subset Y$ is a stable orbit equivalence. By ergodicity of $\cR$, we can choose a Borel map $\gamma : X \to \cU$ such that $\gamma(x) \sim_\cR x$ for a.e.\ $x \in X$ and $\gamma(x) = x$ for all $x \in \cU$. Since the action $G \actson Y$ is essentially free, this leads to a well defined cocycle
$$
\om : \cR \to G : \Delta(\gamma(x)) = \om(x,y) \cdot \Delta(\gamma(y)) \quad\text{for a.e.\ $(x,y) \in \cR$.}
$$
Since $G \actson Y$ is ergodic, we can choose disjoint Borel sets $\cV_n \subset Y$ and elements $g_n \in G$ such that $\bigcup_n \cV_n$ has a complement of measure zero and $g_n \cdot \cV_n \subset \cV$. Define
$$D := \bigcup_n \bigl( \Delta^{-1}(g_n \cdot \cV_n) \times \{g_n\}\bigr) \; .$$
We prove that $D$ is a fundamental domain of $\cR_\om$. First, if $(x,g_n)$ and $(y,g_m)$ both belong to $D$ and $(x,g_n) \sim_{\cR_\om} (y,g_m)$, then $(x,y) \in \cR$ and $g_n = \om(x,y) g_m$. Because $\gamma(z) = z$ for all $z \in \cU$, we have that $\Delta(x) = \om(x,y) \cdot \Delta(y)$. Therefore,
$$g_n^{-1} \cdot \Delta(x) = g_n^{-1} \om(x,y) \cdot \Delta(y) = g_m^{-1} \cdot \Delta(y) \; .$$
The left hand side belongs to $\cV_n$ and the right hand side to $\cV_m$. If $n \neq m$, we have $\cV_n \cap \cV_m = \emptyset$. So, $n=m$. Then $\om(x,y) = e$, so that $\Delta(x) = \Delta(y)$ and $x=y$. We have proven that $(x,g_n) = (y,g_m)$. Thus, the restriction of $\cR_\om$ to $D$ is the trivial equivalence relation. Secondly, for all $g \in G$ and a.e.\ $x \in X$, we can take a unique $n$ such that $g^{-1} \cdot \Delta(\gamma(x)) \in \cV_n$. Put $y = \Delta^{-1}(g_ng^{-1} \cdot \Delta(\gamma(x)))$. Then, $(y,g_n) \in D$. Since $\gamma(y) = y$ and $\Delta(y) = g_ng^{-1} \cdot \Delta(\gamma(x))$, we get that $\om(y,x) = g_n g^{-1}$, so that $(y,g_n) \sim_{\cR_\om} (x,g)$. So, we have proven that $D$ is a fundamental domain for $\cR_\om$.

Define $\theta : X \times G \to Y$ as the essentially unique $\cR_\om$-invariant map satisfying $\theta(x,g_n) = g_n^{-1} \cdot \Delta(x)$ whenever $(x,g_n) \in D$. Then, $\theta : (X \times G)/\cR_\om \to Y$ is an isomorphism of $G$-actions. By construction, $\theta(x,e) \in G \cdot \Delta(x)$ for a.e.\ $x \in \cU$.

Conversely, assume that we are given a pair $(\om,\theta)$ as in the formulation of the lemma. Since the restriction of $\cR_\om$ to $X \times \{e\}$ remains of type~I, we can choose a nonnegligible Borel set $\cU \subset X$ such that the restriction of $\cR_\om$ to $\cU \times \{e\}$ is the trivial equivalence relation. We then find a nonnegligible Borel set $\cV \subset Y$ and a nonsingular isomorphism $\Delta : \cU \to \cV$ such that $\theta(x,e) = \Delta(x)$ for all $x \in \cU$. Fix $x,y \in \cU$ and $g \in G$. We have $(x,y) \in \cR$ with $\om(x,y) = g$ iff $(x,g) \sim_{\cR_\om} (y,e)$ iff $\theta(x,g) = \theta(y,e)$ iff $g^{-1} \cdot \Delta(x) = \Delta(y)$. So, $\Delta$ is a stable orbit equivalence.

Finally assume that for $i \in \{1,2\}$, we are given stable orbit equivalences $\Delta_i : \cU_i \to \cV_i$ and pairs $(\om_i,\theta_i)$ as in the lemma satisfying $\theta_i(x,e) \in  G \cdot \Delta_i(x)$ for a.e.\ $x \in \cU_i$. We have to prove that $\Delta_1$ is similar to $\Delta_2$ if and only if $(\om_1,\theta_1)$ is similar to $(\om_2,\theta_2)$.

Note that since $\theta_i$ is $\cR_{\om_i}$-invariant and $G$-equivariant, we have that
\begin{equation}\label{eq.better-inclusion}
\Delta_i(\cU_i \cap \cR \cdot x) \subset G \cdot \theta_i(x,e) \quad\text{for a.e.\ $x \in X$ and all $i \in \{1,2\}$.}
\end{equation}
So if $\Delta_1$ and $\Delta_2$ are similar, we get that $\theta_2(x,e) \in G \cdot \theta_1(x,e)$ for a.e.\ $x \in X$. Since the action $G \actson Y$ is essentially free, we find a Borel map $\vphi : X \to G$ such that $\theta_2(x,e) = \vphi(x) \cdot \theta_1(x,e)$ for a.e.\ $x \in X$. It follows that for a.e.\ $(x,y) \in \cR$,
\begin{align*}
& \theta_2(y,e) = \theta_2(x,\om_2(x,y)) = \om_2(x,y)^{-1} \cdot \theta_2(x,e) = \om_2(x,y)^{-1} \vphi(x) \cdot \theta_1(x,e) \quad\text{and}\\
& \theta_2(y,e) = \vphi(y) \cdot \theta_1(y,e) = \vphi(y) \cdot \theta_1(x,\om_1(x,y)) = \vphi(y) \om_1(x,y)^{-1} \cdot \theta_1(x,e) \; .
\end{align*}
Since $G \actson Y$ is essentially free, we conclude that $\om_2(x,y) = \vphi(x) \om_1(x,y) \vphi(y)^{-1}$ for a.e.\ $(x,y) \in \cR$. By construction, we have
$$\theta_2(x,g) = g^{-1} \cdot \theta_2(x,e) = g^{-1} \vphi(x) \cdot \theta_1(x,e) = \theta_1(x,\vphi(x)^{-1} g) \; .$$
So, $(\om_1,\theta_1)$ and $(\om_2,\theta_2)$ are similar.

Conversely, assume that $(\om_1,\theta_1)$ and $(\om_2,\theta_2)$ are similar. Then $\theta_2(x,e) \in G \cdot \theta_1(x,e)$ for a.e.\ $x \in X$. It follows from \eqref{eq.better-inclusion} that $\Delta_1$ and $\Delta_2$ are similar.
\end{proof}

\begin{remark}\label{rem.induced}
Let $\cR$ be a countable nonsingular equivalence relation on $(X,\mu)$ and $\om : \cR \to G$ a cocycle with values in a countable group $G$ such that $\cR_\om$ is of type~I. Consider $G \actson Y = (X \times G)/\cR_\om$. If $G_1 \subset G$ is a subgroup and $\om(\cR) \subset G_1$, we have that $X \times G_1$ is $\cR_\om$-invariant. Denote $Y_1 = (X \times G_1)/\cR_\om \subset Y$. By construction, $G \actson Y$ is induced from $G_1 \actson Y_1$.
\end{remark}

As is well known, OE-superrigidity follows from cocycle superrigidity with countable targets. Such a result goes back to \cite[Proposition 4.2.11]{Zim84}. The most general version of this principle, for nonsingular actions of lcsc groups was contained in the proof of \cite[Lemma 5.10]{PV08}. For nonsingular actions of countable groups, it immediately follows from Lemma \ref{lem.correspondence} and we thus include a proof for completeness.

\begin{lemma}\label{lem.cocycle-OE-superrigid}
Let $\Gamma \actson (X,\mu)$ be an essentially free, ergodic, nonsingular action of a countable group $\Gamma$ on a standard probability space $(X,\mu)$. Assume that $\Gamma \actson X$ is cocycle superrigid with countable targets.

If $G \actson (Y,\eta)$ is any essentially free, ergodic, nonsingular action of a countable group on a standard probability space that is stably orbit equivalent to $\Gamma \actson (X,\mu)$, there exists a normal subgroup $N \lhd \Gamma$ such that $N \actson X$ admits a fundamental domain and $G \actson Y$ is conjugate to an induction of $\Gamma/N \actson X/N$.
\end{lemma}
\begin{proof}
We apply Lemma \ref{lem.correspondence} to $\cR = \cR(\Gamma \actson X)$. By cocycle superrigidity, we find a group homomorphism $\delta : \Gamma \to G$ such that the action
$$\Gamma \actson X \times G : g \cdot (x,k) = (g \cdot x, \delta(g)k)$$
admits a fundamental domain and $G \actson Y$ is isomorphic with the action of $G$ on $(X \times G)/\Gamma$ by translation in the second variable. Define $G_1 = \delta(\Gamma)$. By Remark \ref{rem.induced}, we find that $G \actson Y$ is induced from an action $G_1 \actson Y_1$ that is isomorphic with the action of $G_1$ on $(X \times G_1)/\Gamma$. Define $N = \Ker \delta$. Since the action $\Gamma \actson X \times G_1$ admits a fundamental domain, also $N \actson X$ admits a fundamental domain, say $D \subset X$. Then $D \times \{e\}$ is a fundamental domain for $\Gamma \actson X \times G_1$ and we obtain a conjugacy between the actions $\Gamma/N \actson X/N$ and $G_1 \actson Y_1$.
\end{proof}

\subsection{Cross section equivalence relations}\label{sec.cross-section}

Throughout this section, $G$ is a lcsc group and $(X,\mu)$ is a standard probability space.

Let $G \actson (X,\mu)$ be an essentially free nonsingular action. Following \cite{For74}, a \emph{cross section} is a Borel set $Y \subset X$ for which there exists a neighborhood $\cU$ of the identity in $G$ such that the map $\theta : \cU \times Y \to X : (g,y) \mapsto g \cdot y$ is injective and such that $G \cdot Y$ has a complement of measure zero. Note here that the first condition implies that the Borel map $G \times Y \to X : (g,y) \mapsto g \cdot y$ is countable-to-one, so that $G \cdot Y$ is a Borel set. One may then choose a probability measure $\eta$ on $Y$ such that $\theta$ is a nonsingular isomorphism between $\cU \times Y$ and $\cU \cdot Y$.

By \cite[Proposition 2.10]{For74}, every essentially free nonsingular action $G \actson (X,\mu)$ admits a cross section $Y \subset X$. Then,
$$\cR = \{(x,y) \in Y \times Y \mid x \in G \cdot y\}$$
is a countable nonsingular Borel equivalence relation on $(Y,\eta)$. It is called a \emph{cross section equivalence relation} of $G \actson (X,\mu)$. If the action $G \actson (X,\mu)$ is moreover ergodic, all cross section equivalence relations are stably isomorphic. Also, the von Neumann algebra $L(\cR)$ is stably isomorphic with the crossed product $L^\infty(X) \rtimes G$. This and related results on cross section equivalence relations are discussed in \cite[Section 4]{KPV13}.

We say that an essentially free nonsingular action $G \actson (X,\mu)$ is of \emph{type~I} if a (equivalently, all) cross section equivalence relation is of type~I. By choosing a fundamental domain for a cross section equivalence relation, one sees that being of type~I is equivalent with the existence of a Borel set $Y \subset X$ and a probability measure $\eta$ on $Y$ such that the map $\theta : G \times Y \to X : (g,y) \mapsto g \cdot y$ is a nonsingular isomorphism. We denote the standard probability space $(Y,\eta)$ as $X/G$. Note that we can more canonically define $X/G$ by identifying $L^\infty(X/G)$ with the von Neumann algebra of $G$-invariant elements in $L^\infty(X)$.

A typical example of a free action of type~I is given by the translation action $H \actson G$ of a closed subgroup $H$ of a lcsc group $G$.

We note here, for later use, that if $G \actson (X,\mu)$ is an essentially free nonsingular action of type~I, then every $1$-cocycle $\om : G \times X \to \Lambda$ with values in a Polish group $\Lambda$ is cohomologous to the trivial cocycle: there exists a Borel map $\vphi : X \to \Lambda$ such that $\om(g,x) = \vphi(g \cdot x) \, \vphi(x)^{-1}$ for all $g \in G$ and a.e.\ $x \in X$. Indeed, by the Fubini theorem, there exists a $k \in G$ such that $\vphi : X \to \Lambda$ defined by $\vphi(\theta(h,y)) = \om(hk^{-1},\theta(k,y))$ satisfies the required property.

If $H \lhd G$ is a closed normal subgroup and $G \actson (X,\mu)$ is an essentially free nonsingular ergodic action such that the restriction $H \actson (X,\mu)$ is of type~I, we get the canonical nonsingular action $G/H \actson X/H$, which is still essentially free and ergodic. Writing $X$ as $H \times Y$ (up to measure zero), so that we can identify $X/H$ and $Y$, we see that a cross section $Z \subset X/H = Y$ for the action $G/H \actson X/H$ also is a cross section for $G \actson X$, with identical cross section equivalence relations. So, all cross section equivalence relations of $G \actson X$ and $G/H \actson X/H$ are stably isomorphic.

In this paper, we will briefly use Gromov's notion of measure equivalence, as developed in \cite[Section 3]{Fur98}. Let $G_1$ and $G_2$ be lcsc groups and let $G_1 \times G_2 \actson (Z,\mu)$ be an essentially free nonsingular ergodic action. Assume that both $G_1 \actson Z$ and $G_2 \actson Z$ are of type~I. Since the actions of $G_1$ and $G_2$ commute, the actions $G_1 \actson Z/G_2$ and $G_2 \actson Z/G_1$ are well defined, essentially free, ergodic and nonsingular actions. By the previous paragraph, their cross section equivalence relations are both stably isomorphic with the cross section equivalence relation of $G_1 \times G_2 \actson Z$. In particular, the von Neumann algebras $L^\infty(Z/G_2) \rtimes G_1$, $L^\infty(Z/G_1) \rtimes G_2$ and $L^\infty(Z) \rtimes (G_1 \times G_2)$ are all stably isomorphic.

\subsection{Treeability}\label{sec.treeable}

Following \cite{Ada88}, a countable nonsingular equivalence relation $\cR$ on a standard probability space $(X,\mu)$ is called \emph{treeable} if there exists a \emph{treeing}, i.e.\ a Borel set $\cT \subset \cR$ that is the set of edges of a graph (undirected, without loops) with vertex set $X$ such that for a.e.\ $x \in X$, the graph with vertex set $\cR \cdot x = \{y \in X \mid (y,x) \in \cR\}$ and edges $\cT \cap (\cR \cdot x \times \cR \cdot x)$ is a tree.

Let $\cR$ be treeable and assume that $\cR$ is ergodic and nonamenable. If $\nu \sim \mu$ is a $\sigma$-finite equivalent measure that is $\cR$-invariant, by \cite[Corollary 1.2]{Hjo05}, we can choose a nonnegligible Borel set $Y \subset X$ such that the restriction $\cR|_Y$ can be written as the orbit equivalence relation of an essentially free ergodic pmp action $\F_n \actson Y$ of a free group $\F_n$ with $2 \leq n \leq +\infty$. For that reason, treeable equivalence relations preserving a finite or infinite measure can never satisfy cocycle or OE-superrigidity properties (see e.g.\ \cite[Theorem 2.27]{MS02}).

Following \cite[Definition B.1]{CGMTD21}, a lcsc group $G$ is called \emph{strongly treeable} if for every essentially free pmp action $G \actson (X,\mu)$, the cross section equivalence relation is treeable. By \cite[Theorem 6]{CGMTD21}, the group $\SL(2,\R)$ is strongly treeable. By \cite[Example A.10]{CGMTD21}, a lcsc group $G$ that admits a proper action on a tree is strongly treeable. This includes the groups $\SL(2,\Q_p)$, see e.g.\
\cite[Chapter 2]{Ser80}, and more generally, algebraic groups over $\Q_p$ with $\Q_p$-rank equal to one, because their Bruhat-Tits building is a tree (see \cite{BT71}).

\section{Cocycle and OE superrigidity, proof of Theorem~\ref{thm.main-cocycle-superrigid}}

The main goal of this section is to prove the cocycle superrigidity Theorem \ref{thm.main-cocycle-superrigid}. In Proposition \ref{prop.clarify-essential-cocycle-superrigid}, we reinterpret the notion of essential cocycle superrigidity in two important cases: for totally disconnected groups and for connected Lie groups. We deduce cocycle and OE-superrigidity results for the associated actions $\Gamma \actson G/P$ on homogeneous spaces in Propositions \ref{prop.cocycle-superrigid-homogeneous-spaces}, \ref{prop.cocycle-superrigid-homogeneous-totally-disconnected} and \ref{prop.OE-superrigid-homogeneous}, each dealing with different situations depending on the connectivity properties of $G$ and $P$.

\begin{proof}[{Proof of Theorem~\ref{thm.main-cocycle-superrigid}}]
We view $\pr(\Gamma)$ as acting by right translations on $G$ and take a $1$-cocycle $\om : \pr(\Gamma) \times G \to \Lambda$ with values in a countable group $\Lambda$. We have to prove that $\om$ essentially untwists, as described in Definition \ref{def.essential-cocycle-superrigid}.

Consider the action of $H \times \Gamma$ on $G \times H$ by left translation by $H$ in the second variable and right translation by $\Gamma$. Define the $1$-cocycle
\begin{equation}\label{eq.cocycle-Om}
\Om : (H \times \Gamma) \times (G \times H) \to \Lambda : \Om((h,\gamma),(g,k)) = \om(\pr(\gamma),g) \; .
\end{equation}
We will prove Theorem \ref{thm.main-cocycle-superrigid} by constructing a covering $\pi : \Gtil \to G_0 \subset G$ such that the lifted $1$-cocycle
$$\Omtil : (H \times \Gammatil) \times (\Gtil \times H) \to \Lambda$$
can be extended to a $1$-cocycle $(\Gtil \times H \times \Gammatil) \times (\Gtil \times H) \to \Lambda$, where $\Gtil \times H$ acts by left translation and $\Gammatil$ acts by right translation. Such an extended $1$-cocycle is trivially cohomologous to a group homomorphism $\Gammatil \to \Lambda$ and this will conclude the proof of the theorem. To construct the lift and the extension of the $1$-cocycle, we will first locally extend $\Om$ to a map $(\cU \times H \times \Gamma) \times (G \times H) \to \Lambda$, where $\cU$ is a neighborhood of $e$ in $G$.

Write $Y = (G \times H)/\Gamma$ and denote by $\mu$ the $(G\times H)$-invariant probability measure on $Y$. As explained in Section \ref{sec.cross-section}, since $\Gamma$ is a closed subgroup of $G \times H$, the restriction of $\Omega$ to $\Gamma \times (G \times H)$ is cohomologous to the trivial $1$-cocycle. This means that we find a $1$-cocycle
$$C : H \times Y \to \Lambda$$
such that $\Om$ is cohomologous to $((h,\gamma),(g,k)) \mapsto C(h,(g,k)\Gamma)$.

We want to extend $C$ to a $1$-cocycle $(G \times H) \times Y \to \Lambda$ for the natural action of $G \times H$ on $Y$. We can however only do this locally, on a neighborhood $\cU$ of $e$ in $G$, and then do it globally on the appropriate lift of $G$. We define for every $k \in G$, the $1$-cocycle
$$C_k : H \times Y \to \Lambda : C_k(h,y) = C(h, k \cdot y) \; .$$
Constructing the (local) extension of $C$ basically amounts to proving that for $k$ close to $e$, the $1$-cocycles $C_k$ and $C$ are cohomologous.

{\bf Notations.} Consider the Polish group $\cF(Y,\Lambda)$ of measurable maps from $Y$ to $\Lambda$ (up to equality a.e.) with metric $d(F,F') = \mu \bigl(\{y \in Y \mid F(y) \neq F'(y)\}\bigr)$. Denote by $Z^1$ the space of $1$-cocycles $c : H \times Y \to \Lambda$, i.e.\ the space of continuous maps $c : H \to \cF(Y,\Lambda)$ satisfying $c_{hh'}(\cdot) = c_h(h'\cdot) \, c_{h'}(\cdot)$ for all $h,h' \in H$.

Whenever $c \in Z^1$ and $\vphi \in \cF(Y,\Lambda)$, define the cohomologous $1$-cocycle $\vphi \cdot c$ by
$$(\vphi \cdot c)(h,y) = \vphi(h \cdot y) \, c(h,y) \, \vphi(y)^{-1} \; .$$
Denote by $1 \in \cF(Y,\Lambda)$ the unit element, given by $y \mapsto e$.

Note that if $\vphi,\psi \in \cF(Y,\Lambda)$ and $c \in Z^1$ are such that $\vphi \cdot c = \psi \cdot c$ and $d(\vphi,\psi) < 1$, then $\vphi = \psi$. Indeed, in that case the set $\{y \in Y \mid \vphi(y) = \psi(y)\}$ is non-negligible and $H$-invariant, and thus of measure $1$, because $H \actson Y$ is ergodic.

{\bf Step 1.} \textit{There exists an open neighborhood $\cU$ of $e$ in $G$ and, for every $k \in \cU$, an element $\vphi_k \in \cF(Y,\Lambda)$ with $d(\vphi_k,1) < 1/4$ and $C_k = \vphi_k \cdot C$.}

When $H$ has property~(T), it follows from \cite[Proposition 5.14]{Fur09} that there exists an open neighborhood $\cU$ of $e$ in $G$ such that
$$\mu\bigl(\{ y \in Y \mid C_k(h,y) \neq C(h,y)\}\bigr) < 1/8 \quad\text{for all $k \in \cU$ and $h \in H$.}$$
Step~1 then follows from \cite[Lemma 2.1]{DIP19}.

When $H = H_1 \times H_2$ with the action $H_1 \actson Y$ being strongly ergodic and the action $H_2 \actson Y$ being ergodic, we prove step~1 by repeating part of the proofs of \cite[Lemma 3.1]{GITD16} and \cite[Lemma 2.4]{DIP19}. We start by constructing an open neighborhood $\cU$ of $e$ in $G$ such that
\begin{equation}\label{eq.estim-1-8}
\mu\bigl(\{ y \in Y \mid C_k(h,y) \neq C(h,y)\}\bigr) < 1/8 \quad\text{for all $k \in \cU$ and $h \in H_2$.}
\end{equation}

Since $H_1 \actson Y$ is strongly ergodic, we can take a compact subset $K_1 \subset H_1$ and a $\delta > 0$ such that the following holds: if $\cW \subset Y$ is a Borel set and $\mu(\cW \vartriangle g \cdot \cW) < \delta$ for all $g \in K_1$, then $\min \{\mu(\cW),\mu(Y \setminus \cW)\} < 1/8$.

For $i \in \{1,2\}$, $k \in G$ and $h \in H_i$, write
$$A_i(k,h) = \{y \in Y \mid C_k(h,y) = C(h,y)\} \; .$$
Take a compact subset $K_2 \subset H_2$ that generates $H_2$ as a group. Take an open $\cU \subset G$ with $e \in \cU$ such that $\mu(A_1(k,h_1)) > 1-\delta/2$ and $\mu(A_2(k,h_2)) > 7/8$ for all $k \in \cU$, $h_1 \in K_1$ and $h_2 \in K_2$. Whenever $k \in G$, $h_1 \in H_1$, $h_2 \in H_2$, $y \in Y$, we have
\begin{align*}
& C_k(h_1, h_2 \cdot y) \, C_k(h_2,y) = C_k(h_1h_2,y) = C_k(h_2h_1,y) = C_k(h_2,h_1 \cdot y) \, C_k(h_1,y) \; ,\\
& C(h_1, h_2 \cdot y) \, C(h_2,y) = C(h_1h_2,y) = C(h_2h_1,y) = C(h_2,h_1 \cdot y) \, C(h_1,y) \; .
\end{align*}
So, whenever $k \in \cU$, $h_1 \in K_1$, $h_2 \in H_2$ and $y \in A_1(k,h_1) \cap h_2^{-1} \cdot A_1(k,h_1)$, we have that $y \in A_2(k,h_2)$ iff $y \in h_1^{-1} \cdot A_2(k,h_2)$. Since
$$\mu( A_1(k,h_1) \cap h_2^{-1} \cdot A_1(k,h_1)) > 1 - \delta \; ,$$
it follows that $\mu(h_1 \cdot A_2(k,h_2) \vartriangle A_2(k,h_2)) < \delta$ for all $h_1 \in K_1$, $h_2 \in G_2$. Therefore,
$$\min\{ \mu(A_2(k,h)) , \mu(Y \setminus A_2(k,h))\} < 1/8$$
for all $k \in \cU$ and $h \in H_2$. Fix $k \in \cU$ and write
$$F_k = \{h \in H_2 \mid \mu(Y \setminus A_2(k,h)) < 1/8 \} \; .$$
By construction, $K_2 \subset F_k$. We prove that $F_k$ is a subgroup of $H_2$.

Since $A_2(k,h^{-1}) = h \cdot A_2(k,h)$ for all $h \in H_2$, we have that $F_k = F_k^{-1}$. When $h,h' \in F_k$, the cocycle identity implies that $\mu(Y \setminus A_2(k,hh')) < 1/4$. Since $1-1/4 > 1/8$, it is impossible that $\mu(A_2(k,hh'))< 1/8$. So, $\mu(Y \setminus A_2(k,hh')) < 1/8$ and $hh' \in F_k$.

Since $F_k \subset H_2$ is a subgroup and $K_2 \subset F_k$, we get that $F_k = H_2$. We have thus proven that \eqref{eq.estim-1-8} holds.

By \cite[Lemma 2.1]{DIP19}, we find for every $k \in \cU$, a unique $\vphi_k \in \cF(Y,\Lambda)$ such that $d(\vphi_k,1) < 1/4$ and
\begin{equation}\label{eq.wehavethisalready}
C_k(h_2,y) = (\vphi_k \cdot C)(h_2,y) \quad\text{for all $h_2 \in H_2$ and a.e.\ $y \in Y$.}
\end{equation}
Take a compact subset $K_0 \subset H_1$ that generates $H_1$ as a group. Making $\cU$ smaller if necessary, we may assume that for every $k \in \cU$ and $h_1 \in K_0$, we have that $\mu(A_1(k,h_1)) \geq 3/4$. Since $d(\vphi_k,1) < 1/4$, we get that for every $h_1 \in K_0$, the set
$$\{y \in Y \mid (\vphi_k \cdot C)(h_1,y) = C(h_1,y)\}$$
has measure at least $1/2$. We conclude that for every $k \in \cU$ and $h_1 \in K_0$, the set
$$Y(k,h_1) := \{y \in Y \mid C_k(h_1,y) = (\vphi_k \cdot C)(h_1,y)\}$$
is nonnegligible. Fix $k \in \cU$ and $h_1 \in K_0$. Using \eqref{eq.wehavethisalready} and the cocycle identity for $C_k$ and $\vphi_k \cdot C$, we get for all $h_2 \in H_2$ and a.e.\ $y \in Y(k,h_1)$ that
\begin{align*}
C_k(h_1,h_2 \cdot y) &= C_k(h_1h_2,y) \, C_k(h_2,y)^{-1} = C_k(h_2 h_1,y) \, C_k(h_2,y)^{-1} \\
& = C_k(h_2,h_1 \cdot y) \, C_k(h_1,y) \, C_k(h_2,y)^{-1} \\
& = (\vphi_k \cdot C)(h_2,h_1 \cdot y) \, (\vphi_k \cdot C)(h_1,y) \, (\vphi_k \cdot C)(h_2,y)^{-1}\\
& = (\vphi_k \cdot C)(h_1,h_2 \cdot y) \; ,
\end{align*}
so that $h_2 \cdot y \in Y(k,h_1)$. The nonnegligible set $Y(k,h_1) \subset Y$ is thus essentially $H_2$-invariant. Since $H_2 \actson Y$ is ergodic, we conclude that $\mu(Y(k,h_1)) = 1$ for all $k \in \cU$ and $h_1 \in K_0$. Since $K_0$ generates $G_1$ as a group, the cocycle identity implies that $C_k = \vphi_k \cdot C$ for all $k \in \cU$. So, step~1 is complete.

{\bf Step 2.} \textit{The map $\cU \to \cF(Y,\Lambda) : k \mapsto \vphi_k$ is continuous and satisfies the local $1$-cocycle property}
\begin{equation}\label{eq.cocycle-identity-varphi}
\vphi_{kk'}(y) = \vphi_k(k' \cdot y) \, \vphi_{k'}(y) \quad\text{\it for all $k,k' \in \cU$ with $kk' \in \cU$ and for a.e.\ $y \in Y$.}
\end{equation}
We first prove that $d(\vphi_k,1) \to 0$ when $k \to e$. We could have proven this while constructing $\vphi_k$ and by checking the constants in the proof of \cite[Lemma 2.1]{DIP19}. But also the following argument works. Choose $\eps > 0$ with $\eps < 1/2$. Since $H_1 \actson Y$ is strongly ergodic, take a compact subset $K_1 \subset H_1$ and a $\delta > 0$ such that whenever $\cW \subset Y$ is a Borel set satisfying $\mu(\cW \vartriangle h_1 \cdot \cW) < \delta$ for all $h_1 \in K_1$, then $\min\{\mu(\cW),\mu(Y \setminus \cW)\} < \eps$. Take an open subset $\cU_0 \subset \cU$ with $e \in \cU_0$ such that $\mu(A_1(k,h_1)) > 1-\delta/2$ for all $k \in \cU_0$ and $h_1 \in K_1$. Fix $k \in \cU_0$ and define $\cW_k = \{y \in Y \mid \vphi_k(y) = e\}$. Fix $h_1 \in K_1$. Since
$$\vphi_k(h_1 \cdot y) = C_k(h_1,y) \, \vphi_k(y) \, C(h_1,y)^{-1}$$
for a.e.\ $y \in Y$, we conclude that $h_1 \cdot (\cW_k \cap A_1(h_1,k)) = \cW_k \cap h_1 \cdot A_1(h_1,k)$. We get that $\mu(\cW_k \vartriangle h_1 \cdot \cW_k) < \delta$ for all $h_1 \in K_1$. Since $\mu(\cW_k) \geq 3/4 > \eps$, we conclude that $\mu(\cW_k) > 1-\eps$. This means that $d(\vphi_k,1) < \eps$ for all $k \in \cU_0$.

Next take $k,k' \in \cU$ with $kk' \in \cU$. We prove that \eqref{eq.cocycle-identity-varphi} holds. Define $\psi \in \cF(Y,\Lambda)$ by $\psi(y) = \vphi_k(k' \cdot y) \, \vphi_{k'}(y)$. Then, $d(\psi,1) \leq 1/2$. Evaluating the equality $C_k = \vphi_k \cdot C$ in $(h,k'\cdot y)$ implies that $C_{kk'} = \vphi_k(k'\cdot) \cdot C_{k'}$ and thus $C_{kk'} = \psi \cdot C$. Also, $C_{kk'} = \vphi_{kk'} \cdot C$. Since $d(\vphi_{kk'},1) < 1/4$, we have that $d(\psi,\vphi_{kk'}) < 3/4$. Thus, $\psi = \vphi_{kk'}$ and \eqref{eq.cocycle-identity-varphi} is proven.

We have already proven that $d(\vphi_k,1) \to 0$ when $k \to e$. In combination with \eqref{eq.cocycle-identity-varphi}, it follows that the map $\cU \to \cF(Y,\Lambda) : k \mapsto \vphi_k$ is continuous. So, step~2 is complete.

{\bf Step 3.} \textit{Construction of the covering $\pi : \Gtil \to G_0 \subset G$ and extending the lifted $1$-cocycle.}

By the general version of Iwasawa's splitting theorem proven in \cite[Theorem 4.1]{HM00}, we can pick a compact subgroup $K < G$, a simply connected Lie group $L$ and a continuous group homomorphism $\rho : L \to G$ such that the following holds: $K \subset \cU$, the subgroups $K$ and $\rho(L)$ of $G$ commute and the resulting continuous group homomorphism $\pi : K \times L \to G : \pi(k,g) = k \rho(g)$ is open and has discrete kernel $\Ker \pi$. We write $G_0 = \pi(K \times L)$ and $\Gtil = K \times L$.

By \eqref{eq.cocycle-identity-varphi}, the restriction $K \to \cF(Y,\Lambda) : k \mapsto \vphi_k$ is a $1$-cocycle.

We now reproduce part of the argument of \cite[Theorem 5.21]{Fur09} to prove the following claim: there exists a unique continuous map $\psi : L \to \cF(Y,\Lambda) : k \mapsto \psi_k$ such that $\psi_e = 1$ and $C_{\rho(k)} = \psi_k \cdot C$ for all $k \in L$. Write $\cV = \rho^{-1}(\cU)$. For every $k \in \cV$, define $\psi_k = \vphi_{\rho(k)}$. By construction, $C_{\rho(k)} = \psi_k \cdot C$ and $d(\psi_k,1) < 1/4$.

Evaluating the equality $C_{\rho(k)} = \psi_k \cdot C$ in $(\rho(k'),y)$, we get that $C_{\rho(kk')} = \psi_k(\rho(k')\cdot) \cdot C_{\rho(k')}$ for all $k \in \cV$ and $k' \in L$. It follows that the equivalence relation on $L$ defined by $k \sim k'$ iff $C_{\rho(k)}$ and $C_{\rho(k')}$ are cohomologous, has open equivalence classes. Since $L$ is connected, it follows that $C_{\rho(k)}$ is cohomologous to $C$ for all $k \in L$.

Define the closed subset $T \subset L \times \cF(Y,\Lambda)$ by
$$T = \{(k,\vphi) \in L \times \cF(Y,\Lambda) \mid C_{\rho(k)} = \vphi \cdot C \} \; .$$
Define the continuous map $p : T \to L : p(k,\vphi) = k$. In the previous paragraph, we have proven that $p$ is surjective. We have that $(e,1) \in T$ and $p(e,1) = e$. To prove the claim above, since $L$ is simply connected, it suffices to prove that $p$ is a covering map.

Fix $k_0 \in L$. Define $I = \{\vphi \in \cF(Y,\Lambda) \mid C_{\rho(k_0)} = \vphi \cdot C\}$. For every $\vphi \in I$, define the open set $\cW_\vphi \subset T$ by
$$\cW_\vphi = T \cap (\cV k_0 \times B_\vphi(1/4)) \; ,$$
where $B_\vphi(1/4)$ is the open ball with radius $1/4$ around $\vphi \in \cF(Y,\Lambda)$.

When $\vphi,\psi \in I$ are distinct, since $\vphi \cdot C = \psi \cdot C$, we have that $d(\vphi,\psi) = 1$. So, the open sets $(\cW_\vphi)_{\vphi \in I}$ are disjoint. When $\vphi \in I$ and $k \in \cV$, the element $\zeta_k \in \cF(Y,\Lambda)$ defined by
$$\zeta_k(y) = \psi_k(\rho(k_0) \cdot y) \, \vphi(y)$$
satisfies $C_{\rho(kk_0)} = \zeta_k \cdot C$ and $d(\zeta_k,\vphi) < 1/4$. So, $\zeta_k$ is the necessarily unique element of $B_\vphi(1/4)$ satisfying $C_{\rho(kk_0)} = \zeta_k \cdot C$. We have proven that $\cW_\vphi = \{(kk_0,\zeta_k) \mid k \in \cU\}$. Since $k \mapsto \zeta_k$ is continuous, the restriction of $p$ to $\cW_\vphi$ is a homeomorphism of $\cW_\vphi$ onto $\cV k_0$.

Finally, when $k \in \cV$ and $(kk_0,\zeta) \in T$ is an element in $p^{-1}(\cV k_0)$, we have $C_{\rho(kk_0)} = \zeta \cdot C$. Since $C_{\rho(kk_0)} = \psi_k(\rho(k_0) \cdot \,) \cdot C_{\rho(k_0)}$, we get that $\vphi := \psi_k(\rho(k_0) \cdot \,)^{-1} \, \zeta(\cdot)$ belongs to $I$, so that $(kk_0,\zeta) \in \cW_\vphi$. We have thus proven that $p^{-1}(\cV k_0)$ equals $\bigcup_{\vphi \in I} \cW_\vphi$.

We have therefore proven that $p : T \to G_0$ is a covering map and the claim above follows.

We next prove that $L \to \cF(Y,\Lambda) : k \mapsto \psi_k$ is a $1$-cocycle, i.e.\ for all $k,k' \in L$, we have that $\psi_{kk'}(y) = \psi_k(\rho(k')\cdot y) \, \psi_{k'}(y)$ for a.e.\ $y \in Y$. We first prove that for all $k \in L$, we have $\psi_k = \psi_{k^{-1}}(\rho(k)\cdot)^{-1}$. Write $\zeta_k = \psi_{k^{-1}}(\rho(k)\cdot)^{-1}$. Then, $k \mapsto \zeta_k$ is continuous, $\zeta_e = 1$ and $C_{\rho(k)} = \zeta_k \cdot C$ for all $k \in L$. By uniqueness, $\zeta_k = \psi_k$ for all $k \in L$.

Next, fix $k_0 \in L$. Defining $\gamma_k(y) = \psi_k(\rho(k_0)\cdot y) \, \psi_{k_0}(y)$, we have $C_{\rho(kk_0)} = \gamma_k \cdot C$ for all $k \in L$. Replacing $k$ by $kk_0^{-1}$, we also have that $C_{\rho(k)} = \gamma_{kk_0^{-1}} \cdot C$ for all $k \in L$. Since $k \mapsto \gamma_{kk_0^{-1}}$ is continuous and, by the previous paragraph, maps $e$ to $1$, it follows from uniqueness that $\gamma_{kk_0^{-1}} = \psi_k$ for all $k \in L$. Thus, $\psi_{kk_0} = \gamma_k$ for all $k \in L$, proving the cocycle identity.

We finally prove that there is a unique $1$-cocycle $\Psi : \Gtil \to \cF(Y,\Lambda)$ such that $\Psi_{(k,e)} = \vphi_k$ for all $k \in K$ and $\Psi_{(e,l)} = \psi_l$ for all $l \in L$. Uniqueness being obvious, we only have to prove that
$$\vphi_k(\rho(l) k^{-1} \cdot) \, \psi_l(k^{-1} \cdot) \, \vphi_{k^{-1}}(\cdot) = \psi_l \quad\text{for all $k \in K$, $l \in L$.}$$
Fix $k \in K$ and denote the left hand side by $\xi_l$. Then, the map $l \mapsto \xi_l$ is continuous, $\xi_e = 1$ and
$$\xi_l \cdot C = C_{k\rho(l)k^{-1}} = C_{\rho(l)} = \psi_l \cdot C \quad\text{for all $l \in L$.}$$
By uniqueness, $\xi_l = \psi_l$ for all $l \in L$. So, $\Psi$ is a well defined $1$-cocycle. We have $C_{\pi(g)} = \Psi_g \cdot C$ for all $g \in \Gtil$.

Consider the action of $\Gtil \times H$ on $Y$ given by $(g,h) \cdot y = (\pi(g),h)\cdot y$. Since $C_{\pi(g)} = \Psi_g \cdot C$ and since $\Psi$ is a $1$-cocycle, we get that
$$\Om_1 : (\Gtil \times H) \times Y \to \Lambda : \Om_1((g,h),y) = \Psi_g(h\cdot y) \, C(h,y)$$
is a $1$-cocycle that extends the $1$-cocycle $C : H \times Y \to \Lambda$.

{\bf Step 4.} \textit{End of the proof.}

Define $\Gamma_0 = \Gamma \cap (G_0 \times H)$ and $\Gammatil = (\pi \times \id)^{-1}(\Gamma_0)$. Denote by $\pr : \Gtil \times H \to \Gtil$ the projection onto the first factor. Then, $\Gammatil$ is a lattice in $\Gtil \times H$ and $\pr(\Gammatil) = \pi^{-1}(\pr(\Gamma) \cap G_0)$. Note that $(\Gtil \times H)/\Gammatil = (G_0 \times H)/\Gamma_0$ can be viewed as a nonnegligible subset $Y_0 \subset Y = (G \times H)/\Gamma$. Define the $1$-cocycles
\begin{align*}
& \Omtil_1 : (\Gtil \times H \times \Gammatil) \times (\Gtil \times H) \to \Lambda : \Omtil_1((g,h,\gamma),(x,y)) = \Om_1((g,h),(\pi(x),y)\Gamma) \; ,\\
& \Ctil : (H \times \Gammatil) \times (\Gtil \times H) \to \Lambda : \Ctil((h,\gamma),(x,y)) = C(h,(\pi(x),y)\Gamma) \; .
\end{align*}
Since $\Om_1$ extends $C$, we get that $\Omtil_1$ extends $\Ctil$. By the definition of $C$ after \eqref{eq.cocycle-Om}, we get that $\Ctil$ is cohomologous to the $1$-cocycle
$$\Omtil : (H \times \Gammatil) \times (\Gtil \times H) \to \Lambda : \Omtil((h,\gamma),(x,y)) = \om(\pi(\pr(\gamma)),\pi(x)) \; .$$
Since $\Omtil_1$ extends $\Ctil$, also the $1$-cocycle $\Omtil$ can be extended to a $1$-cocycle $(\Gtil \times H \times \Gammatil) \times (\Gtil \times H) \to \Lambda$. Such a $1$-cocycle is always cohomologous to a group homomorphism $\Gammatil \to \Lambda$. In particular, $\Omtil$ is cohomologous to a group homomorphism $\delta : \Gammatil \to \Lambda$.

This means that we find a Borel map $V : \Gtil \times H \to \Lambda$ such that for all $h \in H$ and $\gamma \in \Gammatil$, we have
$$V((e,h)(g,k)\gamma^{-1}) \, \delta(\gamma) \, V(g,k)^{-1} = \om(\pi(\pr(\gamma)),\pi(g)) \quad\text{for a.e.\ $(g,k) \in \Gtil \times H$.}$$
Taking $\gamma = e$, it follows that $V$ essentially only depends on the $\Gtil$-variable. We thus find a Borel map $W : \Gtil \to \Lambda$ such that $V(g,k) = W(g)$ for a.e.\ $(g,k) \in \Gtil \times H$. It follows that for all $\gamma \in \Gammatil$, we have
$$W(g \pr(\gamma)^{-1}) \, \delta(\gamma) \, W(g)^{-1} = \om(\pi(\pr(\gamma)),\pi(g)) \quad\text{for a.e.\ $g \in \Gtil$.}$$
If $\pr(\gamma) = e$, it follows that $W(g) \, \delta(\gamma) \, W(g)^{-1} = e$ for a.e.\ $g \in \Gtil$ and thus, $\delta(\gamma) = e$. Since $\pr(\Gammatil) = \pi^{-1}(\pr(\Gamma) \cap G_0)$, we find a group homomorphism $\delta_0 : \pi^{-1}(\pr(\Gamma) \cap G_0) \to \Lambda$ such that $\delta = \delta_0 \circ \pr$. We have proven that
$$W(g \gamma^{-1}) \, \delta_0(\gamma) \, W(g)^{-1} = \om(\pi(\gamma),\pi(g))$$
for all $\gamma \in \pi^{-1}(\pr(\Gamma) \cap G_0)$ and a.e.\ $g \in \Gtil$. We have thus proven that $\pr(\Gamma) < G$ is essentially cocycle superrigid in the sense of Definition~\ref{def.essential-cocycle-superrigid}.
\end{proof}

The following proposition clarifies the notion of essential cocycle superrigidity (Definition~\ref{def.essential-cocycle-superrigid}) in the two most important situations.

\begin{proposition}\label{prop.clarify-essential-cocycle-superrigid}
Let $\Gamma < G$ be a dense subgroup in a lcsc group.
\begin{enumlist}
\item If $G$ is totally disconnected, then $\Gamma < G$ is essentially cocycle superrigid with countable targets if and only if $\Gamma < G$ is \emph{virtually} cocycle superrigid with countable targets: for every $1$-cocycle $\om : \Gamma \times G \to \Lambda$ with values in a countable group $\Lambda$, there exists an open subgroup $G_0 < G$ such that the restriction of $\om$ to $(\Gamma \cap G_0) \times G_0$ is cohomologous to a group homomorphism $\Gamma \cap G_0 \to \Lambda$.
\item If $G$ is a connected Lie group with universal cover $\pi : \Gtil \to G$, then $\Gamma < G$ is essentially cocycle superrigid with countable targets if and only if for every $1$-cocycle $\om : \Gamma \times G \to \Lambda$ with values in a countable group $\Lambda$, the lifted $1$-cocycle $\omtil : \pi^{-1}(\Gamma) \times \Gtil \to \Lambda : \omtil(\gamma,g) = \om(\pi(\gamma),\pi(g))$ is cohomologous to a group homomorphism $\pi^{-1}(\Gamma) \to \Lambda$.
\end{enumlist}
\end{proposition}

\begin{proof}
Both in 1 and 2, one implication is obvious. So we assume that $\Gamma < G$ is essentially cocycle superrigid with countable targets and let $\om : \Gamma \times G \to \Lambda$ be a $1$-cocycle with values in a countable group $\Lambda$. We take an open subgroup $G_1 < G$ and a covering $\theta : G_2 \to G_1$ such that the lifted $1$-cocycle
$$\om_2 : \theta^{-1}(\Gamma \cap G_1) \times G_2 \to \Lambda : \om_2(\gamma,g) = \om(\theta(\gamma),\theta(g))$$
is cohomologous to a group homomorphism $\theta^{-1}(\Gamma \cap G_1) \to \Lambda$.

1.\ Since $\theta$ is a covering, we can take an open symmetric neighborhood $\cU$ of $e$ in $G_2$ such that $\theta|_{\cU}$ is a homeomorphism of $\cU$ onto $\theta(\cU)$ and such that the sets $(k \cU)_{k \in \Ker \theta}$ are disjoint. We take an open symmetric neighborhood $\cV$ of $e$ in $G_2$ with $\cV \cV \subset \cU$. Since $G$ is totally disconnected, we can choose an open subgroup $G_0 < G_1$ with $G_0 \subset \theta(\cV)$. Since $\theta^{-1}(\theta(\cV))$ is the disjoint union of the sets $(k \cV)_{k \in \Ker \theta}$, it follows that $G_3 := \cV \cap \theta^{-1}(G_0)$ is an open subgroup of $G_2$ and that the restriction of $\theta$ to $G_3$ is a homeomorphism of $G_3$ onto $G_0$. The restriction of $\om_2$ to $(G_3 \cap \theta^{-1}(\Gamma \cap G_1)) \times G_3$ is cohomologous to a group homomorphism. Hence, the restriction of $\om$ to $(\Gamma \cap G_0) \times G_0$ is cohomologous to a group homomorphism.

2.\ Since $G$ is connected, we have that $G_1 = G$. Since $\theta$ is a covering and $G$ is connected, the connected component of the identity in $G_2$ is an open subgroup $G_3 < G_2$ and $\theta(G_3) = G$. So, we may replace $G_2$ by $G_3$ and assume that $G_2$ is connected. Since $\pi : \Gtil \to G$ is the universal cover, we find a covering $\rho : \Gtil \to G_2$ such that $\theta \circ \rho = \pi$. Since $\om_2$ is cohomologous to a group homomorphism, a fortiori $\omtil$ is cohomologous to a group homomorphism.
\end{proof}

\begin{remark}\label{rem.link-earlier-work}
Assume that $\Gamma$ is a countable dense subgroup of a compact profinite group $G$ and consider the translation action $\Gamma \actson^\al G$. If $\Gamma$ is a lattice in a lcsc group $H$, then the induction $\Ind_\Gamma^H(\al)$ of $\al$ to an $H$-action is naturally isomorphic to the action $H \actson (G \times H)/\Gamma$, where $\Gamma$ is embedded diagonally in $G \times H$. Since $G$ is compact, this diagonal embedding $\Gamma < G \times H$ is still a lattice. Using Proposition \ref{prop.clarify-essential-cocycle-superrigid}, we then get that for compact profinite groups $G$, the first point of Theorem \ref{thm.main-cocycle-superrigid} is identical to \cite[Theorem A]{DIP19}.

When $\Gamma$ is a countable property~(T) group and $\Gamma < G$ is a dense embedding of $\Gamma$ into a compact profinite group $G$, then the diagonal embedding $\Gamma < G \times \Gamma$ trivially is a lattice. By Theorem \ref{thm.main-cocycle-superrigid} and Proposition \ref{prop.clarify-essential-cocycle-superrigid}, it follows that the translation action $\Gamma \actson G$ is virtually cocycle superrigid with countable targets, as was proven in \cite[Theorem B]{Ioa08}.

Finally consider the second point of Theorem \ref{thm.main-cocycle-superrigid} and assume that $G$ is simply connected. Since $H$ has property~(T) and $\Gamma < G \times H$ is a lattice, the action $H \actson \Gamma \backslash (G \times H)$ by right translation in the second variable has property~(T) in the sense of Zimmer. By \cite[Proposition 3.5]{PV08}, also the action $\Gamma \actson (G \times H)/H$ has property~(T). This means that the translation action $\Gamma \actson G$ has property~(T). It thus follows from \cite[Theorem B]{Ioa14} that $\Gamma \actson G$ is cocycle superrigid with countable target groups. So the second point of our Theorem \ref{thm.main-cocycle-superrigid} is mainly added for completeness and unification, since it is essentially a special case of \cite[Theorem~B]{Ioa14}.
\end{remark}

\begin{proposition}\label{prop.cocycle-superrigid-homogeneous-spaces}
Let $G$ be a connected Lie group with dense subgroup $\Gamma < G$ and universal cover $\pi : \Gtil \to G$. Let $P < G$ be a closed subgroup. Assume that $\Gamma < G$ is essentially cocycle superrigid with countable targets.

The action $\Gamma \actson G/P$ is cocycle superrigid with countable targets if and only if $\pi^{-1}(P)$ is connected.
\end{proposition}

Note that the action $\Gamma \actson G/P$ is not necessarily essentially free. Also note that $\pi^{-1}(P)$ is connected if and only if $P$ is connected and the natural homomorphism from the fundamental group of $P$ to the fundamental group of $G$ is surjective. Finally note that by taking $P = \{e\}$, it follows that the translation action $\Gamma \actson G$ can only be plainly cocycle superrigid in the case where $G$ is simply connected.

\begin{proof}
We fix a probability measure $\mu$ on $\Gtil$ that has the same null sets as the Haar measure of $\Gtil$. As in the proof of Theorem \ref{thm.main-cocycle-superrigid}, we consider the Polish group $\cF(\Gtil,\Lambda)$ of Borel maps $\Gtil \to \Lambda$, where we identify two maps that are equal a.e., with metric $d(\vphi,\psi) = \mu\bigl(\{g \in \Gtil \mid \vphi(g) \neq \psi(g)\}\bigr)$.

Write $\Ptil = \pi^{-1}(P)$ and $\Gammatil = \pi^{-1}(\Gamma)$. First assume that $\Ptil$ is connected. Let $\om : \Gamma \times G/P \to \Lambda$ be a $1$-cocycle with values in a countable group $\Lambda$. We prove that $\om$ is cohomologous to a group homomorphism $\Gamma \to \Lambda$. Define the $1$-cocycle
$$\omtil : \Gammatil \times \Gtil \to \Lambda : \omtil(\gamma,g) = \om(\pi(\gamma),\pi(g)P) \; .$$
By Proposition \ref{prop.clarify-essential-cocycle-superrigid}, $\omtil$ is cohomologous to a group homomorphism $\delta_1 : \Gammatil \to \Lambda$. We thus find $\vphi \in \cF(\Gtil,\Lambda)$ such that
\begin{equation}\label{eq.untwist}
\om(\pi(\gamma),\pi(g)P) = \vphi(\gamma g) \, \delta_1(\gamma) \, \vphi(g)^{-1}
\end{equation}
for all $\gamma \in \Gammatil$ and a.e.\ $g \in \Gtil$. For every $k \in \Ptil$, define $\vphi_k \in \cF(\Gtil,\Lambda)$ by $\vphi_k(g) = \vphi(gk)$. Since $\pi(gk)P = \pi(g)P$ for all $g \in \Gtil$, $k \in \Ptil$, we get that for all $\gamma \in \Gammatil$ and $k \in \Ptil$,
\begin{equation}\label{eq.my-eq}
\vphi(\gamma g) \, \delta_1(\gamma) \, \vphi(g)^{-1} = \vphi_k(\gamma g) \, \delta_1(\gamma) \, \vphi_k(g)^{-1}
\end{equation}
for a.e.\ $g \in \Gtil$. Whenever $k,k' \in \Ptil$ and $d(\vphi_k,\vphi_{k'}) < 1$, the set $\{g \in \Gtil \mid \vphi_k(g) = \vphi_{k'}(g)\}$ is nonnegligible and, by \eqref{eq.my-eq}, essentially invariant under translation by $\Gammatil$. Since $\Gammatil < \Gtil$ is dense, it follows that $\vphi_k = \vphi_{k'}$ a.e. Since the map $\Ptil \to \cF(\Gtil,\Lambda) : k \mapsto \vphi_k$ is continuous and $\Ptil$ is connected, the map is constant. So, for all $k \in \Ptil$, we have that $\vphi_k = \vphi$ a.e.

Since $\Gtil/\Ptil = G/P$, we find a Borel map $\psi : G/P \to \Lambda$ such that $\vphi(g) = \psi(\pi(g)P)$ for a.e.\ $g \in \Gtil$. By \eqref{eq.untwist}, we get that
$$\om(\pi(\gamma),x) = \psi(\pi(\gamma)\cdot x) \, \delta_1(\gamma) \, \psi(x)^{-1}$$
for all $\gamma \in \Gammatil$ and a.e.\ $x \in G/P$. When $\gamma \in \Gammatil \cap \Ker \pi$, it follows that $\psi(x) \, \delta_1(\gamma) \, \psi(x)^{-1} = e$ for a.e.\ $x \in G/P$, so that $\delta_1(\gamma) = e$. We thus find a group homomorphism $\delta : \Gamma \to \Lambda$ such that $\delta_1 = \delta \circ \pi$. We have proven that $\om$ is cohomologous to $\delta$.

Conversely assume that $\Ptil$ is not connected. We construct a countable group $\Lambda$ and a $1$-cocycle $\om : \Gamma \times G/P \to \Lambda$ that is not cohomologous to a group homomorphism $\Gamma \to \Lambda$. Denote by $Q < \Ptil$ the connected component of the identity and define the nontrivial countable group $\Lambda_1 = \Ptil / Q$ with quotient homomorphism $\delta_1 : \Ptil \to \Lambda_1$. Define $N_1 = \Ker \pi$ and recall that $N_1$ is a discrete subgroup of $\Gtil$ and that $N_1$ lies in the center of $\Gtil$. Also, $N_1 \subset \Gammatil$ and $N_1 \subset \Ptil$.

Define the central subgroup $N = \{(n,n) \mid n \in N_1\}$ of $\Gammatil \times \Ptil$. We consider the action of $(\Gammatil \times \Ptil)/N$ by left and right multiplication on $\Gtil$. We also consider the central subgroup $L = \{(n,\delta_1(n))\mid n \in N_1\}$ of $\Gammatil \times \Lambda_1$. Define the countable group $\Lambda = (\Gammatil \times \Lambda_1)/L$. Denote by $\delta : (\Gammatil \times \Ptil)/N \to \Lambda$ the natural quotient homomorphism.

We can then define the $1$-cocycle
$$\Om : (\Gammatil \times \Ptil)/N \times \Gtil \to \Lambda : \Om((g,h),x) = \delta(g,h) \; .$$
We identify $\Ptil$ with the closed normal subgroup $(N_1 \times \Ptil)/N$ of $(\Gammatil \times \Ptil)/N$. Under this identification, the action of $\Ptil$ on $\Gtil$ is given by right translation. As explained in Section \ref{sec.cross-section}, it follows that the restriction of $\Om$ to $\Ptil \times \Gtil$ is cohomologous to the trivial $1$-cocycle. The quotient of $(\Gammatil \times \Ptil)/N$ by the closed normal subgroup $\Ptil$ is naturally identified with $\Gamma$ and $\Gtil/\Ptil = G/P$. Therefore, we find a $1$-cocycle $\om : \Gamma \times G/P \to \Lambda$ such that $\Om$ is cohomologous to
$$\Om' : (\Gammatil \times \Ptil)/N \times \Gtil \to \Lambda : \Om'((g,h),x) = \om(\pi(g),\pi(x)P) \; .$$
We prove that $\om$ is not cohomologous to a group homomorphism $\delta_0 : \Gamma \to \Lambda$, when viewed as a $1$-cocycle for the action $\Gamma \actson G/P$.

Assume the contrary. Then, $\Om'$ and thus also $\Om$, is cohomologous to the group homomorphism $(\Gammatil \times \Ptil)/N \to \Lambda : (g,h)N \mapsto \delta_0(\pi(g))$. So, we find a Borel map $\vphi : \Gtil \to \Lambda$ such that
\begin{equation}\label{eq.my-formula}
\delta(g,h) = \Om((g,h),x) = \vphi(gxh^{-1}) \, \delta_0(\pi(g)) \, \vphi(x)^{-1}
\end{equation}
for all $(g,h) \in \Gammatil \times \Ptil$ and a.e.\ $x \in \Gtil$.

Replacing $\delta_0$ by $(\Ad k) \circ \delta_0$ and replacing $\vphi(\cdot)$ by $\vphi(\cdot) k^{-1}$, for some $k \in \Lambda$, we may assume that the Borel set $\cU = \{x \in \Gtil \mid \vphi(x) = e\}$ is nonnegligible. It follows from \eqref{eq.my-formula} that
$$\vphi(gx) = \delta(g,e) \, \delta_0(\pi(g))^{-1}$$
for all $g \in \Gammatil$ and a.e.\ $x \in \cU$. So, the set
$$T = \{g \in \Gtil \mid \mu(g \cU \cap \cU) = 0 \;\;\text{or}\;\; \mu(g \cU \vartriangle \cU)=0\}$$
contains $\Gammatil$. Since $T$ is closed and $\Gammatil \subset \Gtil$ is dense, we have that $T = \Gtil$. Defining the subgroup $G_1 < G$ by
$$G_1 = \{g \in \Gtil \mid \mu(g \cU \vartriangle \cU) = 0\} \; ,$$
we get that $\mu(g \cU \cap \cU) = 0$ for all $g \in \Gtil \setminus G_1$. Therefore,
$$G_1 = \{g \in \Gtil \mid \mu(\cU \setminus g \cU) < \mu(\cU)\} \; ,$$
and it follows that $G_1 \subset G$ is open. Since $\Gtil$ is connected, we conclude that $G_1 = \Gtil$. Since $\cU$ is, up to measure zero, invariant under translation by $g \in G_1$ and since $\cU$ is nonnegligible, we conclude that $\Gtil \setminus \cU$ has measure zero. We have proven that $\vphi(x) = e$ for a.e.\ $x \in \Gtil$.

It then follows from \eqref{eq.my-formula} that $\delta(g,h) = \delta_0(\pi(g))$ for all $(g,h) \in \Gammatil \times \Ptil$. It follows that $(e,h) \in \Ker \delta$ for all $h \in \Ptil$, which means that $\Ptil = \Ker \delta_1$, contradicting that $Q < \Ptil$ is a proper subgroup.
\end{proof}

We have seen in Proposition \ref{prop.clarify-essential-cocycle-superrigid} that for totally disconnected groups $G$, essential cocycle superrigidity of a dense subgroup $\Gamma < G$ is the same as virtual cocycle superrigidity. Moreover, this virtual cocycle superrigidity property automatically passes to homogeneous spaces, contrary to what happens in the connected case that we treated in Proposition \ref{prop.cocycle-superrigid-homogeneous-spaces}.

\begin{proposition}\label{prop.cocycle-superrigid-homogeneous-totally-disconnected}
Let $G$ be a totally disconnected lcsc group and let $\Gamma < G$ be a dense subgroup that is virtually cocycle superrigid with countable targets.

For every closed subgroup $P < G$, the action $\Gamma \actson G/P$ is virtually cocycle superrigid with countable targets: if $\om : \Gamma \times G/P \to \Lambda$ is a $1$-cocycle with values in a countable group $\Lambda$, there exists an open subgroup $G_0 < G$ such that, writing $\Gamma_0 = \Gamma \cap G_0$ and $P_0 = P \cap G_0$, the restricted $1$-cocycle $\om_0 : \Gamma_0 \times G_0/P_0 \to \Lambda$ is cohomologous to a group homomorphism.
\end{proposition}

\begin{proof}
Fix a closed subgroup $P < G$ and a $1$-cocycle $\om : \Gamma \times G/P \to \Lambda$. Since $\Gamma < G$ is virtually cocycle superrigid and $(\gamma,g) \mapsto \om(\gamma,gP)$ is a $1$-cocycle, we find an open subgroup $G_1 < G$, a Borel map $\vphi : G_1 \to \Lambda$ and a group homomorphism $\delta$ from $\Gamma_1 := \Gamma \cap G_1$ to $\Lambda$ such that
$$\om(\gamma,gP) = \vphi(\gamma g) \, \delta(\gamma) \, \vphi(g)^{-1}$$
for all $\gamma \in \Gamma_1$ and a.e.\ $g \in G_1$. For every $k \in P \cap G_1$, define $\vphi_k : G_1 \to \Lambda : \vphi_k(g) = \vphi(gk)$. Note as before that
$$\vphi_k(\gamma g) \, \delta(\gamma) \, \vphi_k(g)^{-1} = \vphi(\gamma g) \, \delta(\gamma) \, \vphi(g)^{-1} \; ,$$
so that for every $k \in P \cap G_1$, the set $Y_k := \{g \in G_1 \mid \vphi_k(g) = \vphi(g)\}$ is invariant under translation by $\Gamma_1$. Taking a neighborhood $\cU$ of the identity in $P \cap G_1$ such that $Y_k$ is nonnegligible for all $k \in \cU$, we get that $\vphi_k(g) = \vphi(g)$ for all $k \in \cU$ and a.e.\ $g \in G_1$. Since $G$ is totally disconnected, we can take an open subgroup $G_0 < G_1$ such that $P \cap G_0 \subset \cU$. We can now restrict $\vphi$ to a well defined map $G_0/(P \cap G_0) \to \Lambda$ and, writing $\Gamma_0 = \Gamma \cap G_0$ and $P_0 = P \cap G_0$, it follows that the restricted $1$-cocycle $\om_0 : \Gamma_0 \times G_0/P_0 \to \Lambda$ is cohomologous to the restriction of $\delta$ to $\Gamma_0$.
\end{proof}

We also record the following stability result for essential cocycle superrigidity of dense subgroups $\Gamma < G$ of a connected Lie group. The result follows immediately by applying \cite[Lemma 5.1]{Ioa14} to the universal cover. Note that the method of extending cocycle superrigidity results from groups $\Gamma_1$ to larger groups $\Gamma$ with $\Gamma_1$ being weakly normal (in some way) inside $\Gamma$ goes back to \cite[Proposition 3.6]{Pop05b}.

\begin{proposition}\label{prop.almost-normal-stability}
Let $G$ be a connected Lie group and let $\Gamma_1 < G$ be a dense subgroup that is essentially cocycle superrigid with countable target groups. Assume that $\Gamma < G$ is a countable subgroup with $\Gamma_1 < \Gamma$ and assume that $\Gamma$ is generated by elements $g \in \Gamma$ such that $g \Gamma_1 g^{-1} \cap \Gamma_1$ is dense in $G$. Then also $\Gamma < G$ is essentially cocycle superrigid with countable target groups.
\end{proposition}

We now develop further the result in Proposition \ref{prop.cocycle-superrigid-homogeneous-spaces} and consider general actions $\Gamma \actson G/P$ on homogeneous spaces when $G$ is a connected Lie group, $P < G$ is a closed subgroup and $\Gamma < G$ is a dense subgroup that is essentially cocycle superrigid. If $P$ is connected and the lift of $P$ in the universal cover $\Gtil$ remains connected, one deduces OE-superrigidity from Proposition \ref{prop.cocycle-superrigid-homogeneous-spaces} and Lemma \ref{lem.cocycle-OE-superrigid}. In general for arbitrary closed subgroups $P < G$, which could be a discrete subgroup, or a connected subgroup whose lift in $\Gtil$ is no longer connected, we can still completely describe all stably orbit equivalent actions in the following proposition. The proof of this result is essentially contained in \cite[Theorem 6.1]{Ioa14}.

\begin{proposition}\label{prop.OE-superrigid-homogeneous}
Let $G$ be a connected Lie group with finite center $\cZ(G)$ and let $P < G$ be any closed subgroup. Let $\Gamma < G$ be a countable dense subgroup that is essentially cocycle superrigid with countable targets. Assume that $\cZ(G) \subset \Gamma$ and that the action $\Gamma \actson G/P$ is essentially free. Let $\pi : \Gtil \to G$ be the universal cover.

Given an essentially free, nonsingular, ergodic action $\Lambda \actson (Y,\eta)$, the following are equivalent.
\begin{enumlist}
\item The actions $\Gamma \actson G/P$ and $\Lambda \actson Y$ are stably orbit equivalent.
\item The action $\Lambda \actson Y$ is induced from $\Lambda_0 \actson Y_0$ and there exists a normal subgroup $\Sigma_0 \lhd \Lambda_0$ such that $\Sigma_0 \actson Y_0$ admits a fundamental domain and $\Lambda_0/\Sigma_0 \actson Y_0/\Sigma_0$ is conjugate with $\Gamma/\cZ(G) \actson G/(P \cdot \cZ(G))$.
\item Writing $\Gammatil = \pi^{-1}(\Gamma)$ and $\Ptil = \pi^{-1}(P) \cdot \cZ(\Gtil)$, there exists an open normal subgroup $L \lhd \Ptil$ such that the action $\Lambda \actson Y$ is conjugate to an induction of the action
\begin{equation}\label{eq.my-action}
\frac{\Gammatil \times \Ptil}{(\{e\} \times L)N} \actson \Gtil / L \quad\text{by left-right translation, where $N = \{(g,g) \mid g \in \cZ(\Gtil)\}$.}
\end{equation}
\end{enumlist}
\end{proposition}
\begin{proof}
Since $\cZ(G)$ is finite and $\cZ(G) \subset \Gamma$, the actions $\Gamma \actson G/P$ and $\Gamma/\cZ(G) \actson G/(P \cdot \cZ(G))$ are stably orbit equivalent. The composition of $\pi : \Gtil \to G$ with the quotient map $G \to G/\cZ(G)$ is the universal cover of $G/\cZ(G)$. We may thus replace $G$ by $G/\cZ(G)$, $\Gamma$ by $\Gamma/\cZ(G)$ and $P$ by $(P \cdot \cZ(G))/\cZ(G)$, and assume throughout the proof that $G$ has trivial center. Then, $\cZ(\Gtil) = \Ker \pi$ and $\Ptil = \pi^{-1}(P)$.

3 $\Longrightarrow$ 2. Denote by $\Gamma_1 \actson X_1$ the action defined in \eqref{eq.my-action}. Write $N_1 = (\{e\} \times L)N$. Note that the action $\Gamma_1 \actson X_1$ is essentially free: if $(g,h)N_1 \in \Gamma_1$ fixes a nonnegligible part of $X_1$, applying the factor map $\Gtil/L \to \Gtil/\Ptil = G/P$ and using that $\Gamma \actson G/P$ is essentially free, it follows that $\pi(g) = e$. Thus, $g \in \Ker \pi = \cZ(\Gtil)$ and $(g,h)N_1 = (e,h')N_1$ for some $h' \in \Ptil$. The action of $\Ptil/L$ on $\Gtil/L$ by right translation is free, so that $h' \in L$. This means that $(e,h')N = N$.

We can view $\Ptil/L$ as a normal subgroup of $\Gamma_1$. The action of $\Ptil/L$ on $X_1 = \Gtil/L$ is given by right translation. Since $\Ptil$ is a closed subgroup of $\Gtil$, the action of $\Ptil/L$ on $X_1$ admits a fundamental domain and the quotient space is given by $\Gtil/\Ptil = G/P$. The quotient group $\Gamma_1/(\Ptil/L)$ is naturally isomorphic with $\Gammatil/\cZ(\Gtil) = \Gamma$. Therefore, any action $\Lambda \actson Y$ that is conjugate to an induction of $\Gamma_1 \actson X_1$ satisfies the conclusions of point~2. We have thus proven the implication 3 $\Longrightarrow$ 2.

2 $\Longrightarrow$ 1 holds trivially.

1 $\Longrightarrow$ 3. By Lemma \ref{lem.correspondence}, we can take a $1$-cocycle $\om : \Gamma \times G/P \to \Lambda$ such that the skew product equivalence relation $\cR_\om$ on $G/P \times \Lambda$ is of type~I and the action $\Lambda \actson Y$ is conjugate with $\Lambda \actson (G/P \times \Lambda)/\cR_\om$.

Write $H = (\Gammatil \times \Ptil)/N$ and consider the action $H \actson \Gtil$ by left-right translation. We view $\Gammatil$ and $\Ptil$ as commuting subgroups of $H$ with $\Gammatil \cap \Ptil = \cZ(\Gtil)$. Define the $1$-cocycle
$$\Om : H \times \Gtil \to \Lambda : \Om((g,h)N,x) = \om(\pi(g),\pi(x)P) \; .$$
Since $\Gamma < G$ is essentially cocycle superrigid with countable targets, the restriction of $\Om$ to $\Gammatil \times \Gtil$ is cohomologous to a group homomorphism. We thus find a Borel map $\vphi : \Gtil \to \Lambda$ and a group homomorphism $\delta : \Gammatil \to \Lambda$ such that the $1$-cocycle
\begin{equation}\label{eq.cocycle-Omnul}
\Om_0 : H \times \Gtil \to \Lambda : \Om_0((g,h)N,x) = \vphi(gxh^{-1}) \, \om(\pi(g),\pi(x)P) \, \vphi(x)^{-1}
\end{equation}
satisfies $\Om_0(g,x) = \delta(g)$ for all $g \in \Gammatil$ and a.e.\ $x \in \Gtil$. When $h \in \Ptil$, define $S \subset \Gtil \times \Gtil$ by
$$S := \{(x,y) \in \Gtil \times \Gtil \mid \Om_0(h,x) = \Om_0(h,y) \} \; .$$
As in the proof of \cite[Lemma 5.1]{Ioa14}, it first follows that $S$ is invariant under the transformation $(x,y) \mapsto (gx,gy)$ for all $g \in \Gammatil$ and then that $S$ has a complement of measure zero. So, for every $h \in \Ptil$, the map $x \mapsto \Om_0(h,x)$ is essentially constant. By the cocycle relation, for every $(g,h)N \in H$, the map $x \mapsto \Om_0((g,h)N,x)$ is essentially constant. Therefore, $\delta$ extends to a continuous group homomorphism $\delta : H \to \Lambda$ satisfying
\begin{equation}\label{eq.intermediate}
\vphi(gxh^{-1}) \, \om(\pi(g),\pi(x)P) \, \vphi(x)^{-1} = \delta((g,h)N) \quad\text{for all $(g,h)N \in H$ and a.e.\ $x \in \Gtil$.}
\end{equation}
Since $\delta$ is continuous and $\Lambda$ is discrete, we can define the open normal subgroup $L \lhd \Ptil$ by $L := \Ptil \cap \Ker \delta$. It follows from \eqref{eq.intermediate} that $\vphi(xh^{-1}) = \vphi(x)$ for all $h \in L$ and a.e.\ $x \in \Gtil$. So we may view $\vphi$ as a Borel map from $\Gtil/L$ to $\Lambda$.

Define $N_1 = (\{e\} \times L)N$ and $\Gamma_1 = (\Gammatil \times \Ptil)/N_1$. Since $L \subset \Ker \delta$, we can define the group homomorphism $\delta_1 : \Gamma_1 \to \Lambda$ such that $\delta_1((g,h)N_1) = \delta((g,h)N)$. Write $X_1 = \Gtil/L$. Consider the action $\Gamma_1 \actson X_1$ by left-right translation. Define the $1$-cocycle
$$\Om_1 : \Gamma_1 \times X_1 \to \Lambda : \Om_1((g,h)N_1,xL) = \om(\pi(g),\pi(x)P) \; .$$
We have proven that $\Om_1$ is cohomologous to the group homomorphism $\delta_1 : \Gamma_1 \to \Lambda$.

In the proof of 3 $\Longrightarrow$ 2, we have already seen that the action $\Gamma_1 \actson X_1$ is essentially free and that the restriction of this action to the normal subgroup $\Ptil/L$ of $\Gamma_1$ admits a fundamental domain $V \subset X_1$ with quotient space $G/P$. Let $W \subset G/P \times \Lambda$ be a fundamental domain for $\cR_\om$. Since the restriction of $\Om_1$ to $\Ptil/L \times X$ is identically equal to $e$, it follows that $(V \times \Lambda) \cap (\pi \times \id)^{-1}(W)$ is a fundamental domain for the skew product equivalence relation $\cR_{\Om_1}$ on $X_1 \times \Lambda$. So, $\cR_{\Om_1}$ is of type~I and there is a natural identification
$$(X_1 \times \Lambda)/\cR_{\Om_1} = (G/P \times \Lambda)/\cR_\om \; .$$
So, the action $\Lambda \actson Y$ is conjugate to the action $\Lambda \actson (X_1 \times \Lambda)/\cR_{\Om_1}$.

As in the proof of Lemma \ref{lem.cocycle-OE-superrigid}, since $\Om_1$ is cohomologous to the group homomorphism $\delta_1$, it follows that the action of $\Ker \delta_1$ on $X_1$ admits a fundamental domain and that $\Lambda \actson Y$ is conjugate to an induction of $\Gamma_1 / \Ker \delta_1 \actson X_1 / \Ker \delta_1$. We prove that $\Ker \delta_1 = \{e\}$.

Write $L_0 = L \cap \cZ(\Gtil)$ and view $\Gammatil/L_0$ as a subgroup of $\Gamma_1$. Denote $\Gamma_2 = \Ker \delta_1 \cap (\Gammatil/L_0)$. We claim that $\Gamma_2$ is a discrete subgroup of $\Gtil/L_0$. Assume the contrary and let $g_n \in \Gamma_2$ be a sequence of mutually distinct elements such that $g_n \to g \in \Gtil/L_0$. Let $D \subset X_1 = \Gtil/L$ be a fundamental domain for the action of $\Ker \delta_1$. Let $\mu$ be a probability measure on $X_1$ that has the same null sets as the Haar measure. For all $n \neq m$, we have that $\mu(g_n D \cap g_m D) = 0$. By first taking $m \to \infty$ and then $n \to \infty$, it follows that $0 = \mu(gD \cap gD) = \mu(gD)$, so that $\mu(D) = 0$, which is absurd. So, $\Gamma_2 < \Gtil/L_0$ is a discrete subgroup. Define the subgroup $\Gamma_3 < \Gtil$ with $L_0 \subset \Gamma_3$ and $\Gamma_3/L_0 = \Gamma_2$. Then also $\Gamma_3 < \Gtil$ is a discrete subgroup.

Next assume that $(g,k)N_1$ is an arbitrary element of $\Ker \delta_1$. Since $\Ker \delta_1$ is a normal subgroup of $\Gamma_1$, we get for every $h \in \Gammatil$ that $(hgh^{-1},k)N_1 \in \Ker \delta_1$ and thus, $(hgh^{-1}g^{-1},e)N_1 \in \Ker \delta_1$. Using the notation of the previous paragraph, this means that $hgh^{-1}g^{-1} \in \Gamma_3$ for all $h \in \Gammatil$. Since $\Gammatil < \Gtil$ is dense and $\Gamma_3 < \Gtil$ is closed, we conclude that $hgh^{-1}g^{-1} \in \Gamma_3$ for all $h \in \Gtil$. When $h \to e$, we get that $hgh^{-1}g^{-1} \to e$. Since $\Gamma_3 < \Gtil$ is discrete, it follows that $hgh^{-1}g^{-1} = e$ for all $h$ in a neighborhood of $e$. So, the centralizer of $g$ in $\Gtil$ is an open subgroup, and hence equal to $\Gtil$ by connectedness. We have proven that $g \in \cZ(\Gtil)$. Since this holds for an arbitrary element $(g,k)N_1$ of $\Ker \delta_1$ and viewing $\Ptil/L$ as a subgroup of $\Gamma_1$, we conclude that $\Ker \delta_1 \subset \Ptil/L$. By definition of $L$, we have that $\Ptil \cap \Ker \delta = L$ and thus, $(\Ptil/L) \cap \Ker \delta_1 = \{e\}$. So, $\delta_1$ is an injective group homomorphism and we have proven that $\Lambda \actson Y$ is conjugate to an induction of $\Gamma_1 \actson X_1$.
\end{proof}

\begin{corollary}\label{cor.OE-superrigid-homogeneous}
Make the same assumptions as in Proposition \ref{prop.OE-superrigid-homogeneous}. Assume moreover that $\pi^{-1}(P)$ is connected. Then an essentially free, nonsingular, ergodic action $\Lambda \actson (Y,\eta)$ is stably orbit equivalent with $\Gamma \actson G/P$ if and only if $\Lambda \actson Y$ is conjugate to an induction of $\Gamma/\Sigma \actson G/(P \cdot \Sigma)$ for a subgroup $\Sigma < \cZ(G)$.
\end{corollary}
\begin{proof}
Since $\pi^{-1}(P)$ is connected, every open subgroup of $\Ptil$ contains $\pi^{-1}(P)$ and is thus of the form $L = \pi^{-1}(P \cdot \Sigma)$ for a subgroup $\Sigma < \cZ(G)$. Then, the action in \eqref{eq.my-action} is isomorphic with $\Gamma/\Sigma \actson G/ (P \cdot \Sigma)$.
\end{proof}

\section{Examples; proofs of Theorem~\ref{thm.main-hyperbolic-plane} and Corollary \ref{cor.OE-superrigid-homogeneous-lattice}}\label{sec.examples}

Using linear algebraic groups and $S$-arithmetic lattices, one can provide numerous examples to which Theorem \ref{thm.main-cocycle-superrigid} applies. Rather than formulating an abstract result, we have chosen to discuss in detail two concrete cases: dense subgroups of $\SL(n,\R)$ and dense subgroups of $\SO(n,m,\R)$. These examples will then be used to prove OE-superrigidity and W$^*$-superrigidity for natural infinite measure preserving group actions on homogeneous spaces.

Given a field $K$ of characteristic zero, we denote by $B_{n,m}$ the bilinear form on $K^{n+m}$ given by
\begin{equation}\label{eq.bilinear}
B_{n,m}(x,y) = x_1 y_1 + \cdots + x_n y_n - x_{n+1} y_{n+1} - \cdots - x_{n+m} y_{n+m} \; .
\end{equation}
We denote by $\SO(n,m,K) < \GL(K^{n+m})$ the group of matrices preserving $B_{n,m}$ and having determinant $1$. Given any subring $\cO < K$, we consider the subgroup $\SO(n,m,\cO) < \SO(n,m,K)$ of matrices having all entries in $\cO$. In the case of $K = \R$, we denote by $\SO^+(n,m,\R)$ the connected component of the identity in $\SO(n,m,\R)$. When $m \geq 1$, this subgroup has index $2$. When $m=0$, we have $\SO^+(n,\R) = \SO(n,\R)$.

For every prime $p$ and positive integer $n$, we denote by $W_p(n)$ the maximal dimension of a totally isotropic subspace of $(\Q_p^n,B_{n,0})$. By convention, we put $W_p(0) = 0$. For our purposes, it only matters when $W_p(n)$ equals $0$ or $1$ or is at least $2$. We have $W_p(0) = W_p(1) = 0$ for all primes $p$ and, by e.g.\ \cite[Propositions 1 and 2]{Dia93}, we have
\begin{align*}
& W_p(2) = \begin{cases} 0 &\;\;\text{if $p \not\equiv 1 \mathbin{\operatorname{mod}} 4$,}\\ 1 &\;\;\text{if $p \equiv 1 \mathbin{\operatorname{mod}} 4$,}\end{cases} \quad\quad W_p(n) = \begin{cases}\bigl\lfloor \frac{n-3}{2}\bigr\rfloor &\;\;\text{if $p=2$,}\\ \bigl\lfloor \frac{n}{2}\bigr\rfloor &\;\;\text{if $p$ is odd,}\end{cases} \;\;\text{for $n=3,4,5$,} \\
& W_p(n) \geq 2 \;\;\text{for $n \geq 6$ and all primes $p$.}
\end{align*}
Whenever $n \geq m \geq 0$ with $n+m \geq 3$, we define
$$r_p(n,m) = \begin{cases} m+W_p(n-m) &\;\;\text{if $n+m \neq 4$,}\\
W_p(3) &\;\;\text{if $(n,m) = (4,0)$,}\\
1 &\;\;\text{if $(n,m) = (3,1)$ or $(2,2)$.}\end{cases}$$
Fix $n \geq m \geq 0$ with $n+m \geq 3$. If $n+m=3$ or $n + m \geq 5$, the group $\SO(n,m,\Q_p)$ is almost $\Q_p$-simple. By \cite[Proposition 2.14]{PR94}, its $\Q_p$-rank equals $r_p(n,m)$. When $n+m = 4$, exceptional isomorphisms may occur and $r_p(n,m)$ equals the minimal $\Q_p$-rank of the almost $\Q_p$-simple components of $\SO(n,m,\Q_p)$ (see the proof of Proposition \ref{prop.p-adic} for further details).

Theorem \ref{thm.main-cocycle-superrigid} then provides the following natural dense subgroups of $\SL(n,\R)$ and $\SO^+(n,m,\R)$ that are essentially cocycle superrigid.

\begin{proposition}\label{prop.p-adic}
Let $\cS$ be any set of prime numbers. Define the following subgroups.
\begin{enumlist}
\item For $n \geq 2$, consider the subgroup $\Gamma = \SL(n,\Z[\cS^{-1}])$ of $G = \SL(n,\R)$. Put $r = (n-1)|\cS|$.
\item For $n \geq m \geq 0$ with $n + m \geq 3$, consider the subgroup $\Gamma = \SO^+(n,m,\Z[\cS^{-1}])$ of $G = \SO^+(n,m,\R)$. Put $r = \sum_{p \in \cS} r_p(n,m)$.
\end{enumlist}
Then the following holds.
\begin{itemlist}
\item If $r \geq 2$, then $\Gamma < G$ is dense and the translation action $\Gamma \actson G$ is strongly ergodic and essentially cocycle superrigid with countable targets.
\item In the following cases with $r=1$, the translation action $\Gamma \actson G$ is strongly ergodic and its orbit equivalence relation is treeable. So, $\Gamma < G$ is not essentially cocycle superrigid.
\begin{itemlist}
\item $\SL(2,\Z[1/p]) < \SL(2,\R)$ for a single prime $p$.
\item $\SO(3,\Z[\cS^{-1}]) < \SO(3,\R)$ if $|\cS \setminus \{2\}| = 1$.
\item $\SO^+(3,1,\Z[1/p]) < \SO^+(3,1,\R)$ for a single prime $p \not\equiv 1 \mathbin{\operatorname{mod}} 4$.
\item $\SO(5,\Z[1/2]) < \SO(5,\R)$.
\end{itemlist}
\item If $r = 0$, then $\Gamma < G$ is a lattice and the translation action $\Gamma \actson G$ is of type~I.
\end{itemlist}
\end{proposition}

\begin{proof}
Throughout the proof, we write $\cO = \Z[\cS^{-1}] \subset \Q$.

1.\ Assume that $\cS$ is a finite set of primes, $\Gamma = \SL(n,\cO) < G = \SL(n,\R)$ and $r \geq 2$. Consider the diagonal embedding
$$\Delta : \SL(n,\cO) \to L = \SL(n,\R) \times \prod_{p \in \cS} \SL(n,\Q_p)$$
and write $\Gamma_1 = \Delta(\SL(n,\cO))$. Viewing $\SL_n$ as a linear algebraic group that is semisimple, connected, simply connected and absolutely almost simple, it follows from \cite[Theorem 5.7]{PR94} that $\Gamma_1 < L$ is a lattice. It follows from strong approximation (see e.g.\ \cite[Theorem 7.12]{PR94}) that for every $p \in \cS$, the action $\SL(n,\Q_p) \actson L/\Gamma_1$ is ergodic. In particular, if $\cS \neq \emptyset$, we get that $\SL(n,\cO)$ is dense in $\SL(n,\R)$. It follows from ``superstrong approximation'' (see \cite[Th\'{e}or\`{e}me 3.1]{Clo02}) that $\SL(n,\Q_p) \actson L/\Gamma_1$ is strongly ergodic for every $p \in \cS$.

Since $r \geq 2$, we have that $\cS \neq \emptyset$. If $|\cS| \geq 2$, the first point of Theorem \ref{thm.main-cocycle-superrigid} applies. If $\cS = \{p\}$, we have that $n \geq 3$ and the group $\SL(n,\Q_p)$ has property (T) (see e.g.\ \cite[Theorem 1.4.15]{BHV08}), so that the second point of Theorem \ref{thm.main-cocycle-superrigid} applies. Since the projection of $\Gamma_1$ to $\SL(n,\R)$ equals $\Gamma = \SL(n,\cO)$, we conclude that $\Gamma < G$ is essentially cocycle superrigid with countable targets.

2.\ Assume that $\Gamma = \SL(n,\cO) < G = \SL(n,\R)$ and $r \in \{0,1\}$. When $r=1$, we have seen in 1 that $\Gamma = \SL(2,\Z[1/p])$ sits as a lattice in $G \times G_1$ with $G_1 = \SL(2,\Q_p)$ and with $G_1 \actson (G \times G_1)/\Gamma$ being strongly ergodic. The group $G_1$ is strongly treeable (see Section \ref{sec.treeable} for background and references). The action of $\Gamma \times G_1$ on $G \times G_1$, where $\Gamma$ acts by left translation and $G_1$ acts by right translation in the second variable, is essentially free and ergodic. Restricting the action to $\Gamma$ or to $G_1$ gives an action of type~I. As explained in Section \ref{sec.cross-section}, it follows that the orbit equivalence relation $\cR$ of $\Gamma \actson G = (G \times G_1)/G_1$ is stably isomorphic with the cross section equivalence relation of $G_1 \actson (G \times G_1)/\Gamma$, which is strongly ergodic and treeable. So also $\cR$ is strongly ergodic and treeable.

When $r=0$, we are considering the lattice $\SL(n,\Z) < \SL(n,\R)$.

3.\ Assume that $\cS$ is a finite set of primes, $\Gamma = \SO^+(n,m,\cO) < G = \SO^+(n,m,\R)$ and $r \geq 2$. The proof is very similar to the proof of 1, but we have to be more careful since the linear algebraic group $\SO(B_{n,m})$ is not simply connected. We consider the semisimple, connected and simply connected linear algebraic group $\cG = \Spin(B_{n,m})$ (see e.g.\ \cite[Section 24i]{Mil17}). When $n+m \neq 4$, the algebraic group $\cG$ is absolutely almost simple. When $n+m = 4$, we have the following isomorphisms for every field $K$ of characteristic zero (see e.g.\ \cite[\S 9, Ex.\ 16]{Bou07} for a very concrete computation):
$$\Spin(B_{4,0},K) \cong \Spin(B_{3,0},K) \times \Spin(B_{3,0},K) \quad\text{and}\quad \Spin(B_{2,2},K) \cong \SL(2,K) \times \SL(2,K) \; .$$
When $-1$ is not a square in $K$, we have $\Spin(B_{3,1},K) \cong \SL(2,K[\sqrt{-1}])$. When $-1$ is a square in $K$, we obviously have
$$\Spin(B_{4,0},K) \cong \Spin(B_{3,1},K) \cong \Spin(B_{2,2},K) \cong \SL(2,K) \times \SL(2,K) \; .$$
So, for all values of $n \geq m \geq 0$ with $n+m \geq 3$, and for all primes $p$, we get that $\cG(\Q_p)$ is noncompact if and only if every almost $\Q_p$-simple factor of $\cG(\Q_p)$ is noncompact if and only if $r_p(n,m) \geq 1$.

Using the standard orthogonal basis for $B_{n,m}$, we view $\cG \subset \GL_{2^{n+m}}$ and define in this way $\Lambda = \cG(\cO)$ as a subgroup of $\cG(\Q)$. We again consider the diagonal embedding
$$\Delta: \Lambda \to S = \cG(\R) \times \prod_{p \in \cS} \cG(\Q_p) \; .$$
Write $\Lambda_1 = \Delta(\Lambda)$. By \cite[Theorem 5.7]{PR94}, $\Lambda_1 < S$ is a lattice. Define
$$\cS_0 = \{p \in \cS \mid \cG(\Q_p) \;\;\text{is noncompact}\;\} \; .$$
By the discussion above, an element $p \in \cS$ belongs to $\cS_0$ if and only if $r_p(n,m) \geq 1$.

It follows from \cite[Theorem 7.12]{PR94} that for every $p \in \cS_0$, the action $\cG(\Q_p) \actson S/\Lambda_1$ is ergodic. In particular, if $\cS_0 \neq \emptyset$, we get that $\Lambda < \cG(\Q)$ is dense. Writing $\cG$ as a product of almost $\Q$-simple components (in the exceptional case where $(n,m) = (4,0)$ or $(2,2)$ where $\cG$ is not $\Q$-simple), it follows from \cite[Th\'{e}or\`{e}me 3.1]{Clo02} that for every $p \in \cS_0$, the action $\cG(\Q_p) \actson S/\Lambda_1$ is strongly ergodic.

Now denote by $\cG_0$ the semisimple, connected linear algebraic group $\SO(B_{n,m})$. We view $\cG_0 \subset \GL_{n+m}$ and define in this way $\cG_0(\cO)$. For every field $K \in \{\Q, \R, \Q_p\}$, we use the natural homomorphism $\pi_K : \cG(K) \to \cG_0(K)$ with $\Ker \pi_K \cong \{\pm 1\}$. Note that $\pi_K(\cG(\cO)) \subset \cG_0(\cO)$. We denote $G_1 = \SO^+(n,m,\R)$, so that $G_1 = \pi_\R(\cG(\R))$. We also denote $G_p = \pi_{\Q_p}(\cG(\Q_p))$. By \cite[Chapter 10, Theorem 3.3]{Cas78}, we get that $G_p$ is an open, finite index subgroup of $\cG_0(\Q_p)$. We still denote by
$$\Delta : \cG_0(\Q) \to L = \cG_0(\R) \times \prod_{p \in \cS} \cG_0(\Q_p)$$
the diagonal embedding. Write $L_1 = G_1 \times \prod_{p \in \cS} G_p$ and $\Gamma_1 = \Delta(\cG_0(\cO)) \cap L_1$. Writing $\pi = \pi_\R \times \prod_{p \in \cS} \pi_{\Q_p}$, we get that $\pi : S \to L$ is a homomorphism that is continuous, open, with finite kernel and with the image $L_1$ being an open, finite index subgroup of $L$. Since $\Lambda_1 < S$ is a lattice, also $\pi(\Lambda_1) < L$ is a lattice. By construction, $\pi(\Lambda_1) \subset \Gamma_1 \subset \Delta(\cG_0(\cO))$. By \cite[Theorem 5.7]{PR94}, also $\Delta(\cG_0(\cO)) < L$ is a lattice. It follows that $\pi(\Lambda_1)$ is a finite index subgroup of $\Delta(\cG_0(\cO))$ and thus also a finite index subgroup of $\Gamma_1$.

For every $p \in \cS_0$, the action $\cG(\Q_p) \actson S/\Lambda_1$ is strongly ergodic. A fortiori, $G_p \actson L_1/\pi(\Lambda_1)$ and $G_p \actson L_1 / \Gamma_1$ are strongly ergodic.

When $n+m = 4$, by definition $r_p(n,m) \in \{0,1\}$. So, if $r_p(n,m) \geq 2$, we have that $\cG(\Q_p)$ is almost $\Q_p$-simple with $\Q_p$-rank at least $2$. In that case, the group $\cG(\Q_p)$ has property~(T) by \cite[Theorem 1.6.1]{BHV08}. The assumption $r \geq 2$ thus implies that either $\cS_0 = \{p\}$ with $G_p$ having property~(T), or $|\cS_0| \geq 2$. In both cases, we can apply Theorem \ref{thm.main-cocycle-superrigid} to the lattice $\Gamma_1 < L_1$. We conclude that the projection of $\Gamma_1$ to $G_1$ is essentially cocycle superrigid with countable targets. But this projection is equal to $\Gamma = \SO^+(n,m,\cO)$.

4.\ In each of the cases with $r=1$ and $G=\SO^+(n,m,\R)$ listed in the proposition, it follows from 3 and using the notation of 3 that we find a prime number $p$ and a lattice $\Gamma_1 < G \times G_p$, with $G_p < \SO(n,m,\Q_p)$ a finite index, open subgroup, such that $G_p \actson (G \times G_p)/\Gamma_1$ is strongly ergodic, such that the projection of $\Gamma_1$ to $G$ equals $\Gamma$ and such that $\SO(n,m,\Q_p)$ is almost $\Q_p$-simple of $\Q_p$-rank $1$. So, $G_p$ is again strongly treeable. As in 2, it follows that $\Gamma \actson G$ is strongly ergodic and has a treeable orbit equivalence relation. In the case $r=0$, we have seen in 3 that $\Gamma$ is a lattice in $G$.

5.\ Assume that $\cS$ is infinite, $\Gamma = \SL(n,\cO) < G = \SL(n,\R)$ or $\Gamma = \SO^+(n,m,\cO) < G = \SO^+(n,m,\R)$ as in the theorem, with $r \geq 2$. Take a finite subset $\cS_1 \subset \cS$ such that the value of $r$ for $\cS_1$ is at least $2$. Define $\Gamma_1$ to be either $\SL(n,\Z[\cS_1^{-1}])$ or $\SO^+(n,m,\Z[\cS_1^{-1}])$. As proven in 1 and 3 above, $\Gamma_1 < G$ is dense and essentially cocycle superrigid with countable targets. By e.g.\ the argument given in \cite[Proof of Theorem 11.2]{Ioa14}, we have that for every $g \in \Gamma$, the group $g \Gamma_1 g^{-1} \cap \Gamma_1$ is dense in $G$. So by Proposition \ref{prop.almost-normal-stability}, we get that $\Gamma < G$ is essentially cocycle superrigid with countable targets.
\end{proof}

Entirely similarly, we get the following examples of essentially cocycle superrigid subgroups of $\SL(n,\R)$ and $\SO^+(n,m,\R)$. We again take $n \geq m \geq 0$ with $n+m \geq 3$. To deal with the exceptional case $n+m=4$, we denote $R(n,m) = m$ when $n+m \neq 4$ and we denote
$$R(4,0) = 0 \;\; , \;\; R(3,1) = R(2,2) = 1 \; .$$
By definition, $R(n,m)$ is the minimal real rank of the almost simple components of $\SO(n,m,\R)$. We also write $I(3) = I(4) = 1$ and $I(n) = \lfloor n/2 \rfloor$ for all $n \geq 5$, so that $I(n+m)$ is the minimal real rank of the almost simple components of $\SO(n,m,\C) \cong \SO(n+m,\C)$.

\begin{proposition}\label{prop.real}
Let $K \subset \R$ be a real algebraic number field having $a \geq 1$ real embeddings and $b \geq 0$ complex embeddings. Denote by $\cO_K$ the ring of integers of $K$. Define the following subgroups.
\begin{enumlist}
\item For $n \geq 2$, consider the subgroup $\Gamma = \SL(n,\cO_K)$ of $G = \SL(n,\R)$. Put $r = (n-1)(a+b-1)$.
\item For $n \geq m \geq 0$ with $n + m \geq 3$, consider the subgroup $\Gamma = \SO^+(n,m,\cO_K)$ of $G = \SO^+(n,m,\R)$. Put $r = (a-1) \, R(n,m) + b \, I(n+m)$.
\end{enumlist}
Then the following holds.
\begin{itemlist}
\item If $r \geq 2$, then $\Gamma < G$ is dense and the translation action $\Gamma \actson G$ is strongly ergodic and essentially cocycle superrigid with countable targets.
\item If $K = \Q(\sqrt{N})$ where $N \geq 2$ is a square-free integer, the subgroups
$$\Gamma = \SL(2,\Z[\sqrt{N}]) < G = \SL(2,\R) \quad\text{and}\quad \Gamma = \SO^+(2,1,\Z[\sqrt{N}]) < G = \SO^+(2,1,\R) \; ,$$
are dense, the translation actions $\Gamma \actson G$ are strongly ergodic and their orbit equivalence relations are treeable. So, these $\Gamma < G$ are not essentially cocycle superrigid.
\item If $r = 0$, then $\Gamma < G$ is a lattice and the translation action $\Gamma \actson G$ is of type~I.
\end{itemlist}
\end{proposition}

Note that Proposition \ref{prop.real} does not exhaustively deals with all the cases $r = 1$. We discuss this in Remark \ref{rem.rank-one} below, where we show in particular that the orbit equivalence relation of $\Gamma \actson G$ is \emph{not} treeable in these remaining $r=1$ cases.

\begin{proof}
The argument is entirely similar to the proof of Proposition \ref{prop.p-adic}. We denote by $\al_1,\ldots,\al_a$ the real embeddings of $K$ and denote by $\be_1,\ldots,\be_b$ the complex embeddings of $K$. We take $\al_1$ to be the identity embedding. We again denote by $\cG$ one of the linear algebraic groups $\SL_n$ or $\Spin(B_{n,m})$ and consider the diagonal embedding
$$\Delta : \cG(\cO) \to \cG(\R) \times \prod_{i=2}^a \cG(\R) \times \prod_{j=1}^b \cG(\C) : \Delta(g) = (g,\al_2(g),\ldots,\al_a(g),\be_1(g),\ldots,\be_b(g)) \; .$$
We now make the following observations.
\begin{itemlist}
\item For every $n \geq 2$, the groups $\SL(n,\R)$ and $\SL(n,\C)$ are noncompact and they have property~(T) if $n \geq 3$, see \cite[Theorem 1.4.15]{BHV08}.
\item Let $n \geq m \geq 0$ with $n + m \geq 3$. When $m=0$, the group $\Spin(B_{n,m},\R)$ is compact. When $n=m=2$, we have that $\Spin(B_{2,2},\R) \cong \SL(2,\R) \times \SL(2,\R)$. In all other cases, we have that $\Spin(B_{n,m},\R)$ is noncompact, almost $\R$-simple, with real rank $m$. If this real rank is at least $2$, then $\Spin(B_{n,m},\R)$ has property~(T) by \cite[Theorem 1.6.1]{BHV08}.
\item Let $n \geq m \geq 0$ with $n + m \geq 3$. Then, $\Spin(B_{n,m},\C) \cong \Spin(B_{n+m,0},\C)$ is noncompact. If $n+m \neq 4$, the group $\Spin(B_{n,m},\C)$ is almost $\C$-simple and has real rank $\lfloor (n+m)/2 \rfloor$. If this real rank is at least $2$, then $\Spin(B_{n,m},\R)$ has property~(T) by \cite[Theorem 1.6.1]{BHV08}. Finally, if $n+m = 4$, we have $\Spin(B_{n,m},\C) \cong \SL(2,\C) \times \SL(2,\C)$.
\end{itemlist}
So when $r \geq 2$, we can repeat the argument of Proposition \ref{prop.p-adic} and get that $\Gamma < G$ is dense and essentially cocycle superrigid with countable targets.

Also in the case where $K = \Q(\sqrt{N})$ for some square-free integer and $\cG$ equals $\SL_2$ or $\Spin(B_{2,1})$, the argument is the same, since \cite[Theorem 6]{CGMTD21} says that the (locally isomorphic) groups $\SL(2,\R)$ and $\SO^+(2,1,\R)$ are strongly treeable.
\end{proof}

\begin{remark}\label{rem.rank-one}
In Proposition \ref{prop.real}, we do not exhaust all cases with $r=1$. In point~1, there remains the case where $K$ is a complex cubic field ($a=1$, $b=1$) and $n=2$. In point~2, there are several cases left:
\begin{enumlist}
\item $\SO(3,\cO_K) < \SO(3,\R)$ where $K$ has precisely one complex embedding.
\item $\SO^+(2,1,\cO_K) < \SO^+(2,1,\R)$ where $K$ is a complex cubic field.
\item $\SO^+(n,1,\cO_K) < \SO^+(n,1,\R)$ where $K = \Q(\sqrt{N})$ with $N \geq 2$ a square-free integer and $n \geq 3$.
\end{enumlist}
In all these cases, the argument in the proof of Proposition \ref{prop.p-adic} shows that $\Gamma < G$ is dense and that $\Gamma \actson G$ is strongly ergodic with the orbit equivalence relation $\cR$ being stably isomorphic to the cross section equivalence relation of an essentially free, strongly ergodic, pmp action of one of the following Lie groups $G_1$: $\SL(2,\C)$, $\SO(3,\C)$, $\SO^+(n,1,\R)$ with $n \geq 3$. All these Lie groups have real rank one. By \cite[Corollary 11]{PV13}, these groups $G_1$ have vanishing first $L^2$-Betti number. By \cite[Theorem A]{KPV13}, the same holds for the restriction of $\cR$ to any finite measure subset of $G$. This means that $\cR$ is not treeable in these cases. We do not expect that these pmp actions of $G_1$ satisfy a cocycle superrigidity theorem for arbitrary countable target groups (cf.\ \cite{BFS10}, where such a superrigidity theorem is proven under an extra integrability assumption on the cocycle).
\end{remark}

We are now ready to prove Theorem \ref{thm.main-hyperbolic-plane}. We view $\PSL(2,\R)$ as the group of orientation preserving isometries of the hyperbolic plane in its upper half plane model $\bH^2 = \{z \in \C \mid \Im z > 0\}$. For each $n \geq 2$, we also view $\SO^+(n,1,\R)$ as the group of orientation preserving isometries of hyperbolic $n$-space $\bH^n$~: using the notation \eqref{eq.bilinear}, we write
$$\bH^n = \{x \in \R^{n+1} \mid B_{n,1}(x,x) = -1 \;\;\text{and}\;\; x_{n+1} > 0 \}$$
so that $\SO^+(n,1,\R)$ precisely consists of the elements $g \in \SO(n,1,\R)$ that globally preserve $\bH^n$.

For every countable subring $\cO \subset \R$, the action of $\PSL(2,\cO)$ on $\bH^2$ and the action of $\SO^+(n,1,\cO)$ on $\bH^n$ is essentially free and preserves the natural infinite measure on $\bH^n$.

\begin{theorem}\label{thm.list-Wstar-superrigid}
Let $\cS$ be a finite set of prime numbers and denote $\cO_\cS = \Z[\cS^{-1}]$. Let $\Q \subset K \subset \R$ be a real algebraic number field and denote by $\cO_K$ the ring of algebraic integers.

The natural isometric actions of all the following groups on hyperbolic space are essentially free, strongly ergodic, infinite measure preserving, OE-superrigid and W$^*$-superrigid in the sense of Definition \ref{def.superrigid}.
\begin{enumlist}
\item $\PSL(2,\cO_\cS)$ and $\SO^+(2,1,\cO_\cS)$ if $|\cS| \geq 2$,
\item $\PSL(2,\cO_K)$, $\SO^+(2,1,\cO_K)$ and $\SO^+(3,1,\cO_K)$ if $[K:\Q] \geq 4$ or if $K$ is a totally real cubic field,
\item $\SO^+(3,1,\cO_\cS)$ if $|\cS| \geq 2$,
\item $\SO^+(n,1,\cO_K)$ if $n \geq 2$ and $K$ is totally real with $[K:\Q] \geq 3$.
\end{enumlist}
\end{theorem}
\begin{proof}
Denote these groups as $\Gamma$ sitting inside the Lie group $G$, which is either $\PSL(2,\R)$ or $\SO^+(n,1,\R)$. Defining $K = \PSO(2,\R)$ or $K = \SO(n,\R)$, respectively, we have $\bH^n = G/K$. As mentioned before the theorem, all these group actions $\Gamma \actson G/K$ are essentially free and preserve the natural infinite measure on $G/K$. Also note that the groups $K$ are connected and that the natural homomorphism from the fundamental group of $K$ to the fundamental group of $G$ is surjective. So, the inverse image of $K$ in the universal cover of $G$ is connected.

By Propositions \ref{prop.p-adic} and \ref{prop.real}, the subgroups $\Gamma < G$ are dense and essentially cocycle superrigid with countable targets, and the translation action $\Gamma \actson G$ is strongly ergodic. By Proposition \ref{prop.cocycle-superrigid-homogeneous-spaces}, the action $\Gamma \actson G/K$ is cocycle superrigid with countable targets. By Corollary \ref{cor.OE-superrigid-homogeneous}, the action $\Gamma \actson G/K$ is OE-superrigid.

To conclude that $\Gamma \actson G/K$ is W$^*$-superrigid, it then suffices to prove that the II$_\infty$ factor $L^\infty(G/K) \rtimes \Gamma$ has a unique Cartan subalgebra up to unitary conjugacy. From the proofs of Proposition \ref{prop.p-adic} and \ref{prop.real}, we know that $\Gamma$ can be viewed as a lattice in $G \times G_1 \times \cdots \times G_k$ where for every $i \in \{1,\ldots,k\}$, the group $G_i$ is an almost simple real or $p$-adic Lie group of rank $1$, or has a center $Z_i \cong \{\pm 1\}$ such that $G_i / Z_i$ is a direct product of two such groups. For coherence of notations, we thus find finite central subgroups $Z_i < G_i$, which may be trivial, such that $\cG = G_1/Z_1 \times \cdots \times G_k/Z_k$ is a product of almost simple real or $p$-adic Lie groups of rank $1$.

Consider the action $\Gamma \times \cG \actson^\al G/K \times \cG$, where $\Gamma$ acts by left translation and $\cG$ acts by right translation in the second variable. Since the action $\Gamma \actson G/K$ is essentially free and ergodic, also the action $\al$ is essentially free and ergodic. The restriction of the action $\al$ to $\Gamma$ or to $\cG$ is of type~I. As explained in Section \ref{sec.cross-section}, identifying $G/K = (G/K \times \cG)/\cG$ and writing $Z = \Gamma \backslash (G/K \times \cG)$, it follows that the von Neumann algebras $L^\infty(G/K) \rtimes \Gamma$ and $L^\infty(Z) \rtimes \cG$ are stably isomorphic. By \cite[Theorem A]{BDV17}, the von Neumann algebra $L^\infty(Z) \rtimes \cG$ has a unique Cartan subalgebra up to unitary conjugacy. So, $L^\infty(G/K)$ is, up to unitary conjugacy, the unique Cartan subalgebra of $L^\infty(G/K) \rtimes \Gamma$.
\end{proof}

Theorem \ref{thm.main-hyperbolic-plane} and Corollary \ref{cor.OE-superrigid-homogeneous-lattice} are now an immediate consequence of results proven so far.

\begin{proof}[{Proof of Theorem \ref{thm.main-hyperbolic-plane}}]
1.\ This is well known, with $\{z \in \C \mid \Im z > 0, |z| > 1 , |\Re z| < 1/2\}$ being a fundamental domain for the action $\PSL(2,\Z) \actson \bH^2$.

2.\ Write $G_1 = \PSL(2,\Q_p)$. As in the proof of Proposition \ref{prop.p-adic}, the lcsc group $G_1$ is strongly treeable and the orbit equivalence relation $\cR$ of the action $\Gamma \actson \bH^2$ is stably isomorphic with the cross section equivalence relation of the pmp action $G_1 \actson \Gamma \backslash (\bH^2 \times G_1)$ by right translation in the second variable. Thus, $\cR$ is treeable.

3.\ This is point~1 of Theorem \ref{thm.list-Wstar-superrigid}.
\end{proof}

\begin{proof}[Proof of Corollary \ref{cor.OE-superrigid-homogeneous-lattice}]
Define $\Gamma < G$ as in the formulation of the corollary. By Proposition \ref{prop.p-adic}, $\Gamma < G$ is dense and essentially cocycle superrigid with countable targets. The equivalence of 1 and 2 then follows from Proposition \ref{prop.OE-superrigid-homogeneous}. If $\cS$ is a finite set, we have seen in the proof of Proposition \ref{prop.p-adic} that $\Gamma$ is a lattice in
$$\PSL(2,\R) \times \prod_{p \in \cS} \PSL(2,\Q_p) \; .$$
It thus follows from \cite[Theorem 1.3]{PV12} that the crossed product II$_1$ factor $L^\infty(G/\Sigma) \rtimes \Gamma$ has a unique Cartan subalgebra up to unitary conjugacy. The equivalence of points 1 and 3 is then also proven.
\end{proof}

For completeness, we also record the following result. The proof is identical to the proof of Theorem \ref{thm.list-Wstar-superrigid}. The only reason why we cannot prove W$^*$-superrigidity is because we do not know if uniqueness of Cartan holds in this setting.

\begin{proposition}\label{prop.S-infinite}
Let $\cS$ be any set of prime numbers with $|\cS| \geq 2$. The natural isometric actions of $\PSL(2,\Z[\cS^{-1}])$ on $\bH^2$ and of $\SO^+(n,1,\Z[\cS^{-1}])$ on $\bH^n$ are essentially free, strongly ergodic, infinite measure preserving and OE-superrigid in the sense of Definition \ref{def.superrigid}.
\end{proposition}

Also the following result is an immediate consequence of what we have proven so far.

\begin{corollary}\label{cor.superrigid-Rn}
Let $\cO \subset \R$ be either the ring $\Z[\cS^{-1}]$ where $\cS$ is any nonempty set of prime numbers or the ring of algebraic integers of a real algebraic number field $\Q \subset K \subset \R$ with $2 \leq [K:\Q] < \infty$. Let $n \geq 2$ and put $\Gamma = \SL(n,\cO)$. Consider the linear action $\Gamma \actson \R^n$, which is essentially free, strongly ergodic and of type II$_\infty$. Let $\Lambda \actson (Y,\eta)$ be an arbitrary essentially free, ergodic, nonsingular action that is stably orbit equivalent with $\Gamma \actson \R^n$.
\begin{enumlist}
\item If $n \geq 3$, the action $\Gamma \actson \R^n$ is cocycle superrigid with arbitrary countable target groups.
\item If $n \geq 3$ and $n$ is odd, the action $\Gamma \actson \R^n$ is OE-superrigid in the sense of Definition \ref{def.superrigid}: $\Lambda \actson Y$ is conjugate to an induction of $\Gamma \actson \R^n$.
\item If $n \geq 3$ and $n$ is even, then $\Lambda \actson Y$ is conjugate to an induction of either $\Gamma \actson \R^n$ or $\Gamma/\{\pm 1\} \actson \R^n / \{\pm 1\}$.
\item If $n = 2$ and either $|\cS|\geq 2$ or $[K:\Q] \geq 4$ or $K$ is a totally real cubic field, then $\Lambda \actson Y$ is conjugate to an induction of one of the actions $\Gamma_n \actson^{\al_n} X_n$ defined in \eqref{eq.actions-al-n} with $n \in \{0\} \cup \N$.
\end{enumlist}
\end{corollary}

Let $\pi : \Gtil \to \SL(2,\R)$ be the universal cover. Let $P < \Gtil$ be the unique closed connected subgroup with $\pi(P) = \bigl(\begin{smallmatrix} 1 & * \\ 0 & 1 \end{smallmatrix}\bigr)$. Take $g_0 \in \Gtil$ such that $\cZ(\Gtil) = g_0^\Z$. For every $n \in \{0\} \cup \N$, define
\begin{equation}\label{eq.actions-al-n}
\Gamma_n \actson^{\al_n} X_n \quad\text{where}\quad \Gamma_n = \pi^{-1}(\SL(2,\cO))/g_0^{n \Z} \quad\text{and}\quad X_n = \Gtil/(P \cdot g_0^{n\Z}) \; .
\end{equation}
Note that $\Ker \pi = g_0^{2\Z}$ and that $\pi(g_0) = -1$. Therefore, the action $\Gamma \actson \R^2$ corresponds to $\al_2$, while the action $\al_1$ is isomorphic with $\Gamma/\{\pm 1\} \actson \R^2/\{\pm 1\}$.

Also note that in \cite[Theorems 1.3 and 1.5]{PV08}, a similar superrigidity result as Corollary \ref{cor.superrigid-Rn} is proven for the linear action of $\Gamma = \SL(n,\Z)$ on $\R^n$ if $n \geq 5$. For $n=2$, the action $\SL(2,\Z) \actson \R^2$ is amenable. It remains a challenging open problem to understand $\SL(n,\Z) \actson \R^n$ when $n=3,4$.

\begin{proof}[Proof of Corollary \ref{cor.superrigid-Rn}]
Write $G = \SL(n,\R)$. In all cases, it follows from Propositions \ref{prop.p-adic} and \ref{prop.real} that $\Gamma < G$ is essentially cocycle superrigid with countable targets.

Let $\pi : \Gtil \to G$ be the universal cover. The linear action of $G$ on $\R^n$ is essentially transitive and, up to measure zero, we identify $\R^n$ with $G / P_0$ where
$$P_0 = \begin{pmatrix} 1 & \R^{n-1} \\ 0 & \SL(n-1,\R) \end{pmatrix} \; .$$
When $n \geq 3$, the embedding of $\SO(n-1)$ into $\SO(n)$ induces a surjective homomorphism from the fundamental group of $\SO(n-1)$ to the fundamental group of $\SO(n)$. So, for $n \geq 3$, the closed subgroup $\pi^{-1}(P_0) < \Gtil$ is connected. Thus, point~1 follows from Proposition \ref{prop.cocycle-superrigid-homogeneous-spaces}, while points 2 and 3 follow from Corollary \ref{cor.OE-superrigid-homogeneous}.

Finally consider the case $n=2$. Using the notation introduced before \eqref{eq.actions-al-n}, we get that $\pi^{-1}(P_0) \cdot \cZ(\Gtil) = P \cdot \cZ(\Gtil)$. Since $P$ is connected, the open subgroups of $\pi^{-1}(P_0) \cdot \cZ(\Gtil)$ are precisely $P \cdot g_0^{n\Z}$ with $n \in \{0\} \cup \N$. So point~4 follows from Proposition \ref{prop.OE-superrigid-homogeneous}.
\end{proof}

\section{Superrigidity for wreath product equivalence relations}

We prove in this section an OE-superrigidity theorem for left-right wreath product equivalence relations, see Theorem \ref{thm.main-wreath}. Note that Theorem \ref{thm.stable-non-implement} stated in the introduction is an immediate corollary of Theorem \ref{thm.main-wreath} below.

Rather then considering left-right wreath products associated with the action $\Gamma \times \Gamma \actson \Gamma$ by left-right translation, it is more natural to prove this result for the appropriate class of generalized wreath products, associated with an action $\Gamma \actson I$.

So, we change notations and let $\Gamma$ be a countable group and $\Gamma \actson I$ an action on a countable set. Whenever $\Lambda$ is a countable group, the \emph{wreath product} $\Lambda \wr_I \Gamma$ is defined as the semidirect product $\Lambda^{(I)} \rtimes \Gamma$, where $\Gamma$ acts on $\Lambda^{(I)}$ by the automorphisms $(g^{-1} \cdot \lambda)_i = \lambda_{g \cdot i}$.

Similarly, when $\cR_0$ is a countable pmp equivalence relation on the standard probability space $(X_0,\mu_0)$, the wreath product $\cR_0 \wr_I \Gamma$ is defined as the countable pmp equivalence relation $\cR$ on $(X,\mu) = (X_0,\mu_0)^I$ given by $(x,y) \in \cR$ if and only if there exists a $g \in \Gamma$ and a finite subset $J \subset I$ such that $x_{g\cdot i} = y_i$ for all $i \in I \setminus J$ and $(x_{g\cdot i},y_i) \in \cR_0$ for all $i \in J$.

If $\Lambda \actson (X_0,\mu_0)$ is a pmp action, there is a natural action of $\Lambda \wr_I \Gamma$ on $(X_0,\mu_0)^I$ and its orbit equivalence relation is given by $\cR_0 \wr_I \Gamma$, where $\cR_0$ denotes the orbit equivalence relation of $\Lambda \actson X_0$.

We fix throughout a transitive action $\Gamma \actson I$ and make the following assumptions.
\begin{enumlist}
\item A rigidity assumption: the action should satisfy one of the following properties.
\begin{itemlist}
\item (Spectral gap rigidity) The group $\Gamma$ is generated by commuting subgroups $\Gamma_1$ and $\Gamma_2$ such that $I$ does not admit a $\Gamma_1$-invariant mean and $\Gamma_2 \actson I$ has infinite orbits.
\item ($w$-rigidity) The group $\Gamma$ admits a normal subgroup $\Gamma_1$ with the relative property (T) and with $\Gamma_1 \actson I$ having infinite orbits.
\end{itemlist}
\item For every $i \in I$, the stabilizer $\Stab i$ acts with infinite orbits on $I \setminus \{i\}$. For every $g \in \Gamma \setminus \{e\}$, there are infinitely many $i \in I$ with $g \cdot i \neq i$.
\item The group $\Gamma$ has infinite conjugacy classes.
\end{enumlist}
Note that the rigidity assumptions in point~1 were introduced in \cite{Pop05b,Pop06a}, where it was proven that under these hypotheses, the generalized Bernoulli action $\Gamma \actson X_0^I$ is cocycle superrigid with countable targets, which is the starting point of our proof of Theorem \ref{thm.main-wreath}. Assumption~2 above goes back to \cite{PV06} and makes this generalized Bernoulli action even more rigid since it allows in particular to completely describe all self orbit equivalences of $\Gamma \actson X_0^I$.

A wide and interesting class of examples satisfying all the hypotheses is given by the action $\Gamma_1 \times \Gamma_1 \actson \Gamma_1$ by left-right translation whenever $\Gamma_1$ is a nonamenable icc group. In that sense, left-right wreath products are covered by our results.

We refer to Section \ref{sec.prelim} for basic terminology and results on stable orbit equivalence.

\begin{theorem}\label{thm.main-wreath}
Let $\Gamma \actson I$ satisfy the assumptions above. Let $\cR_0$ be any countable nontrivial pmp equivalence relation on a standard probability space $(X_0,\mu_0)$.

Let $G \actson (Y,\eta)$ be an essentially free action of a countable group $G$ on a standard probability space $(Y,\eta)$. Let $t > 0$.

Then the following two statements are equivalent.

\begin{itemlist}
\item $(\cR_0 \wr_I \Gamma)^t$ is isomorphic with the orbit equivalence relation $\cR(G \actson Y)$.

\item The equivalence relation $\cR_0$ is implemented by an essentially free pmp action of a countable group $\Lambda_0 \actson (X_0,\mu_0)$. The action $G \actson Y$ is induced from $G_1 \actson Y_1$ and there exists a finite normal subgroup $\Sigma < G_1$ such that $G_1/\Sigma \actson Y_1/\Sigma$ is conjugate with the action $\Lambda_0 \wr_I \Gamma \actson X_0^I$. Moreover, $t = |\Sigma| \, [G:G_1]$.
\end{itemlist}
\end{theorem}

Theorem \ref{thm.main-wreath} should be compared with \cite[Theorem B]{KV15} in which all group measure space decompositions of $L(\cR_0 \wr_\Gamma (\Gamma \times \Gamma))$ were determined, under the hypothesis that $\cR_0$ is amenable and $\Gamma$ is a torsion free nonelementary hyperbolic group. So on the one hand, the hypotheses of \cite{KV15} are much more strict, and in particular do not allow for amplifications by $t > 0$, but on the other hand, \cite{KV15} describes all W$^*$-equivalent group actions, while we only describe stably orbit equivalent actions.

\begin{proof}[{Proof of Theorem \ref{thm.main-wreath}}]
It is trivial that the second statement implies the first, canonically. So we assume that the first statement holds. Write $\cR = \cR_0 \wr_I \Gamma$ and $(X,\mu) = (X_0,\mu_0)^I$. Since $\cR$ is ergodic, also $G \actson Y$ is ergodic. Since $\mu$ is $\cR$-invariant, we may and will assume that $\eta$ is a $G$-invariant probability measure. We are given a stable orbit equivalence between $\cR$ and $\cR(G \actson Y)$ with compression constant $t$. By Lemma \ref{lem.correspondence}, we find a cocycle $\om : \cR \to G$ such that the skew product equivalence relation $\cR_\om$ on $X \times G$ admits a fundamental domain $D \subset X \times G$ of measure $t$ and the resulting action of $G$ on $(X \times G)/\cR_\om$ by right translation in the second coordinate is isomorphic with $G \actson Y$.

For ease of reference, we number the steps in the proof.

1. By Popa's cocycle superrigidity of the (generalized) Bernoulli action $\Gamma \actson (X,\mu)$, proven in \cite[Theorem 0.1]{Pop05b} under the $w$-rigidity assumption and in \cite[Theorem 1.1]{Pop06a} under the spectral gap rigidity assumption, we may assume that $\om(g \cdot x,x) = \delta(g)$ for all $g \in \Gamma$ and a.e.\ $x \in X$, where $\delta : \Gamma \to G$ is a group homomorphism. Since $\cR_\om$ is of type~I, also its restriction to $X \times \{e\}$ is of type~I, so that $\Ker \delta$ must be a finite normal subgroup of $\Gamma$. Since we assumed that $\Gamma$ is icc, $\Ker \delta = \{e\}$.

2. For every $i \in I$, denote by $\cR_i \subset \cR$ the equivalence relation $\cR_0$ in coordinate $i$. More precisely: $(x,y) \in \cR_i$ if and only if $x_j = y_j$ for every $j \neq i$ and $(x_i,y_i) \in \cR_0$. We prove that $\om(x,y) = \om_i(x_i,y_i)$ for a.e.\ $(x,y) \in \cR_i$, where $\om_i : \cR_0 \to G$ is a cocycle.

Denote by $\cB_i$ the $\sigma$-algebra of Borel subsets of $(X,\mu)$ of the form $\{x \in X \mid x_i \in \cU\}$, where $\cU \subset X_0$ is a Borel set. Fix $i \in I$. Let $\vphi \in [[\cR_i]]$ with the domain and range of $\vphi$ belonging to $\cB_i$. Every $g \in \Stab i$ leaves the domain and the range of $\vphi$ globally invariant and $\vphi(g \cdot x) = g \cdot \vphi(x)$ for all $x \in D(\vphi)$. Define $F : D(\vphi) \to G : F(x) = \om(\vphi(x),x)$. It follows that $F(g \cdot x) = \delta(g) F(x) \delta(g)^{-1}$ for all $g \in \Stab i$ and a.e.\ $x \in D(\vphi)$. Since $\Stab i$ acts with infinite orbits on $I \setminus \{i\}$, the action of $\Stab i$ on $(X_0,\mu_0)^{I \setminus \{i\}}$ is weakly mixing (see e.g.\ \cite[Proposition 2.3]{PV06}) and it follows that $F$ only depends on the coordinate $x_i$. This precisely means that 2 is proven.

3. Denote by $\Lambda_i < G$ the subgroup generated by the essential range of $\om|_{\cR_i}$. Since $\om(x,y) = \om_i(x_i,y_i)$ when $(x,y) \in \cR_i$, the cocycle relation implies that the subgroups $\Lambda_i < G$ commute. Since $\om(g \cdot x , g \cdot y) = \delta(g) \om(x,y) \delta(g)^{-1}$, we get that $\Lambda_{g \cdot i} = \delta(g) \Lambda_i \delta(g)^{-1}$ for all $g \in \Gamma$ and $i \in I$, and we get that $\Lambda_i$ commutes with $\delta(\Stab i)$ for all $i \in I$. Fix $i_0 \in I$ and write $\Lambda = \Lambda_{i_0}$. Define $G_2 = \Lambda \wr_I \Gamma$ and denote, for every $i \in I$, by $\pi_i : \Lambda \to \Lambda^{(I)} \subset G_2$ the embedding as $i$'th direct summand. Since $\Gamma \actson I$ is transitive, the group homomorphism $\delta : \Gamma \to G$ can be uniquely extended to a group homomorphism $\delta : G_2 \to G$ satisfying $\delta(\pi_{i_0}(\lambda)) = \lambda$ for all $\lambda \in \Lambda$. We also find a unique $1$-cocycle $\om_2 : \cR \to G_2$ satisfying
$$\om_2(g \cdot x,y) = g \;\;\text{if $g \in \Gamma$, and}\quad \om_2(x,y) = \pi_{i_0}(\om_{i_0}(x_{i_0},y_{i_0})) \;\;\text{if $(x,y) \in \cR_{i_0}$.}$$
By construction, $\om = \delta \circ \om_2$. Also note that by construction,
\begin{multline}\label{eq.formula-om-2}
\om_2(x,y) = g a \quad\text{with $g \in \Gamma$ and $a \in \Lambda^{(I)}$ if and only if}\\
x_{g \cdot i} = y_i \;\;\text{for all but finitely many $i \in I$ and}\;\; a_i = \om_{i_0}(x_{g \cdot i},y_i) \;\;\text{for all $i \in I$.}
\end{multline}

4. Define the subequivalence relation $\cK \subset \cR$ as the kernel of $\om$. We prove that $\cK$ is trivial a.e. Denote by $\Delta = \{(x,x) \mid x \in X\}$ the trivial equivalence relation on $(X,\mu)$. We have to prove that $\cK = \Delta$, up to measure zero. Since $\cR_\om$ is of type~I, also its restriction to $X \times \{e\}$ is of type~I, meaning that the pmp equivalence relation $\cK$ is of type~I and thus has finite orbits a.e. Since $(g \cdot x, g \cdot y) \in \cK$ whenever $(x,y) \in \cK$ and $g \in \Gamma$, the function $x \mapsto |\cK \cdot x|$ is $\Gamma$-invariant and hence constant a.e. So, there exists an $N \in \N$ such that $|\cK \cdot x| = N$ for a.e.\ $x \in X$. It follows that $\cK \subset \cR$ is a subset of finite measure, so that $1_\cK \in L^2(\cR)$.

Defining the measure preserving transformations $\si_g : \cR \to \cR : \si_g(x,y) = (g \cdot x, g \cdot y)$, we have that $\si_g(\cK) = \cK$ for all $g \in \Gamma$. So, $1_\cK$ is invariant under the unitary representation $\pi$ of $\Gamma$ on $L^2(\cR)$ defined by $\pi_g(\xi) = \xi \circ \si_g^{-1}$. Denote by $\Delta_0 = \{(x,x) \mid x \in X_0\}$ the diagonal of $\cR_0$. From the definition of $\cR$ as a wreath product, it follows that $\pi$ is unitarily conjugate to the tensor product of the representations
$$\Ad : \Gamma \actson \ell^2(\Gamma) : (\Ad_g \xi)(h) = \xi(g^{-1} h g) \quad\text{and}\quad \pi_1 : \Gamma \actson \bigotimes_{i \in I} (L^2(\cR_0),1_{\Delta_0}) \; ,$$
where $\pi_1$ acts by permuting the tensor factors of the infinite tensor product of the Hilbert space $L^2(\cR_0)$ w.r.t.\ the canonical unit vector $1_{\Delta_0}$. Since $\Gamma$ is icc and since $\Gamma \actson I$ has infinite orbits, the only $\Gamma$-invariant vectors for $\Ad \ot \pi_1$ are the multiples of $\delta_e \ot 1$. It follows that the only $\Gamma$-invariant vectors in $L^2(\cR)$ are the multiples of $1_\Delta$. Since $1_\cK$ is $\Gamma$-invariant, we conclude that $\cK = \Delta$, up to measure zero.

5. In 3, we defined the group homomorphism $\delta : G_2 \to G$. Put $\Sigma_2 = \Ker \delta$. We prove that $\Sigma_2 \subset \Lambda^{(I)}$. Assume that $g a \in \Sigma_2$ with $g \in \Gamma \setminus \{e\}$ and $a \in \Lambda^{(I)}$. Take a finite subset $J \subset I$ such that $a \in \Lambda^{J}$. Since there are infinitely many $i \in I$ with $g \cdot i \neq i$, we can choose $i \in I \setminus J$ such that $g \cdot i \neq i$. Put $j = g \cdot i$. Since $\Sigma_2$ is a normal subgroup of $G_2$, for every $\lambda \in \Lambda$, we have that $\pi_i(\lambda) g a \pi_i(\lambda)^{-1} \in \Sigma_2$. Since $i \not\in J$, we have
$$\pi_i(\lambda) g a \pi_i(\lambda)^{-1} = \pi_i(\lambda) \, \pi_j(\lambda^{-1}) \, g a \; ,$$
so that $\pi_i(\lambda) \, \pi_j(\lambda^{-1}) \in \Sigma_2$ for all $\lambda \in \Lambda$. Denote $\om_0 = \om_{i_0} : \cR_0 \to \Lambda$. Since $\cR_0$ is a nontrivial equivalence relation, we can choose $\vphi \in [[\cR_0]]$ and $\lambda \in \Lambda$ such that the domain $D(\vphi)$ is nonnegligible, $\vphi(x) \neq x$ and $\om_0(\vphi(x),x) = \lambda$ for all $x \in D(\vphi)$. Define $\psi \in [[\cR]]$ with domain consisting of all $x \in X$ such that $x_i \in D(\vphi)$ and $x_j \in R(\vphi)$, and
$$\psi(x)_k = \begin{cases} \vphi(x_i) &\;\;\text{if $k=i$,}\\ \vphi^{-1}(x_j) &\;\;\text{if $k = j$,}\\ x_k &\;\;\text{if $k \not\in \{i,j\}$.}\end{cases}$$
By \eqref{eq.formula-om-2}, $\om_2(\psi(x),x) = \pi_i(\lambda) \, \pi_j(\lambda)^{-1}$ for all $x \in D(\psi)$. Thus, $\om(\psi(x),x) = \delta(\om_2(\psi(x),x)) = e$ for all $x \in D(\psi)$. By 4, we conclude that $\psi(x) = x$ for a.e.\ $x \in D(\psi)$, contradicting our choice of $\vphi$. This contradiction concludes the proof of 5.

6. Define $G_1 = \delta(G_2)$ and define the normal subgroup $H \lhd G_1$ by $H := \delta(\Lambda^{(I)})$. Since by 5, we have $\Sigma_2 \subset \Lambda^{(I)}$, we can view $G_1$ as the semidirect product of $H$ and $\delta(\Gamma)$. Also, by construction, $H$ is the subgroup of $G$ generated by the commuting subgroups $\Lambda_i \subset G$. Denote by $\cS \subset \cR$ the subequivalence relation $\cR_0^{(I)}$. So, $(x,y) \in \cS$ if and only if there exists a finite subset $J \subset I$ such that $x_i = y_i$ for all $i \not\in J$ and $(x_i,y_i) \in \cR_0$ for all $i \in J$. We prove that $\om^{-1}(H) = \cS$ up to measure zero.

Since $\om(\cR_i) \subset \Lambda_i$, we have $\om(\cS) \subset H$. Conversely, assume that $(x,y) \in \cR$ and $\om(x,y) \in H$. Then, $\om_2(x,y) \in \Lambda^{(I)} \, \Sigma_2 = \Lambda^{(I)}$, by 5. By the formula for $\om_2$ in \eqref{eq.formula-om-2}, this means that $(x,y) \in \cS$.

7. Since $\om(\cR) \subset G_1$ and since $\cR_\om$ admits a fundamental domain of finite measure, we have that $G_1 < G$ is of finite index. By Remark \ref{rem.induced}, we get that $G \actson Y$ is induced from $G_1 \actson Y_1$, where the action $G_1 \actson Y_1$ is isomorphic with the action $G_1 \actson (X \times G_1)/\cR_\om$ by right translation in the second variable.

8. We prove that the skew product equivalence relation $\cS_\om$ on $X \times H$ admits a fundamental domain of finite measure and that such a fundamental domain is also a fundamental domain for $\cR_\om$ on $X \times G_1$. Once this statement is proven, we will assume that the fundamental domain $D$ chosen at the beginning of the proof is a subset of $X \times H$ and is, at the same time, a fundamental domain for $\cS_\om$ on $X \times H$.

Since $\cR_\om$ admits a fundamental domain, a fortiori $\cS_\om$ admits a fundamental domain $F \subset X \times H$. We claim that the restriction of $\cR_\om$ to $F$ is trivial. Indeed, if $(x,h), (x',h') \in F$ and $(x,h) \sim_{\cR_\om} (x',h')$, we have $(x,x') \in \cR$ and $h = \om(x,x') h'$, so that $\om(x,x') \in H$. By 6, it follows that $(x,x') \in \cS$, so that $(x,h) \sim_{\cS_\om} (x',h')$. Since $F$ is a fundamental domain for $\cS_\om$, we conclude that $(x,h) = (x',h')$ and the claim is proven. To prove that $F$ is a fundamental domain for $\cR_\om$, it suffices to prove that a.e.\ point $(x,\zeta) \in X \times G_1$ is equivalent (in $\cR_\om$) with a point in $F$. Writing $\zeta = \delta(g) h$ with $g \in \Gamma$ and $h \in H$, we have that $(x,\zeta) \sim_{\cR_\om} (g^{-1} \cdot x, h)$. The latter is equivalent (in $\cS_\om$) with a point in $F$, so that the proof of 7 is complete.

9. We now prove in a rather indirect way that there exists a finite subgroup $\Lambda_2 \subset \Lambda$ such that $\Sigma_2$ is a finite index subgroup of $\Lambda_2^{(I)}$.

Denote by $\cS^\infty$ the amplification of $\cS$ to $X \times H$ with $(x,h) \sim_{\cS^\infty} (x',h')$ if and only if $(x,x') \in \cS$. Identifying $D \cong (X \times H)/\cS_\om \cong (X \times G_1)/\cR_\om \cong Y_1$, the resulting isomorphism $\Delta : D \to Y_1$ is an orbit equivalence between the restriction $\cS_1$ of $\cS^\infty$ to $D$ and $\cR(H \actson Y_1)$, cf.\ Lemma \ref{lem.correspondence}. Consider the action of $\Gamma$ on $X \times H$ given by $g \cdot (x,h) = (g \cdot x, \delta(g) h \delta(g)^{-1})$. Note that each transformation $(x,h) \mapsto g \cdot (x,h)$ is an automorphism of $\cS_\om$. We thus have a well defined transformation $\be_g$ of $Y_1 = (X \times H)/\cS_\om$. When identifying $(X \times H)/\cS_\om = (X \times G_1)/\cR_\om$, the transformation $\be_g$ is induced by right translation with $\delta(g)$. Thus, as a transformation of $Y_1$, we get that $\be_g$ is given by the action of $\delta(g) \in G_1$ on $Y_1$. In particular, $\be_g$ is an automorphism of the orbit equivalence relation $\cR(H \actson Y_1)$. Therefore, the transformations $\gamma_g : D \to D : \gamma_g = \Delta^{-1} \circ \be_g \circ \Delta$ define an action of $\Gamma$ by automorphisms of $\cS_1$.

We have defined the transformation $\be_g$ of $(X \times H)/\cS_\om$ as the quotient of the transformation $(x,h) \mapsto g \cdot (x,h)$. Denoting by $q : X \times H \to D$ the map that sends $(x,h) \in X \times H$ to the unique element in $D$ that is equivalent (in $\cS_\om$) with $(x,h)$, the map $(x,h) \mapsto q(x,h)$ is the concrete identification $D \cong (X \times H)/\cS_\om$. We thus find that $\gamma_g(x,h) = q(g \cdot (x,h))$. Writing $\gamma_g(x,h) = (x',h')$, we find that
$$(x',h') \sim_{\cS_\om} g \cdot (x,h) = (g \cdot x, \delta(g) h \delta(g)^{-1}) \; .$$
Denoting $\al_g : X \to X : x \mapsto g \cdot x$, we find in particular that $x' \sim_\cS g \cdot x$, so that $\gamma_g(x,h) \sim_{\cS^\infty} (\al_g(x),h)$ for all $g \in \Gamma$ and $(x,h) \in D$. We conclude that the transformation $\gamma_g \circ (\al_g^{-1} \times \id)$ with domain $(\al_g \times \id)(D)$ and range $D$ belongs to $[[\cS^\infty]]$. We denote this transformation as $\rho_g$. By construction, $\gamma_g = \rho_g \circ (\al_g \times \id)$. Since both $\gamma_g$ and $\al_g$ are group actions, we have
$$\rho_{gh} = \rho_g \circ \bigl((\al_g \times \id) \circ \rho_h \circ (\al_g^{-1} \ot \id)\bigr)$$
for all $g,h \in \Gamma$.

Recasting all this in a von Neumann algebraic language, we consider the von Neumann algebra $N = L(\cS)$ and we view $p = 1_D$ as a projection of finite trace in $L^\infty(X \times H) \subset N \ovt B(\ell^2(H))$.
The orbit equivalence $\Delta$ between $\cS_1$ and $\cR(H \actson Y_1)$ gives rise to a canonical $*$-isomorphism of the associated von Neumann algebras
$$\theta : L^\infty(Y_1) \rtimes H \to p(N \ovt B(\ell^2(H))) p \; .$$
The automorphism $\be_g$ of $\cR(H \actson Y_1)$ corresponds to the canonical automorphism of $L^\infty(Y_1) \rtimes H$ given by the action of $\delta(g) \in G_1$. Denoting by $(u_h)_{h \in H}$ the canonical unitary operators in $L^\infty(Y_1) \rtimes H$, we have that $\be_g(u_h) = u_{\delta(g)h\delta(g)^{-1}}$. Similarly, $\al_g$ corresponds to the canonical automorphism of $N = L(\cS)$ given by the action of $g \in \Gamma$. Writing $N_0 = L(\cR_0)$, we get that $N$ is the infinite tensor product $N_0^I$ with respect to the canonical trace on $L(\cR_0)$. From this point of view, the action $(\al_g)_{g \in \Gamma}$ is the noncommutative generalized Bernoulli action on $N_0^I$.

The elements $\rho_g \in [[\cS^\infty]]$ provide partial isometries $v_g \in N \ovt B(\ell^2(H))$ with $v_g^* v_g = (\al_g \ot \id)(p)$ and $v_g v_g^* = p$. Also, $v_{gh} = v_g (\al_g \ot \id)(v_h)$ and the automorphism $\gamma_g$ of $p(N \ovt B(\ell^2(H)))p$ is given by $\gamma_g = (\Ad v_g) \circ (\al_g \ot \id)$. We have $\theta \circ \be_g = \gamma_g \circ \theta$.

Also the noncommutative generalized Bernoulli action $\Gamma \actson N_0^I$ satisfies Popa's cocycle superrigidity: see \cite[Corollary 0.4]{Pop05a} for this result under the $w$-rigidity assumption and see \cite[Theorem 7.1]{VV14} for a proof under the spectral gap assumptions. So, there exists an integer $n \in \N$, an element $w \in N \ovt B(\C^n,H)$ and a unitary representation $V : \Gamma \to \cU(\C^n)$ such that $ww^* = p$, $w^* w = 1 \ot 1$ and $v_g = w (1 \ot V(g)) (\al_g \ot \id)(w^*)$ for all $g \in \Gamma$.

Since $w w^* = p$ and $w^* w = 1 \ot 1$, the map $\Ad w^*$ is a $*$-isomorphism from $p(N \ovt B(\ell^2(H))) p$ to $N \ot M_n(\C)$. We can thus define the $*$-isomorphism
$$\Psi : L^\infty(Y_1) \rtimes H \to N \ovt M_n(\C) : \Psi = (\Ad w^*) \circ \theta \; .$$
We get that $\Psi \circ \beta_g = (\al_g \ot \Ad V(g)) \circ \Psi$. Denote by $\Pi_i : N_0 \to N$ the embedding as the $i$'th tensor factor. Since $\Lambda_i$ commutes with $\delta(\Stab i)$, we find for every $\lambda \in \Lambda_i$ that $\Psi(u_\lambda)$ is invariant under $\al_g \ot \Ad V(g)$ for all $g \in \Stab i$. By weak mixing, we conclude that $\Psi(L(\Lambda_i)) \subset \Pi_i(N_0) \ot M_n(\C)$.

Define $\Lambda_2 = \{\lambda \in \Lambda \mid \Psi(u_\lambda) \in 1 \ot M_n(\C)\}$. Since $\Psi$ is a von Neumann algebra embedding, $\Lambda_2$ is a finite group. We prove that $\Sigma_2$ is a finite index subgroup of $\Lambda_2^{(I)}$. Take $a \in \Sigma_2$. Take a finite subset $J = \{i_1,\ldots,i_m\} \subset I$ such that $a \in \Lambda^{J}$. Write $b_k = \delta(\pi_{i_k}(a_{i_k})) \in \Lambda_{i_k}$, so that $e = \delta(a) = b_1 \cdots b_m$. Thus,
$$1 = \Psi(u_e) = \Psi(u_{b_1}) \cdots \Psi(u_{b_m}) \; .$$
By the previous paragraph, each $\Psi(u_{b_k})$ is a unitary in $\Pi_{i_k}(N_0) \ot M_n(\C)$ and the indices $i_k$ are distinct. If such a product of unitaries is equal to $1$, we must have that $\Psi(u_{b_k}) \in 1 \ot M_n(\C)$ for every $k$. Taking $g_k \in \Gamma$ such that $i_k = g_k \cdot i_0$, this means that $b_k \in \delta(g_k) \Lambda_2 \delta(g_k)^{-1}$ and thus $a_{i_k} \in \Lambda_2$ for all $k$. We have thus proven that $\Sigma_2 \subset \Lambda_2^{(I)}$.

Noting that $\delta(\Lambda_2^{(I)}) \subset \{ b \in H \mid \Psi(u_b) \in 1 \ot M_n(\C)\}$, it follows that $\delta(\Lambda_2^{(I)})$ is a finite group. This means that $\Sigma_2 = \Ker \delta$ has finite index in $\Lambda_2^{(I)}$.

10. We prove that $\Lambda_2$ is a normal subgroup of $\Lambda$. Denote by $p_i : \Lambda^{(I)} \to \Lambda$ the projection onto the $i$'th coordinate. Since $\pi_{g \cdot i}(a) = \pi_i(g^{-1} \cdot a)$, since the subgroup $\Sigma_2 \subset \Lambda^{(I)}$ is globally invariant under the action of $\Gamma$ and since $\Gamma \actson I$ is transitive, we have that all $p_i(\Sigma_2)$ are equal and we denote this subgroup of $\Lambda$ as $\Lambda_3$. Since $\Sigma_2$ is normal in $\Lambda^{(I)}$, also $\Lambda_3$ is normal in $\Lambda$. By construction, $\Sigma_2 \subset \Lambda_3^{(I)}$. By 9, we get that $\Lambda_3 \subset \Lambda_2$. If $\Lambda_3$ is a proper subgroup of $\Lambda_2$, then $\Lambda_3^{(I)}$ is of infinite index in $\Lambda_2^{(I)}$, contradicting the fact that $\Sigma_2$ has finite index in $\Lambda_2^{(I)}$. So, $\Lambda_3 = \Lambda_2$ and 10 is proven.

11. Define $\Sigma = \delta(\Lambda_2^{(I)})$. Then $\Sigma$ is a finite normal subgroup of $G_1$. Put $\Lambda_0 = \Lambda/\Lambda_2$. Denote by $\rho : G_1 \to G_1 / \Sigma$ the quotient homomorphism. Also consider the natural quotient homomorphism $\psi : G_2 = \Lambda \wr_I \Gamma \to \Lambda_0 \wr_I \Gamma$. By construction, there is a unique group isomorphism $\zeta : \Lambda_0 \wr_I \Gamma \to G_1/\Sigma$ such that $\zeta \circ \psi = \rho \circ \delta$.

Since $\Sigma$ is a finite group, the action of $\Sigma$ on $(X \times G_1)/\cR_\om$ admits a fundamental domain $D_1$, that we can view as an $\cR_\om$-invariant subset $D_1 \subset X \times G_1$. Since $D \subset X \times G_1$ is a fundamental domain for $\cR_\om$, it follows that the image of $D \cap D_1$ in $X \times G_1/\Sigma$ is a fundamental domain of finite measure for $\cR_{\rho \circ \om}$.

Recall that we denoted $\om_0 = \om_{i_0}$. Denoting by $\psi_0 : \Lambda \to \Lambda_0$ the quotient homomorphism, we define the cocycles $\Om_0 : \cR_0 \to \Lambda_0 : \Om_0 = \psi_0 \circ \om_0$ and
$$\Om : \cR \to \Lambda_0 \wr_I \Gamma : \Om = \psi \circ \om_2 = \zeta^{-1} \circ \rho \circ \om \; .$$
Write $G_0 = \Lambda_0 \wr_I \Gamma$. We have seen above that $\cR_{\rho \circ \om}$ admits a fundamental domain of finite measure. So also the skew product equivalence relation $\cR_\Om$ on $X \times G_0$ admits a fundamental domain of finite measure. By construction, the action $G_1/\Sigma \actson Y_1/\Sigma$ is conjugate, under $\zeta$, with the action of $G_0$ on $(X \times G_0)/\cR_\Om$ by right translation in the second variable.

Since $\Om = \psi \circ \om_2$, it follows from \eqref{eq.formula-om-2} that $\Om(x,y) = g a$ with $g \in \Gamma$ and $a \in \Lambda_0^{(I)}$ iff $x_{g \cdot i} = y_i$ for all but finitely many $i \in I$ and $a_i = \Om_0(x_{g \cdot i} , y_i)$ for all $i \in I$. Since $\cR_\Om$ is of type~I, also its restriction to $X \times \{e\}$ is of type~I, meaning that the kernel of $\Om$ is of type~I. But the kernel of $\Om$ is the pmp equivalence relation on $(X,\mu)$ that is given as the infinite product of the kernel of $\Om_0$. This can only be of type~I if it is trivial.

So, $\Om_0$ has trivial kernel. We claim that the skew product $(\cR_0)_{\Om_0}$ on $X_0 \times \Lambda_0$ is of type~I. Since $\cR_\Om$ is of type~I, also the subequivalence relation $(\cR_{i_0})_\Om$ and its restriction to $X \times \pi_{i_0}(\Lambda_0)$ are of type~I. Whenever $\cU \subset X \times \pi_{i_0}(\Lambda_0)$ is a Borel set such that the restriction of $(\cR_{i_0})_\Om$ to $\cU$ is trivial, we define for every $x \in X_0^{I \setminus \{i_0\}}$ the Borel set $\cU_x \subset X_0 \times \Lambda_0$ such that $(x',\lambda) \in \cU_x$ iff $(y,\pi_{i_0}(\lambda)) \in \cU$ with $y \in X$ given by $y_{i_0} = x'$ and $y_i = x_i$ if $i \neq i_0$. It follows that the restriction of $(\cR_0)_{\Om_0}$ to $\cU_x$ is trivial. Since this holds for all choices of $\cU$ and $x$, the claim is proven.

Let $Z_0 \subset X_0 \times \Lambda_0$ be a fundamental domain for $(\cR_0)_{\Om_0}$. Since $\Om_0$ has trivial kernel, we may assume that $X_0 \times \{e\} \subset Z_0$. Defining $Z \subset X \times \Lambda_0^{(I)}$ as the set of $(x,a)$ such that $(x_i,a_i) \in Z_0$ for all $i \in I$, the explicit form of $\Om$ implies that $Z$ is a fundamental domain of $\cR_\Om$. So, $Z$ must have finite measure. If $Z_0 \setminus (X_0 \times \{e\})$ has positive measure, then $Z$ has infinite measure. We have thus proven that $X_0 \times \{e\}$ is a fundamental domain of $(\cR_0)_{\Om_0}$. This precisely means that we have found an essentially free pmp action $\Lambda_0 \actson X_0$ such that $\cR_0 = \cR(\Lambda_0 \actson X_0)$ and such that $\Om_0(\lambda \cdot x , x) = \lambda$ for all $x \in X_0$, $\lambda \in \Lambda_0$.

Then, by construction, the action of $G_0$ on $(X \times G_0)/\cR_\Om$ is isomorphic to the natural action of $G_0 = \Lambda_0 \wr_I \Gamma$ on $X_0^I$. This concludes the proof of the theorem.
\end{proof}

Theorem \ref{thm.main-wreath} describes exactly which essentially free group actions $G \actson (Y,\eta)$ have an orbit equivalence relation that is isomorphic to $(\cR_0 \wr_I \Gamma)^t$ with $t > 0$ finite. In Remark \ref{rem.full-description} below, we make this description even more explicit. In the next proposition, we provide a description of which infinite measure preserving actions $G \actson (Y,\eta)$ are orbit equivalent with $(\cR_0 \wr_I \Gamma)^\infty$. This description is necessarily more subtle and we need the following notation.

Let $(Z_0,\mu_0)$ be a standard $\sigma$-finite measure space and $X_0 \subset Z_0$ a Borel set with $\mu_0(X_0)=1$. Let $I$ be a countable set and consider the \emph{restricted direct product}
$$Z = \resprod_{i \in I} (Z_0,X_0) = \{x \in Z_0^I \mid x_i \in X_0 \;\;\text{for all but finitely many $i \in I$}\;\} \; .$$
Then the product measure $\mu = \mu_0^I$ is a well defined $\sigma$-finite measure on $Z$.

Whenever $\Gamma \actson I$ is an action of a countable group $\Gamma$ on the countable set $I$ and whenever $\Lambda \actson (Z_0,\mu_0)$ is a measure preserving action, the natural action of $\Lambda \wr_I \Gamma$ on $(Z,\mu)$ is well defined and measure preserving.

\begin{proposition}\label{prop.rigidity-infinite-measure}
Assume that $\Gamma \actson I$ satisfies the assumptions above. Let $\cR_0$ be any countable nontrivial pmp equivalence relation on a standard probability space $(X_0,\mu_0)$.

Let $G \actson (Y,\eta)$ be an essentially free action of a countable group $G$ on a standard $\sigma$-finite measure space $(Y,\eta)$ with $\eta(Y) = +\infty$.

Then $(\cR_0 \wr_I \Gamma)^\infty$ is isomorphic with the orbit equivalence relation $\cR(G \actson Y)$ if and only if the following holds: there exists an essentially free, measure preserving action of a countable group $\Lambda$ on a standard $\sigma$-finite measure space $(Z_0,\mu_0)$, an embedding of $X_0$ as a Borel set of measure $1$ in $Z_0$, a subgroup $\Lambda_2$ of the center of $\Lambda$ and a globally $\Gamma$-invariant subgroup $\Sigma_2 \subset \Lambda_2^{(I)}$ such that
\begin{itemlist}
\item the restriction of the orbit equivalence relation $\cR(\Lambda \actson Z_0)$ to $X_0$ equals $\cR_0$ and $\Lambda \cdot X_0 = Z_0$~;
\item considering the natural action of $G_0 = \Lambda \wr_I \Gamma$ on $(Z,\mu) = \prod'_I (Z_0,X_0)$, we have that $\Sigma_2 \actson (Z,\mu)$ admits a fundamental domain of measure $t \in [1,+\infty]$ containing $X_0^I$~;
\item the action $G \actson Y$ is induced from an action $G_1 \actson Y_1$ that is conjugate with the action $G_0/\Sigma_2 \actson Z/\Sigma_2$~;
\item we have $t \, [G:G_1] = +\infty$.
\end{itemlist}
\end{proposition}

\begin{proof}
In the constructive direction, we first prove that the action $G_0 \actson (Z,\mu)$ is ergodic. Since $\Lambda \cdot X_0 = Z_0$, we have that $\Lambda^{(I)} \cdot X_0^I = Z$. It thus suffices to prove that the restriction of $\cR(G_0 \actson Z)$ to $X_0^I$ is ergodic. But this restriction is, by construction, equal to $\cR_0 \wr_I \Gamma$, which is ergodic because the action of $\Gamma$ on $X_0^I$ is ergodic. Since $X_0^I$ has measure $1$ and since $\Sigma_2 \actson (Z,\mu)$ admits a fundamental domain of measure $t$, we have
$$(\cR_0 \wr_I \Gamma)^t \cong \cR(G_0/\Sigma_2 \actson Z/\Sigma_2) \cong \cR(G_1 \actson Y_1) \; .$$
It follows that $(\cR_0 \wr_I \Gamma)^{t \, [G:G_1]} \cong \cR(G \actson Y)$ and thus, $(\cR_0 \wr_I \Gamma)^\infty \cong \cR(G \actson Y)$.

Conversely, write $(X,\mu) = (X_0,\mu_0)^I$ and $\cR = \cR_0 \wr_I \Gamma$ and assume that $\cR^\infty \cong \cR(G \actson Y)$. By Lemma \ref{lem.correspondence}, we find a cocycle $\om : \cR \to G$ such that the skew product equivalence relation $\cR_\om$ on $X \times G$ has a fundamental domain of infinite measure and $G \actson Y$ is isomorphic with the action of $G$ on $(X \times G)/\cR_\om$ by right translation in the second variable.

Steps 1 to 5 in the proof of Theorem \ref{thm.main-wreath} do not rely on the finiteness of the measure of the fundamental domain of $\cR_\om$. We thus find a countable group $\Lambda$, a cocycle $\om_0 : \cR_0 \to \Lambda$ and a group homomorphism $\delta : \Lambda \wr_I \Gamma \to G$ such that, writing $G_0 = \Lambda \wr_I \Gamma$, we may assume that $\om = \delta \circ \om_2$ where $\om_2 : \cR_0 \wr_I \Gamma \to G_0$ is the canonical cocycle given by $\om_2(x,y) = g a$ iff $x_{g \cdot i} = y_i$ for all but finitely many $i \in I$ and $a_i = \om_0(x_{g \cdot i},y_i)$ for all $i \in I$. Moreover, the kernel of $\om$ is trivial. Writing $\Sigma_2 = \Ker \delta$, by step~5, we also have that $\Sigma_2 \subset \Lambda^{(I)}$ and that $\Sigma_2$ intersects $\pi_i(\Lambda)$ trivially for every $i \in I$, by construction.

Write $G_1 = \delta(G_0)$. Since $\om(\cR) \subset G_1$, by Remark \ref{rem.induced}, the action $G \actson Y$ is induced from $G_1 \actson Y_1$ in such a way that $G_1 \actson Y_1$ is conjugate with the action of $G_1$ on $(X \times G_1)/\cR_\om$. Let $D_1 \subset X \times G_1$ be a fundamental domain for $\cR_\om$. Put $D_2 = (\id \ot \delta)^{-1}(D_1)$. Then $D_2 \subset X \times G_0$ is a fundamental domain for $\cR_{\om_2}$ and we consider the action $G_0 \actson Z = (X \times G_0)/\cR_{\om_2}$. Choosing a subset $J \subset G_0$ such that the restriction $\delta|_J : J \to G_1$ is a bijection, the image of $D_2 \cap (X \times J)$ in $Z$ is a fundamental domain for the action of $\Sigma_2 = \Ker \delta$ and, by construction, $G_0/\Sigma_2 \actson Z/\Sigma_2$ is conjugate with $G_1 \actson Y_1$.

Since $\cR_{\om_2}$ admits a fundamental domain, repeating the argument at the end of the proof of Theorem \ref{thm.main-wreath} yields that also the skew product equivalence relation $(\cR_0)_{\om_0}$ on $X_0 \times \Lambda$ admits a fundamental domain $Z_0 \subset X_0 \times \Lambda$. Since $\om$ has trivial kernel, also $\om_0$ has trivial kernel and we may choose the fundamental domain $Z_0$ such that $X_0 \times \{e\} \subset Z_0$. We identify $Z_0 = (X_0 \times \Lambda)/\cR_{\om_0}$ and consider the action $\Lambda \actson Z_0$. We view $X_0$ as a subset of $Z_0$. As discussed before the proposition, we have an associated action of $G_0 = \Lambda \wr_I \Gamma$ on the restricted product $Z_1 = \prod'_{I} (Z_0,X_0)$.

Define $Z_2 \subset X \times \Lambda^{(I)}$ by
$$Z_2 = \{(x,a) \in X \times \Lambda^{(I)} \mid (x_i,a_i) \in Z_0 \;\;\text{for all $i \in I$}\;\} \; .$$
Then $Z_2$ is a fundamental domain for $\cR_{\om_2}$. Above we could thus have made the choice $Z = Z_2$ and consider the action $G_0 \actson Z_2 = (X \times G_0)/\cR_{\om_2}$. The tautological map $\Psi : Z_2 \to Z_1 : \Psi(x,a)_i = (x_i,a_i)$ is an isomorphism of measure spaces that conjugates the actions $G_0 \actson Z_2$ and $G_0 \actson Z_1$.

Since $\Sigma_2$ is a normal subgroup of $G_0$, as a subgroup of $\Lambda^{(I)}$, we have that $\Sigma_2$ is globally $\Gamma$-invariant. Denoting by $p_i : \Lambda^{(I)} \to \Lambda$ the projection onto the $i$'th direct summand, it follows that the subgroup $p_i(\Sigma_2) \subset \Lambda$ is independent of $i$. We denote this subgroup as $\Lambda_2$. Let $a \in \Sigma_2$ and $i \in I$. Since $\Sigma_2 \subset \Lambda^{(I)}$ is a normal subgroup, we have $\pi_i(b) a \pi_i(b)^{-1} \in \Sigma_2$ for all $b \in \Lambda$ and thus
$$\pi_i(b a_i b^{-1} a_i^{-1}) = \pi_i(b) a \pi_i(b)^{-1} a^{-1} \in \Sigma_2 \; .$$
Since $\Sigma_2 \cap \pi_i(\Lambda) = \{e\}$, we conclude that $a_i$ commutes with every $b \in \Lambda$. So, $\Lambda_2$ is a subgroup of the center of $\Lambda$. By construction, $\Sigma_2 \subset \Lambda_2^{(I)}$.
\end{proof}

\begin{remark}\label{rem.full-description}
Analyzing the proof of Theorem \ref{thm.main-wreath} and using the last paragraph of the proof of Proposition \ref{prop.rigidity-infinite-measure}, one may give the following more precise description of the action $G_1 \actson Y_1$ that appears in the second point of Theorem \ref{thm.main-wreath}.

So, under the assumptions of Theorem \ref{thm.main-wreath}, an essentially free group action $G \actson (Y,\eta)$ has the property that $(\cR_0 \wr_I \Gamma)^t$ is isomorphic with $\cR(G \actson Y)$ if and only if the following holds. There exists a measure preserving action $\Lambda \actson (Z_0,\mu_0)$ on a standard finite measure space containing $X_0 \subset Z_0$ as a Borel subset of measure $1$ such that the restriction of $\cR(\Lambda \actson Z)$ to $X_0$ equals $\cR_0$. There exists a finite, central subgroup $\Lambda_2 \subset \Lambda$ such that $X_0$ is a fundamental domain for $\Lambda_2 \actson Z_0$. There exists a globally $\Gamma$-invariant finite index subgroup $\Sigma_2 \subset \Lambda_2^{(I)}$ such that, writing $G_0 = \Lambda \wr_I \Gamma$ and $Z = \prod_{I}' (Z_0,X_0)$, the action $G \actson Y$ is conjugate to an induction of the action $G_0/\Sigma_2 \actson Z/\Sigma_2$. Also, $t = [\Lambda_2^{(I)}:\Sigma_2] \, [G:G_1]$.

Note that in this more precise description, we indeed have that $\cR_0 = \cR(\Lambda/\Lambda_2 \actson X_0)$ and that $G_0/\Sigma_2$ admits the finite normal subgroup $\Sigma$ given by $\Sigma = \Lambda_2^{(I)}/\Sigma_2$. Also, $G_0/\Lambda_2^{(I)} \cong (\Lambda/\Lambda_2)^{(I)}$ and $Z/\Sigma \cong X_0^I$, naturally.

It is then easy to see that the finite normal subgroup $\Sigma$ can be ``essential'': the following example provides a group action $G \actson Y$ that itself is not induced from a wreath product action, but with $G/\Sigma \actson Y/\Sigma$ being a wreath product action. Let $\Lambda = \Z/4\Z$ with subgroup $\Lambda_2 = 2 \Z / 4 \Z$. Put $Z_0 = \Lambda$ and $X_0 = \{0,1\}$. Renormalize the counting measure $\mu_0$ so that $\mu_0(X_0) = 1$. Define $\Sigma_2 = \{a \in \Lambda_2^{(I)} \mid \sum_{i \in I} a_i = 0\}$. Define $G = (\Lambda^{(I)}/\Sigma_2) \rtimes \Gamma$, with the two element group $\Sigma = \Lambda_2^{(I)}/\Sigma_2$ as normal subgroup. Let $Z = \prod_I' (Z_0,X_0)$. Then, $G \actson Z/\Sigma_2$ is a weakly mixing action and its quotient by $\Sigma$ is the wreath product action of $(\Lambda/\Lambda_2) \wr_I \Gamma$ on $X_0^I$.
\end{remark}


\begin{thebibliography}{CCMT12}\setlength{\itemsep}{-1mm} \setlength{\parsep}{0mm} \small
\bibitem[Ada88]{Ada88} S. Adams, Trees and amenable equivalence relations. {\it Ergodic Theory Dynam. Systems} {\bf 10} (1990), 1-14.

\bibitem[BFS10]{BFS10} U. Bader, A. Furman and R. Sauer, Integrable measure equivalence and rigidity of hyperbolic lattices. {\it Invent. Math.} {\bf 194} (2013), 313-379.

\bibitem[BHV08]{BHV08} B. Bekka, P. de la Harpe and A. Valette, Kazhdan's property (T). {\it New Mathematical Monographs} {\bf 11}, Cambridge University Press, Cambridge, 2008.

\bibitem[Bou07]{Bou07} N. Bourbaki, Alg\`{e}bre, chapitre 9. Springer-Verlag, Berlin, 2007.

\bibitem[BDV17]{BDV17} A. Brothier, T. Deprez and S. Vaes, Rigidity for von Neumann algebras given by locally compact groups and their crossed products. {\it Comm. Math. Phys.} {\bf 361} (2018), 81-125.

\bibitem[BT71]{BT71} F. Bruhat and J. Tits, Groupes r\'{e}ductifs sur un corps local. {\it Inst. Hautes \'{E}tudes Sci. Publ. Math.} {\bf 41} (1972), 5-251.

\bibitem[Cas78]{Cas78} J.W.S. Cassels, Rational quadratic forms. {\it London Mathematical Society Monographs} {\bf 13}, Academic Press, London-New York, 1978.

\bibitem[Clo02]{Clo02} L. Clozel, D\'{e}monstration de la conjecture $\tau$. {\it Invent. Math.} {\bf 151} (2003), 297-328.

\bibitem[CGMTD21]{CGMTD21} C.T. Conley, D. Gaboriau, A.S. Marks and R. Tucker-Drob, One-ended spanning subforests and treeability of groups. {\it Preprint.} \href{https://arxiv.org/abs/2104.07431}{arXiv:2104.07431}

\bibitem[Dia93]{Dia93} B. Diarra, $p$-adic Clifford algebras. {\it Ann. Math. Blaise Pascal} {\bf 1} (1994), 85-103.

\bibitem[DIP19]{DIP19} D. Drimbe, A. Ioana and J. Peterson, Cocycle superrigidity for profinite actions of irreducible lattices. {\it Groups Geom. Dyn.}, to appear. \href{https://arxiv.org/abs/1910.08642}{arXiv:1910.08642}

\bibitem[FM75]{FM75} J. Feldman and C.C. Moore, Ergodic equivalence relations, cohomology, and von Neumann algebras. I and II. {\it Trans. Amer. Math. Soc.} {\bf 234} (1977), 289-324, 325-359.

\bibitem[For74]{For74} P. Forrest, On the virtual groups defined by ergodic actions of $\R^n$ and $\Z^n$. {\it Advances in Math.} {\bf 14} (1974), 271-308.

\bibitem[Fur98]{Fur98} A. Furman, Orbit equivalence rigidity. {\it Ann. of Math.} {\bf 150} (1999), 1083-1108.


\bibitem[Fur09]{Fur09} A. Furman, A survey of measured group theory. In {\it Geometry, rigidity, and group actions}, Chicago Lectures in Math., Univ. Chicago Press, Chicago, 2011, pp.\ 296-374.

\bibitem[GITD16]{GITD16} D. Gaboriau, A. Ioana and R. Tucker-Drob, Cocycle superrigidity for translation actions of product groups. {\it Amer. J. Math.} {\bf 141} (2019), 1347-1374.


\bibitem[Hjo05]{Hjo05} G. Hjorth, A lemma for cost attained. {\it Ann. Pure Appl. Logic} {\bf 143} (2006), 87-102.

\bibitem[HM00]{HM00} K.H. Hofmann and S.A. Morris, Transitive actions of compact groups and topological dimension. {\it J. Algebra} {\bf 234} (2000), 454-479.

\bibitem[HV12]{HV12} C. Houdayer and S. Vaes, Type III factors with unique Cartan decomposition. {\it J. Math. Pures Appl.} {\bf 100} (2013), 564-590.

\bibitem[Ioa08]{Ioa08} A. Ioana, Cocycle superrigidity for profinite actions of property (T) groups. {\it Duke Math. J.} {\bf 157} (2011), 337-367.

\bibitem[Ioa12]{Ioa12} A. Ioana, Cartan subalgebras of amalgamated free product II$_1$ factors. With an appendix by A. Ioana and S. Vaes. {\it Ann. Sci. \'{E}c. Norm. Sup\'{e}r.} {\bf 48} (2015), 71-130.

\bibitem[Ioa14]{Ioa14} A. Ioana, Strong ergodicity, property (T), and orbit equivalence rigidity for translation actions. {\it J. Reine Angew. Math.} {\bf 733} (2017), 203-250.

\bibitem[Ioa18]{Ioa18} A. Ioana, Rigidity for von Neumann algebras. In {\it Proceedings of the International Congress of Mathematicians 2018.} World Sci. Publ., Hackensack, 2018, pp.\ 1639-1672.

\bibitem[Iso19]{Iso19} Y. Isono, Unitary conjugacy for type III subfactors and W$^*$-superrigidity. {\it J. Eur. Math. Soc. (JEMS)}, to appear. \href{https://arxiv.org/abs/1902.01049}{arXiv:1902.01049}

\bibitem[KV15]{KV15} A. Krogager and S. Vaes, A class of II$_1$ factors with exactly two group measure space decompositions. {\it J. Math. Pures Appl.} {\bf 108} (2017), 88-110.

\bibitem[KPV13]{KPV13} D. Kyed, H.D. Petersen and S. Vaes, $L^2$-Betti numbers of locally compact groups and their cross section equivalence relations. {\it Trans. Amer. Math. Soc.} {\bf 367} (2015), 4917-4956.

\bibitem[Mil17]{Mil17} J.S. Milne, Algebraic groups. {\it Cambridge Studies in Advanced Mathematics} {\bf 170}, Cambridge University Press, Cambridge, 2017.

\bibitem[MS02]{MS02} N. Monod and Y. Shalom, Orbit equivalence rigidity and bounded cohomology. {\it Ann. of Math.} {\bf 164} (2006), 825-878.

\bibitem[PV13]{PV13} H.D. Petersen and A. Valette, $L^2$-Betti numbers and Plancherel measure. {\it J. Funct. Anal.} {\bf 266} (2014), 3156-3169.

\bibitem[PR94]{PR94} V. Platonov and A. Rapinchuk, Algebraic groups and number theory. {\it Pure and Applied Mathematics} {\bf 139}. Academic Press, Boston, 1994.

\bibitem[Pop05a]{Pop05a} S. Popa, Some rigidity results for non-commutative Bernoulli shifts. {\it J. Funct. Anal.} {\bf 230} (2006), 273-328.

\bibitem[Pop05b]{Pop05b} S. Popa, Cocycle and orbit equivalence superrigidity for malleable actions of $w$-rigid groups. {\it Invent. Math.} {\bf 170} (2007), 243-295.

\bibitem[Pop06a]{Pop06a} S. Popa, On the superrigidity of malleable actions with spectral gap. {\it J. Amer. Math. Soc.} {\bf 21} (2008), 981-1000.

\bibitem[Pop06b]{Pop06b} S. Popa, Deformation and rigidity for group actions and von Neumann algebras. In {\it International Congress of Mathematicians.} Eur. Math. Soc., Z\"{u}rich, 2007, pp.\ 445-477.

\bibitem[PV06]{PV06} S. Popa and S. Vaes, Strong rigidity of generalized Bernoulli actions and computations of their symmetry groups. {\it Adv. Math.} {\bf 217} (2008), 833-872.

\bibitem[PV08]{PV08} S. Popa and S. Vaes, Cocycle and orbit superrigidity for lattices in $\SL(n,\R)$ acting on homogeneous spaces. In {\it Geometry, rigidity, and group actions}, Chicago Lectures in Math., Univ. Chicago Press, Chicago, 2011, pp.\ 419-451.

\bibitem[PV11]{PV11} S. Popa and S. Vaes, Unique Cartan decomposition for II$_1$ factors arising from arbitrary actions of free groups. {\it Acta Math.} {\bf 212} (2014), 141-198.

\bibitem[PV12]{PV12} S. Popa and S. Vaes, Unique Cartan decomposition for II$_1$ factors arising from arbitrary actions of hyperbolic groups. {\it J. Reine Angew. Math.} {\bf 694} (2014), 215-239.

\bibitem[Ser80]{Ser80} J.-P. Serre, Trees. Springer-Verlag, Berlin-New York, 1980.

\bibitem[Vae13]{Vae13} S. Vaes, Normalizers inside amalgamated free product von Neumann algebras. {\it Publ. Res. Inst. Math. Sci.} {\bf 50} (2014), 695-721.

\bibitem[VV14]{VV14} S. Vaes and P. Verraedt, Classification of type III Bernoulli crossed products. {\it Adv. Math.} {\bf 281} (2015), 296-332.

\bibitem[Zim84]{Zim84} R.J. Zimmer, Ergodic theory and semisimple groups. {\it Monographs in Mathematics} {\bf 81}, Birkh\"{a}user Verlag, Basel, 1984.
\end{thebibliography}
\end{document}